\documentclass[reqno, a4paper, 11pt, oneside]{article}
\usepackage[11pt]{extsizes}
\usepackage[a4paper,hmargin=2.5cm,vmargin=2.5cm, bindingoffset=0cm]{geometry}

\usepackage{mathtools}
\usepackage[utf8]{inputenc}
\usepackage[T2A]{fontenc}
\usepackage{amsmath,amssymb,amsthm}
\usepackage[english]{babel}
\usepackage{textcomp}
\usepackage[usenames]{color}
\usepackage{colortbl}
\usepackage{float}
\usepackage{relsize}
\usepackage{amsthm}

\usepackage{thmtools}
\usepackage{thm-restate}
\makeatletter
\newcommand*{\rom}[1]{\expandafter\@slowromancap\romannumeral #1@}
\makeatother

\usepackage[pagebackref,bookmarksopen]{hyperref}
\hypersetup{
	colorlinks=true,
	linkcolor=blue,
	citecolor=magenta,
	urlcolor=cyan,
}
\usepackage{cleveref}

\usepackage{graphicx}
\DeclareGraphicsExtensions{.pdf,.png,.jpg}

\usepackage{amsthm}

\theoremstyle{definition}

\newtheorem*{theorem*}{Theorem}
\newtheorem{theorem}{Theorem}[section]

\newtheorem{lemma}[theorem]{Lemma}
\newtheorem{cor}[theorem]{Corollary}
\newtheorem{remark}[theorem]{Remark}
\newtheorem{sign}[theorem]{Notation}
\newtheorem{con}[theorem]{Conjecture}
\newtheorem{Que}[theorem]{Question}
\newtheorem{proposition}[theorem]{Proposition}
\newtheorem{ex}[theorem]{Example}

\newtheorem{definition}[theorem]{Definition}

\newtheorem*{theorem-non}{Theorem}

\definecolor{applegreen}{rgb}{0.0, 0.42, 0.24}
\definecolor{applegreen}{rgb}{0.55, 0.71, 0.0}

\newcommand{\grad}{\text{grad}}

\newcommand{\tl}{\widetilde}

\newcommand{\Cp}{\mathcal{C}}

\newcommand{\Z}{\mathbb{Z}}
\newcommand{\R}{\mathbb{R}}

\renewcommand{\C}{\mathbb{C}}
\newcommand{\Cc}{\mathbb{C} \setminus 0}

\def\CC{({\mathbb C}\setminus 0)}

\newcommand{\D}{\Delta}

\newcommand{\Pd}{\mathcal P}
\newcommand{\Ps}{\mathbf P}

\newcommand{\As}{\mathbf{A}}

\newcommand{\Ss}{\mathcal {S}}

\newcommand{\Ds}{\mathbf{\Delta}}

\DeclareMathOperator{\IInt}{int}
\DeclareMathOperator{\conv}{Conv}

\newcommand{\define}{\textbf}

\usepackage{abstract}
\DeclareMathOperator{\rk}{rk}
\DeclareMathOperator{\st}{St}
\DeclareMathOperator{\Vol}{Vol}
\DeclareMathOperator{\MV}{MV}

\DeclareMathOperator{\Ehr}{Ehr}

\title{Newton numbers, vanishing polytopes and algebraic degrees}

\author{
Fedor Selyanin \thanks{\textit{Krichever Center for Advanced Studies}, Skolkovo Institute of Science and Technology, Moscow\\
\textit{Department of Mathematics}, National Research University ``Higher School of Economics'', Moscow\\
\textit{International Laboratory of Cluster Geometry}, National Research University ``Higher School of Economics'', Moscow \\
\textit{Email}: Fedor.Selyanin@skoltech.ru}
}
\date{}

\begin{document}
	
\maketitle
%		\begin{center}
%		{\Large \sc
 %               Newton numbers, vanishing polytopes and algebraic degrees
%                }
%	\vspace{3ex}
%		{\sc{Fedor Selyanin}}
%	\end{center}
\begin{abstract}
        
Consider a polynomial $f$ with a convenient Newton polytope $P$ and generic complex coefficients. By the global version of the Kouchnirenko formula, the hypersurface \mbox{$\{f = 0\} \subset \C^n$} has the homotopy type of a bouquet of $(n-1)$-spheres, and the number of spheres is given by a certain alternating sum of volumes, called the Newton number $\nu(P)$. Using the Furukawa-Ito classification of dual defective sets, %(i.e., sets whose GKZ-discriminant equals $1$),
we classify convenient Newton polytopes with vanishing Newton numbers as certain Cayley sums called $B_k$-polytopes. These $B_k$-polytopes generalize the $B_1$- and $B_2$-facets appearing in the local monodromy conjecture in the Newton non-degenerate case. Our classification provides a partial solution to Arnold's monotonicity problem.

The local $h^*$-polynomial (or $\ell^*$-polynomial) is a natural invariant of lattice polytopes that refines the $h^*$-polynomial coming from Ehrhart theory. We obtain decomposition formulas for the Newton number, for instance, prove the inequality $\nu(P) \ge \ell^*(P;1)$. The $B_k$-polytopes are non-trivial examples of thin polytopes. %(a.k.a. polytopes with vanishing local $h^*$-polynomial).

We generalize the Newton number in two independent ways: the $\ell$-Newton number and the $e$-Newton number. The $\ell$-Newton number comes from Ehrhart theory, namely, from certain generalizations of Katz-Stapledon decomposition formulas, and its properties are central to our proof that the $B_k$-polytopes are thin. The $e$-Newton number is the number of points of zero-dimensional critical complete intersections. Vanishing of the $e$-Newton number characterizes dual defective sets. Furthermore, the $e$-Newton number calculates algebraic degrees (such as Maximum Likelihood, Euclidean Distance and Polar degrees). For instance, we show that all known formulas for these algebraic degrees in the Newton non-degenerate case are implied by basic properties of the $e$-Newton number.

\end{abstract}

\setcounter{tocdepth}{2}
	\tableofcontents
	
\section{Introduction}
\graphicspath{{intro/}}

%In the paper we consider dual defective sets, thin polytopes and convex negligible polytopes. We prove that convex negligible polytopes and the convex hulls of dual defective sets are thin. We also prove that convex negligible polytopes are roughly the same as the convex hulls of the dual defective sets.

%\textcolor{red}{
%\begin{enumerate}
%        \item Монотонность $\nu^\ell$ -- может потом
%        \item Подбробнее прописать про обобщение Каца-Степлдона (может быть) -- ровно в 1 месте \cite[Lemma 7]{KS16} Lemma \ref{ks_decomp_h*_ag}
%        \item de Jonqui\`eres transformations -- потом
%        \item Монотонность всех degree -- потом
%        \item SageMath code -- потом
%       \item Cayley trick -- search and find a first place where it was
%\end{enumerate}
%}

Recall that the \textbf{support} of a polynomial $f\in \C[z_1,\dots, z_n]$ is the set of the exponents of its nonzero monomials and the \textbf{Newton polytope} is the convex hull of the support. A convex polytope $P\subset \R^n_{\ge 0}$ is called \textbf{convenient} (cf. Definition \ref{conv_boun_def}) if it contains the origin and a segment on each coordinate axis. 

\begin{definition} (cf. Definition \ref{newton_def})\label{Newton_def}
    Consider a convenient polytope $P \subset \R^n_{\ge 0}$. The \textbf{Newton number} is the alternating sum 
    $$\nu(P) = \sum_{F \subset \R^n} (-1)^{n-\dim F} \cdot \Vol_\Z(P \cap F), \footnote{\text{We use the \textbf{lattice volume} $\Vol_\Z$ (see Notation \ref{lattice_vol_sign}). Here it equals $(\dim F)!$ times the usual volume in $F$.}}$$
    over all the coordinate subspaces of $\R^n$. A convenient polytope in $\R^n_{\ge 0}$ is called \textbf{negligible} if its Newton number vanishes.
\end{definition} 

Consider a generic polynomial $f \in \C[z_1,\dots, z_n]$ with Newton polytope $P$. By the global version of the Kouchnirenko theorem (see \cite{Kou76} and \cite[Proposition 3.3]{Bro88}) the hypersurface $\{f=0\} \subset \C^n$ has the homotopy type of a bouquet of $\nu(P)$ of $(n-1)$-spheres, and $\nu(P)$ also equals the number of critical points of $f$ in $\C^n$. Using the Furukawa-Ito classification of dual defective sets (\cite{FI21}; see Theorem \ref{dual_defect_classification_th}), we classify all negligible polytopes in $\R^n_{\ge 0}$. We give the answer in terms of Cayley sums (we recall them in \S \ref{Mixed_cayley_rec_subsec}) which we call \textbf{$\mathbf{B_k}$-polytopes}.

\begin{definition}\label{B_k_def}
Consider lattice polytopes $P_0, P_1, \dots, P_k \subset \R^{n-k}$ such that $$\dim (P_1 + \dots + P_k) < k \quad \text{and} \quad \dim P_0 = n-k,$$ where ``$+$'' is the Minkowski sum. Then the Cayley sum $P_0 \ast P_1 \ast \cdots \ast  P_k \subset  \R^{n-k} \oplus \R^k_{\ge 0}$ is called \textbf{$\mathbf{B_k}$-polytope} based on $\R^k_{\ge 0}$.
\end{definition}

The notion of $B_k$-polytopes refers to the so-called $B_1$- and $B_2$-facets from \cite{LV11} and \cite{ELT22} arising in the context of the monodromy conjecture (see \cite{Ve} for a general survey). That is, $B_1$- and $B_2$-polytopes are natural projections of the $B_1$- and $B_2$-facets (for instance, this projection is illustrated in \cite [Figure 2 on p.17]{ELT22}). Note that the $B_k$-polytopes differ from the $B$-polytopes \cite[Definition 1.4]{ELT22}, for more details see \S \ref{comp_B_Bk_subsubsec}.

The definitions of a convenient polytope and its Newton number can be extended to the case of the cones $\R^m_{\ge 0} \oplus \R^{n-m}$ without significant changes (see Definitions \ref{conv_boun_def} and \ref{newton_def}).

\begin{theorem} [$B_k$-Theorem, see \S \ref{B_k_sec}]\label{neg_class_th}
    A convenient polytope $P$ in the cone $\R^m_{\ge 0} \oplus \R^{n-m}$ is negligible (i.e. $\nu(P) = 0$) if and only if it is a $B_k$-polytope based on a coordinate subspace $\R^k_{\ge 0} \subset \R^m_{\ge 0}$.
\end{theorem}

%\begin{con}[$B_k$-Conjecture]\label{B_k_conj}
%    Similar classification holds for any cone of the form $\R^m_{\ge 0} \oplus \R^{n-m}$, i.e. Lemma \ref{B_k-neg_lem} provides all examples of strongly convenient negligible polytopes in these cones.
%\end{con}

The latter theorem for all the cones $\R^m_{\ge 0} \oplus \R^{n-m}$ is needed for Arnold's monotonicity problem. Arnold's monotonicity problem suggests to study the (strict) monotonicity conditions of the Newton number (see \cite[(1982-16)]{Arn} for its initial formulation and e.g. \cite{Sel24} for a more up to date version). In \cite[Theorem 3.3.(1)]{Sel24} Arnold's monotonicity problem in dimension $n+1$ is reduced to the classification of (not necessarily convex) convenient lattice polytopes with vanishing Newton number in the cones $\R^m_{\ge 0} \oplus \R^{n-m}$. Thus, the $B_k$-theorem for all the cones $\R^m_{\ge 0} \oplus \R^{n-m}$ represents a significant advance in Arnold's monotonicity problem. A wrong answer to Arnold's monotonicity problem was conjectured in the end of the introduction of \cite{BKW19} and claimed in \cite[Theorem 2.25 and Remark 2.22]{L-AMS20}. For the simplest counterexample (a Morse function in dimension $4$), see \cite[Example 1.1]{Sel24}. Also note that using \cite[Theorem 3.3.(1)]{Sel24} one can obtain a counterexample from every $B_k$-polytope, which is not a $B_1$-polytope.

In what follows, we define two types of Newton numbers: the $\mathbf{\ell}$\textbf{-Newton} number and the $\mathbf {e}$\textbf{-Newton} number. The $\ell$-Newton number arises from generalizations of Katz-Stapledon formulas (from Ehrhart theory). The $\ell$-Newton number is the main tool for proving that the $B_k$-polytopes are thin. We also believe that the $\ell$-Newton number is of independent combinatorial interest and that it admits a natural geometric interpretation. 

The $e$-Newton number counts critical points in the complex torus of generic polynomials with given support. Its vanishing provides a criterion for dual defective sets. It counts algebraic degrees in the Newton non-degenerate case. All the known formulas for the algebraic degrees in the Newton non-degenerate case (maximum likelihood degree \cite{Kh78}, \cite{CHKS06}, \cite{LNRW23}, Euclidean distance degree \cite{BSW22}, \cite{TT24}, polar degree \cite{Huh13}) follow from basic properties of the $e$-Newton number.

Unlike the usual Newton number (see \cite[Example 2.5 on p.353]{GKZ94}) our Newton numbers are non-negative for arbitrary cones and convenient polytopes in it. 
Both $\ell$- and $e$-Newton numbers coincide with the usual Newton number in the case of convenient polytopes in $\R^n_{\ge 0}$. Consider a convenient polytope in $P \subset \R^n_{\ge 0}$ and a generic polynomial $f$ with Newton polytope $P$. Then the number of critical points of $f$ in $\C^n$ provides the calculation (similar to \cite[Section 4]{Oka79} in the local case) of the $e$-Newton number, and counting the middle homology groups of $\{f = 0\}$ corresponds to the $\ell$-Newton number since the $\ell$-Newton number is consistent with the Stapledon's formula for monodromy from \cite{Stap17} (see Example \ref{Stap_mon_ex}). 

\begin{remark}
    In all reasonable problems we know where the usual Newton number appears, it is natural to interpret it (depending on the context) as either the $\ell$- or $e$-Newton number.
\end{remark}

We believe that the following question has a positive answer.

\begin{Que}
    Does the $\ell$-Newton number of convenient polytopes in arbitrary cones have a natural geometric interpretation?
\end{Que}

The $e$-Newton number and the $\ell$-Newton number are generally different. In \S \ref{Comparison_sec} we compare them with the usual Newton number. We believe that the $\ell$-Newton number is greater than or equal to the $e$-Newton number (see Conjecture \ref{nu_ell>nu_e_con}).

\subsection{$\mathbf{\ell}$-Newton number and thin polytopes}

Ehrhart theory generalizes Pick's theorem by studying the number of lattice points in integer dilations of lattice polytopes. The $h^*$-polynomial of lattice polytopes is a convenient way to encode this data. The local $h^*$-polynomial (or $\ell^*$-polynomial) is a natural refinement of the $h^*$-polynomial. Katz-Stapledon decomposition formulas provide decompositions of the $h^*$- and the local $h^*$-polynomials for polyhedral subdivision. These decomposition formulas also incorporate the $h$-polynomials of Eulerian posets and local $h$-polynomials (or $\ell$-polynomials) of strong formal subdivisions.We recall all of these concepts in \S \ref {ks_decomp_rec_sec}.

Recall that ``$F\le C$'' means ``$F$ is a face of the cone $C$'' (see Example \ref{face_pos_ex}).

\begin{definition}\label{conv_boun_def}
A convex polytope $P \subset C$ is called \textbf{convenient} (in the cone $C$) if for every face $F \le C$ we have $\dim (F \cap P) = \dim F$. 
    %It is called \textbf{weakly convenient} if the equation $\dim (F\cap P) = \dim (F), F \le C$ implies $\dim (F'\cap P) = \dim (F')$ for all $F\le F'\le C$. 
    The \textbf{boundary} $\partial_C(P)$ is the boundary of $P$ in the topology of $C$. 
\end{definition}

Inspired by Katz-Stapledon decomposition formulas, we define the local $h^*$-polynomial $\ell^*_C(P;t)$ of convenient polytopes $P$ in a cone $C$. We define the $\ell$-Newton number as the evaluation $\nu_C^\ell (P) = \frac{1}{2} \ell^*_C(P;1)$. We show that the $\ell$-Newton number is a non-negative integer and that it coincides with the usual Newton number if $C$ is a simplicial cone (for example, if $C = \R^n_{\ge 0}$). We generalize Katz-Stapledon's decomposition formulas to this case, which imply the following inequality.

\begin{cor}(cf. Theorem \ref{loc_h*_ag_th})\label{nu>=ell_cor}
    For any convex convenient polytope $P\subset \R^n_{\ge 0}$ we have the following inequality:
    $$\nu(P) \ge \ell^*(P;1)$$
    Thus, for a $B_k$-polytope $P$, Theorem \ref{neg_class_th} gives $\nu(P)=0$, which implies $\ell^*(P;1)=0$. Therefore, \textbf{$\mathbf{B_k}$-polytopes are thin}.
\end{cor}

The $B_k$-polytopes provide non-trivial examples of thin polytopes (i.e., polytopes with vanishing $\ell^*$-polynomial). Thin simplices were defined in \cite{GKZ94}, the study of non-simplicial thin polytopes began in \cite{BKN24}. In \S \ref{B_k_sec} we recall what is known about thin polytopes and show that the $B_k$-polytopes are important in this context.

This also motivates the following definition and a generalization of Arnold's monotonicity problem (solved for the cones $\R^m_{\ge 0} \oplus \R^{n-m}$ in the $B_k$-Theorem \ref{neg_class_th}).

\begin{definition}\label{nu_def}
    A convenient polytope $P$ is called \textbf{negligible} in a cone $C$ if $\nu_C^\ell (P) = 0$.
\end{definition}

\begin{Que}\label{arb_ar_pr}
    Can we classify negligible polytopes in arbitrary cones?
\end{Que}

%\textcolor{red}{Хорошо бы примеры у Huh разобрать}

We now briefly describe the \textbf{agglutination technique} from \S\ref{boundary_sec} and \S\ref{gen_of_polyn_sec}, which is used to define the $\ell_\Cp^*$-polynomial and the $\ell$-Newton number. It was initially developed to prove Corollary \ref{nu>=ell_cor} but we believe that it is of independent interest. Remarkably, we do not know a proof of Corollary \ref{nu>=ell_cor} avoiding this rather complicated technique. Most surprisingly, our proof of Corollary \ref{nu>=ell_cor} crucially relies on the unimodality of the local $h$-polynomial of regular polyhedral subdivisions (see \cite[Theorem 6.1]{KS16}).

\begin{Que}
    Is there a simpler proof of Corollary \ref{nu>=ell_cor}? 
    Is there a simpler way to verify that $B_k$-polytopes are thin?
\end{Que}

The strong formal subdivisions introduced in \cite{KS16} abstract the concept of polyhedral subdivisions. We define combinatorial objects called \textbf{strong formal subdivisions with boundaries} and their \textbf{agglutination along the boundary}. These strong formal subdivisions with boundaries, their boundaries and their agglutinations are also ordinary strong formal subdivisions. Polyhedral subdivisions of convenient polytopes in cones provide the main examples of strong formal subdivision with boundaries. The local $h$-polynomial of the agglutination of these polyhedral subdivisions have non-negative coefficients.

%The local $h$-polynomial of the agglutination of polyhedral subdivisions of strongly convenient polytope along its boundary in cone has nonnegative coefficients and the proof sufficiently relies on the unimodulty of the local $h$-polynomial of regular subdivisions and on the symmetricity of the local $h$-polynomial of arbitrary strong formal subdivision.

%We compute the local $h$-polynomial of the agglutination through the local $h$ -polynomial of the initial strong formal subdivision and the local $h$-polynomial of the boundary. 

We generalize the definitions of the $h^*$- and the local $h^*$-polynomials to the case of arbitrary strong formal subdivisions whose elements in the source are lattice polytopes (we call them \textbf{polyhedral strong formal subdivisions}). Katz-Stapledon decomposition formulas are trivially extended to this case.

Our main example of polyhedral strong formal subdivisions is the \textbf{agglutination} of a convenient polytope $P$ in a cone $C$ with itself along the corresponding boundary $\partial_C(P)$ of $P$ in the topology of the cone $C$. We denote the corresponding $h^*$- and local $h^*$-polynomials by $h^*_\Cp(\Pd\sharp_{\Pd_B}\sharp \Pd;t)$ and $\ell^*_\Cp(\Pd\sharp_{\Pd_B} \sharp\Pd;t)$, respectively. The $h^*_\Cp$- and $\ell_\Cp^*$-polynomials of agglutinated polytopes have nonnegative coefficients, mirroring the property of the standard $h^*$- and $\ell^*$-polynomials. The $\ell$-Newton number is the evaluation $\frac{1}{2} \ell^*_\Cp(\Pd\sharp_{\Pd_B}\sharp \Pd;1)$ which is a non-negative integer for every cone $C$ and any convenient polytope $P$ in it.

%We define the $\mathbf{\ell}$\textbf{-Newton number} $\nu^\ell_C(P)$ of strongly convenient polytope $P$ in cone $C$ through the local $h^*$-polynomial of the agglutination. Unlike the usual Newton number (see Definition \ref{newton_def}), the $\ell$-Newton number is nonnegative integer for arbitrary cone. In the case of simplicial cones, the $\ell$-Newton number coincides with the original one.  We also prove that $\nu^\ell_C (P) \ge \ell^*(P;1)$ which implies Corollary \ref{nu>=ell_cor} and motivates the following problem.

%\begin{theorem}\label{nu>=ell_th}
%    Suppose that $P$ is a strongly convenient polytope in cone $C$. Then we have:
%    $$\nu^\ell_C (P) \ge \ell^*(P;1)$$
%\end{theorem}
%This theorem implies Corollary \ref{nu>=ell_cor} and motivates the following problem.

\subsection{$\mathbf{e}$-Newton number and Algebraic degrees}

Consider a finite lattice set $\Ps \subset \Z^n$ and a point $X$. If $X$ is in $\Z^n$ then the $\mathbf{e}$\textbf{-Newton number} $\nu^e_X(\Ps)$ is the number of critical points of a generic polynomial with support $\Ps - X$ (i.e. $\Ps$ shifted by the vector $-X$) in the torus $(\Cc)^n$. It can be generalized for an arbitrary point $X\in\mathbb Q\mathbb P^n$ as the number of points of a \textbf{zero-dimensional critical complete intersection} introduced in \cite[Section 2.7]{Es18}. We recall this together with the combinatorial formulas in \S \ref{0-cci_subsec}.

All results concerning the $e$-Newton numbers discussed in the paper boil down to \cite[Lemma 2.50]{Es18} together with the expression for Euler obstructions of toric varieties \cite{MT11} (see \cite[Example 1.7(6)]{Es24b} and \cite[Theorem 1.8]{Es24b} for more details). Many important special cases of this formula for the $e$-Newton number were rediscovered independently by other authors (e.g., see \cite[Theorem A.; Theorem 3.8]{FS24}, where the theorems can be rewritten in terms of the $e$-Newton numbers and this rewriting immediately simplifies and generalizes the results, see Remark \ref{Sot_rem}). We believe that our work will help to popularize the theory of (zero-dimensional) critical complete intersections.

Formulas for the algebraic degrees in the Newton non-degenerate case (maximum likelihood degree \cite{CHKS06}, \cite{LNRW23}, Euclidean distance degree \cite{BSW22}, \cite{TT24}, polar degree \cite{Huh13}) boil down to some basic properties of the $e$-Newton numbers. We show this in detail in \S \ref{algebraic_subsec} and refine the results concerning the algebraic degrees. The $e$-Newton number also computes the support function of the Newton polytope of the $\As$-discriminant from \cite{GKZ94}, see \S \ref{eGKZ_subsec}. Comparing these with the $\ell$-Newton number, we, for instance, obtain the following unexpected corollary.

\begin{cor}(see Corollary \ref{pdeg<ell_cor}.)
The polar degree of a homogeneous Newton non-degenerate polynomial is bounded by the local $h^*$-polynomial of the Cayley sum of its Newton polytope with the unit simplex.
\end{cor}

We also discuss the important cases of polar degree equal to zero and one (corresponding to vanishing hessian and homaloidal hypersurfaces), see Remark \ref{pol_01_rem}, Question \ref{huh_que} and Conjecture \ref{hom_conj}. Similar questions are interesting for the maximum likelihood degree (see Question \ref{ml=vol01_que}).

Recall that finite set $\As \subset \Z^n$ containing the origin $O$ is dual defective if and only if $\nu^e_O(\As) = 0$ (see Lemma \ref{dual=no_crit_points_lem}). This fact is crucial for our proof of Theorem \ref{neg_class_th} since it lets us apply the Furukawa-Ito classification (Theorem \ref{dual_defect_classification_th}) of dual defective sets, for more details see \S \ref{B_k_sec}.

\subsection*{Acknowledgment}

I am grateful to Alexander Esterov for supervising me and for the suggestion to compare my results with the formulas for the algebraic degrees.

\section{Preliminaries: Katz-Stapledon decomposition formulas} \label{ks_decomp_rec_sec}

In this section we recall (lower) Eulerian posets that abstract face posets of polytopes and their $h$- and $g$-polynomials. We recall $h^*$- and local $h^*$-polynomials of polytopes defined in \cite{Stan92} coming from Ehrhart theory. We also recall the local $h$-polynomials of strong formal subdivisions and Katz-Stapledon decomposition formulas from \cite{KS16}.

For an introduction to posets (partially ordered sets), $h$- and $g$-polynomials, see \cite{Stan}. For surveys on the local $h$-polynomial of triangulations of simplices see \cite{Cha} and \cite{Ath}. For a survey on the local $h$-polynomials of polyhedral subdivisions and Katz-Stapledon decomposition formulas see \cite[\S 2]{BKN24}. For an introduction to the Ehrhart theory including $h^*$-polynomial see \cite{BR07}. For a survey on local $h^*$-polynomials see \cite[\S 2]{BKN24}.

\subsection{Eulerian posets, $\mathbf{h}$- and $\mathbf{g}$-polynomial}

Eulerian posets were defined by Stanley in \cite{Stan81}. For instance, they abstract face posets of polytopes (see Example \ref{face_pos_ex}).

\begin{definition}
    The dual of a finite poset $\mathcal{P}$ is denoted $\mathcal{P}^\ast$. A finite poset $\mathcal{P}$ is \textbf{locally graded} if every inclusion-maximal chain in every interval $[x,y]$ has the same length $r(x,y)$. The \textbf{rank} $\rk(\mathcal{P})$ is the length of the longest chain in $\mathcal{P}$. If in addition there exists a rank function $\rho: \mathcal{P} \rightarrow \Z$, i.e., $r(x,y) = \rho(y) - \rho(x)$ for every interval $[x,y]$, then $\mathcal{P}$ is called \textbf{ranked}. If $\mathcal{P}$ is ranked and every interval $[x,y]$ with $x \neq y$ has the same number of even rank and odd rank elements, then $\mathcal{P}$ is \textbf{locally Eulerian}. If $\mathcal{P}$ is locally Eulerian and contains a minimal element $\hat{0}$, then it is called \textbf{lower Eulerian}. If it also contains a maximum $\hat{1}$, then $\mathcal{P}$ is called \textbf{Eulerian}. In the presence of a minimum $\hat{0}$ in a ranked poset $\mathcal{P}$, we will always assume that the rank function satisfies $\rho(\hat{0}) = 0$.
\end{definition}

\begin{remark}
    We use the notations of $g$- and $h$-polynomials from \cite{KS16}. Unfortunately, these notations are a bit different in different literature. For more details, see \cite[Remark 2.6]{BKN24}.
\end{remark}

\begin{definition}\label{d:g}
Let $B$ be an Eulerian poset of rank $n$. If $n = 0$, 
then $g(B;t) = 1$. If $n > 0$, then $g(B;t)\in\Z[t]$ is the unique polynomial in $t$ 
of degree strictly less than $n/2$ satisfying
\[
t^{n}g(B;t^{-1}) = \sum_{x \in B} g([\hat{0}_B,x];t) (t - 1)^{n - \rho(\hat{0}_B,x)}. 
\]
\end{definition}

\begin{remark}
The constant term of $g(B;t)$ is $g(B;0) = 1$. The linear coefficient of $g(B;t)$ is $\#\{x \in B | \rho(\hat 0_B,x) = 1\} - n$.
\end{remark}

\begin{remark} \cite[Remark 4.2]{BN08} \label{g=1_rem}
    One can verify that $g(B;t) = 1$ if and only if $B$ is a Boolean algebra.
\end{remark}

\begin{definition}\label{d:hpoly}\cite[Example 7.2]{Stan92}
Let $B$ be a lower Eulerian poset of rank $n$. Then the \textbf{$h$-polynomial} of $B$ is defined by
\[
t^n h(B;t^{-1}) = \sum_{x \in B}  g([\hat{0}_B,x];t) (t - 1)^{n - \rho(\hat{0}_B,x)}.
\]
\end{definition}

\begin{remark}
Let $B$ be a lower Eulerian poset of rank $n$. Then the constant term of $h(B;t)$ is $h(B;0) = 1$, and the linear coefficient of $h(B;t)$ is $\#\{x \in B | \rho(\hat 0_B,x) = 1\} - n$.
\end{remark}

\begin{remark}(see e.g. \cite[Example 3.14]{KS16})\label{h_out_top_ut}
If $B$ is an Eulerian poset of rank $n$, then comparison of the recursive formulas in Definition~\ref{d:g} and Definition~\ref{d:hpoly} implies that $h(B;t) = g(B;t)$. 
 If, furthermore, $n > 0$, then $B \setminus \{ \hat{1} \}$ is a %pure, 
 lower Eulerian poset of rank $n - 1$, and 
\[
(1 - t) h(B \setminus \{ \hat{1} \};t) =  g(B;t) - t^ng(B;t^{-1}).
\]
In particular, $h(B \setminus \{ \hat{1} \};t)$ is a polynomial of degree $n - 1$ with symmetric coefficients.
\end{remark}

\begin{ex}\label{face_pos_ex}
The poset of faces of a polytope $P$ (including the empty face) is an Eulerian poset under inclusion, called the \textbf{face poset} $\Pd$ of $P$, with $\rho (Q) = \dim Q + 1$ for every face $Q$ of $P$. The poset of faces of a cone $C$ (excluding the empty face) is an Eulerian poset under inclusion, called the \textbf{face poset} $\Cp$ of $C$, with $\rho (F) = \dim F$ for every face $F$ of $C$. If the cone $C$ is sharp (i.e., pointed), then its face poset is isomorphic to the face poset of any of its compact cross-sectional polytopes.
\end{ex}

The $h$-polynomial of the face poset of rational polytopes computes the intersection cohomology groups of the corresponding toric varieties, see \cite[Theorem 3.1]{Stan87}. Together with the results concerning the non-rational case from \cite{Kar04} (recall that there are non-rational polytopes which are not combinatorially equivalent to any rational polytopes, see e.g. \cite{Zie08}) this implies the following theorem.

\begin{theorem}(see e.g. \cite[Theorem 2.8]{BKN24})
    Let $P$ be a polytope of dimension $n$. Then 
    $$h (\Pd;t) = \sum_{i=0}^n h_i t^i$$
    is a palindromic polynomial with positive integer coefficients that form a unimodal sequence, i.e., $1 =h_0 \le h_1 \le \dots \le h_{\lfloor d /2\rfloor}$. Equivalently, $g(\Pd;t)$ has non-negative coefficients.
\end{theorem}

\subsection{Local $\mathbf{h}$-polynomials of strong formal subdivisions} \label{loc_h_strong_formal_subsec}

Strong formal subdivisions abstract polyhedral subdivisions. In this subsection we recall Katz-Stapledon decomposition formulas for the local $h$-polynomial of strong formal subdivisions from \cite{KS16}. These definitions refine Stanley's definitions of formal subdivision and local $h$-polynomial from \cite[Definition 7.4]{Stan92} and \cite[Corollary 7.7]{Stan92}; see \cite[Remark 3.33]{KS16}. The local $h$-polynomials were defined by Stanley in \cite{Stan92} in order to prove the monotonicity of the $h$-vector under subdivisions of complexes.

%These local $h$-polynomials of strong formal subdivisions generalize the local $h$-polynomials of subdivisions of a simplex from \cite{Stan92}, where these local $h$-polynomials were defined in order to prove the monotonicity of the $h$-vector under (simplicial) subdivision of (simplicial) complexes.

A function $f: \Gamma \rightarrow B$ between finite posets is \define{order-preserving} if $y \le y' \in \Gamma$ implies $f(y) \le f(y') \in B$.
If $(\Gamma,r_\Gamma)$ and $(B,r_B)$ are ranked posets, then a function  $f: \Gamma \rightarrow B$ is \define{rank-increasing} if $r_\Gamma(y) \le r_B(f(y))$ for all $y \in \Gamma$. 
Throughout, we let $\sigma: \Gamma \rightarrow B$ be an order-preserving, rank-increasing function  between locally Eulerian posets with rank functions $\rho_\Gamma$ and $\rho_B$, respectively. 
\begin{definition}
Consider an order-preserving  function $\sigma: \Gamma \rightarrow B$ between locally Eulerian posets. Then for all $y \in \Gamma$ and $x \in B$ such that $\sigma(y) \le x$, we may 
consider the lower Eulerian posets $\Gamma_{\ge y} := \{ y' \in \Gamma \mid y \le y' \}$ and $(\Gamma_{\ge y})_x := \{ y' \in \Gamma \mid y \le y', \sigma(y') \le x \}$, and the locally Eulerian poset 
$\IInt((\Gamma_{\ge y})_x):=\sigma^{-1}(x) \cap (\Gamma_{\ge y})_x$.   If $\Gamma$ is lower Eulerian, then we write $\Gamma_x  := (\Gamma_{\ge \hat{0}_\Gamma})_x$ and 
 $\IInt(\Gamma_x)  := \IInt((\Gamma_{\ge \hat{0}_\Gamma})_x)$.
%$\IInt(\Gamma_x)=\sigma^{-1}(x)$.   
%For a lower Eulerian poset, we will use $\Gamma_x$ to denote $(\Gamma_{\geq \hat{0}})_x$.
\end{definition}

\begin{definition}
Let $\sigma: \Gamma \rightarrow B$ be an order-preserving, rank-increasing function  between locally Eulerian posets with rank functions $\rho_\Gamma$ and $\rho_B$ respectively. 
Then $\sigma$ is \define{strongly surjective} if it is surjective and for all $y \in \Gamma$ and $x \in B$ such that $\sigma(y) \le x$, there exists $y \le y' \in \Gamma$ such that $\rho_\Gamma(y') = \rho_B(x)$ and $\sigma(y') = x$. 
\end{definition}

The following definition abstracts polyhedral subdivisions.

\begin{definition}\label{strong_formal_def}
    Let $\sigma: \Gamma \rightarrow B$ be an order-preserving, rank-increasing function  between locally Eulerian posets with rank functions $\rho_\Gamma$ and $\rho_B$ respectively. 
Then $\sigma$ is a \textbf{strong formal subdivision} if it is strongly surjective and  for all $y \in \Gamma$ and $x \in B$ such that $\sigma(y) \le x$,
$$
\sum_{ \substack{y \le y'\\ \sigma(y') = x}} (-1)^{\rho_B(x) - \rho_\Gamma(y')} = 1.
$$
If $\Gamma$ and $B$ are lower Eulerian, then it follows that $\sigma(\hat{0}_\Gamma) = \hat{0}_B$, and the \textbf{rank} $\rk(\sigma) \in \Z_{\ge 0}$ of $\sigma$ is defined as $\rk(\sigma) := \rho_B(\hat{0}_B) - \rho_\Gamma(\hat{0}_\Gamma)$. Alternatively, it follows from the fact that $\sigma$ is rank-increasing and strongly surjective that $\rk(\sigma) = \rk(\Gamma)-\rk(B)$.
\end{definition}

\begin{definition} %\cite[Definition 4.1]{KS16}\label{d:localpoly} 
Let $\sigma: \Gamma \rightarrow B$ be a strong formal subdivision  between a lower  Eulerian poset $\Gamma$ and an Eulerian poset $B$. % with rank functions $\rho_\Gamma$ and $\rho_B$ respectively.
Then the  \textbf{local $\mathbf {h}$-polynomial} $\ell_B(\Gamma;t) \in \Z[t]$ is defined by
\[
\ell_B(\Gamma;t) = \sum_{ x \in B } h(\Gamma_{x};t) (-1)^{\rho_B(x,\hat{1}_B)} g([x,\hat{1}_B]^*;t).
\]
We also introduce the following useful notation. 
Let $\sigma: \Gamma \rightarrow B$ be a strong formal subdivision  between locally Eulerian posets with rank functions $\rho_\Gamma$ and $\rho_B$ respectively, and 
fix $y \in \Gamma$ and  $x \in B$. %We define the associated  local $h$-polynomial $l_B(\Gamma,x,y;t) \in \Z[t]$ as follows:
If $\sigma(y) \le x$, then we may consider the restricted strong formal subdivision 
$\sigma: (\Gamma_{\ge y})_x \rightarrow [ \sigma(y),x]$. In this case we set %define %In this case, we define 
\[
\ell_B(\Gamma,x,y;t) := \ell_{[ \sigma(y),x]}((\Gamma_{\ge y})_x;t) = \sum_{\sigma(y) \le x' \le x} h((\Gamma_{\ge y})_{x'};t) (-1)^{\rho_B(x',x)} g([x',x]^*;t). 
\]
We set $\ell_B(\Gamma,x,y;t) = 0$ if $\sigma(y) \nleq x$. 
\end{definition}

\begin{sign}
We say that polynomial $f(t)$ has symmetric coefficients centered at $d \in \Z /2$ if $$f(t) = t^{2k}f(t^{-1})$$
\end{sign}

\begin{remark}\cite[Corollary 4.5]{KS16}\label{loc_h_sym_cor}
     Let $\sigma: \Gamma \to B$ be a strong formal subdivision between a lower Eulerian poset $\Gamma$ of rank $r$ and an Eulerian poset $B$. Then the local $h$-polynomial $\ell_B(\Gamma;t)$ has symmetric coefficients centered at $r/2$.
\end{remark}

\begin{remark}\label{r:reverseh}
Let $\sigma: \Gamma \rightarrow B$ be a strong formal subdivision  between a lower  Eulerian poset $\Gamma$ and an Eulerian poset $B$. % with rank functions $\rho_\Gamma$ and $\rho_B$ respectively.
We recover the $h$-polynomial of $\Gamma$ by the formula:
\[
h(\Gamma;t)= \sum_{ x \in B }  \ell_B (\Gamma, x,\hat 0_\Gamma;t) g([x,\hat{1}_B];t).
\]
\end{remark}

\begin{theorem}
    (Katz-Stapledon decomposition of local $h$-polynomial)\cite[Corollary 4.7]{KS16} \label{c:hrefine}
Let  $\tau: \Omega  \rightarrow  \Gamma$ and let $\sigma: \Gamma \rightarrow B$ be strong formal subdivisions lower Eulerian posets $\Omega,\Gamma$ and Eulerian poset $B$. Then for $z\in\Omega$, $x\in\Gamma$,
\[
\ell_B(\Omega;t) =  \sum_{ y \in \Gamma } \ell_B(\Gamma, \hat 1_B, y;t)  \ell_\Gamma(\Omega, y, \hat 0_\Omega;t)
\]
\end{theorem}

\begin{ex}\label{e:lconstant} \cite[Example 4.10]{KS16}
Let $\sigma: \Gamma \rightarrow B$ be a strong formal subdivision  between a lower  Eulerian poset $\Gamma$ of rank $r$ and an Eulerian poset $B$ of rank $n$. Then we have:

\begin{itemize}

\item The constant term of $\ell_B(\Gamma;t)$ is equal to the coefficient of $t^{r}$ in $\ell_B(\Gamma;t)$, which is equal to 
\[
\left\{\begin{array}{cl}  1 & \text{if } n = 0,  \\ 
0 & \text{if } n > 0.  \end{array}\right. 
\]

\item The linear coefficient of $\ell_B(\Gamma;t)$ is equal to the coefficient of $t^{r - 1}$ in $\ell_B(\Gamma;t)$, which is equal to 
\[
 \left\{\begin{array}{cl}   \beta- r  & \text{if } n=0,  \\ 
\beta  - 1  & \text{if } n=1, \\
\beta & \text{if }  n > 1.
  \end{array}\right. ,
\]
where 
\[
\beta = \# \{ y \in \Gamma \mid \sigma(y) = \hat{1}_B, \rho_\Gamma(\hat{0}_\Gamma,y) =  1 \}.
\]
\end{itemize}
    
\end{ex}

The \textbf{boolean} algebra on $r$ elements consists of all subsets of a set of cardinality $r$ and forms an Eulerian poset under inclusion of rank $r$. It is the face poset of an $(r-1)$-dimensional simplex. A poset $B$ with minimal element $\hat 0$ is \textbf{simplicial} if for every $x$ in $B$, the interval $[\hat 0,x]$ is a Boolean algebra.

\begin{lemma}\cite[Lemma 4.12]{KS16}\label{l:simplicialsub}
Let $\sigma: \Gamma \rightarrow B$ be a strong formal subdivision between 
a simplicial poset $\Gamma$  of rank $r$ and a Boolean algebra $B$ of rank $n$ with rank functions $\rho_\Gamma$ and $\rho_B$ respectively. Then 
\[
\ell_B(\Gamma;t)  = \sum_{x \in B} h(\Gamma_x;t) (-1)^{\rho_B(x,\hat{1}_B)}   = \sum_{ y \in \Gamma } (-1)^{r - \rho_\Gamma(\hat{0}_\Gamma,y)} t^{r- e(y)}(t - 1)^{e(y)},
\]
where $e(y) = \rho_B(\sigma(y)) - \rho_\Gamma(y)$  is the \textbf{excess} of $y$. 
\end{lemma}

\begin{remark}
The local $h$-polynomial of strong formal subdivision is not necessarily non-negative (see \cite[Example 5.6]{KS16} and Corollary \ref{non_pos_cor}).
\end{remark}

For a polyhedral subdivision $\Ss$ of a polytope $P$, we denote its local $h$-polynomial by $\ell_P(\Ss;t)$, where the underlying strong formal subdivision is between the face poset of $\Ss$ and the face poset of $P$. Recall that $\sigma(F), F \in \Ss$ is defined as the minimal face of $P$ containing $F$.
For any face $Q$ of $P$ containing $\sigma(F)$, we may consider the corresponding  local $h$-polynomial:
\[
\ell_P(\Ss,Q,F;t) := \ell_{[\emptyset, P]}(\Ss,Q,F;t) =  \ell_{[\sigma(F),Q]}(\st_{\Ss|_{Q}}(F);t).
\]

\begin{theorem}\cite[Theorem 6.1]{KS16}\label{l_poly_non-neg_uni}
     Let $\Ss$ be a polyhedral subdivision of a polytope P. Then for every cell $F$ of $\Ss$ and face $Q$ of $P$ containing $F$, the local $h$-polynomial $\ell_P(\Ss,Q,F;t)$ has non-negative symmetric coefficients. Moreover, if $\Ss$ is a regular subdivision, then the coefficients of $\ell_P(\Ss,Q,F;t)$ are unimodal.
\end{theorem}

\begin{remark}
    In \cite[Theorem 6.1]{KS16} the condition ``symmetric'' is written only for regular subdivisions. But due to \cite[Corollary 4.5]{KS16} the local $h$-polynomial has symmetric coefficients for arbitrary strong formal subdivision, therefore we include the condition ``symmetric'' for arbitrary polyhedral subdivision. We also removed the word ``rational'' because the results from \cite{Kar19} allows this (see e.g. \cite[Theorem 2.11]{BKN24}). The latter remark is non-trivial since there are non-rational convex polytopes which are not combinatorially equivalent to any rational polytopes; see e.g. \cite{Zie08}.
\end{remark}

The local $h$-polynomials are non-negative in the polyhedral case because they appear when applying the Beilinson-Bernstein-Deligne decomposition theorem from \cite{BBD81} to the toric morphisms (see \cite[Lemma 6.4]{KS16} and \cite{CMM18}; see also \cite[Notes and References, Part 2, Chapter 11, p. 514]{GKZ94}).

\subsection{$\mathbf{h^*}$- and local $\mathbf{h^*}$-polynomials of lattice polytopes}\label{h_loc_h_rec_sebsec}

In this subsection we recall Ehrhart theory including (local) $h^*$-polynomials and Katz-Stapledon decomposition formulas for (local) $h^*$-polynomial from \cite{KS16}. We recommend \cite{BR07} as an introduction to Ehrhart theory including $h^*$-polynomials and \cite[\S 2]{BKN24} for a survey on (local) $h^*$-polynomials of polyhedral subdivisions and Katz-Stapledon decomposition formulas.

 Let $P$ be a non-empty lattice polytope in a lattice $\Z^d$. 
 Consider the function $f_P(m)=\#(mP\cap \Z^d)$, for $m\in \Z_{> 0}$.  By Ehrhart's theorem (\cite{Ehr62} and e.g. \cite[Section~3.4]{BR07}), $f_P(m)$ is a polynomial of degree $\dim P$, called
 the \define{Ehrhart polynomial} of $P$, and $f_P(0) = 1$. It follows that we can write
\begin{equation*}\label{e:Ehrhartpoly}
f_P(m)=f_0(P)+f_1(P)m+\dots+f_{\dim P}(P)m^{\dim P}, 
\end{equation*}
where $f_i(P)\in \Z$, and
\[
\Ehr_P(t)\equiv 1 + \sum_{m > 0} f_P(m) t^m = \frac{h^*(P;t)}{(1 - t)^{\dim P + 1}}, 
\]
where  $h^*(P;t)$ is a polynomial of degree at most $\dim P$ called the \define{$\mathbf{h^*}$-polynomial} of $P$ (see, for example, \cite[Section~3.4]{BR07}). A non-trivial theorem of Stanley \cite[Theorem 2.1]{Stan80} states that the coefficients of the 
$h^*$-polynomial are non-negative integers. Throughout, we will use the notation\[
h^*(P; t) =  \sum_{i = 0}^{\dim P} h^*_i t^i. 
\]

\begin{remark} (see, e.g., \cite[Example 7.1]{KS16}) \label{h=vol_rem}
The $h^*$-polynomial of lattice polytope $P\subset \Z^n$ satisfy the following properties:
\begin{enumerate}
    \item $h^*_0 = 1$
    \item $h^*_1 = \# \{P\cap\Z^n\} - \dim P - 1$
    \item $h^*_{\dim P} = \# \{\text{int} (P)\cap\Z^n\}$
    \item $h^* (P;1) = \Vol_\Z(P)$, where $\Vol_\Z$ is the lattice volume (see Notation \ref{lattice_vol_sign}).
\end{enumerate}

\end{remark}

The following definition was introduced by Stanley in \cite[Example 7.13]{Stan92}, generalizing the definition of Betke and McMullen in the case of a simplex \cite{BeM85}. It was independently introduced by Borisov and Mavlyutov in \cite{BoM03} as $\tl S(t)$ in the context of the Batyrev's construction of mirror symmetry \cite{Bat94}.

\begin{definition}\cite[Example 7.13]{Stan92}\label{d:localstar}
The \textbf{local $\mathbf{h^*}$-polynomial} $\ell^*(P;t)$ of $P$ is 
\begin{equation*}%\label{e:localdef}
\ell^*(P; t) =  \sum_{Q \subseteq P} (-1)^{\dim P - \dim Q}h^*(Q;t)g([Q,P]^*;t). 
\end{equation*}
\end{definition}

\begin{lemma}\cite[see, e.g., Lemma 7.3(2)]{KS16}\label{l:localprop}
Let $P$ be a non-empty lattice polytope. Then the $h^*$-polynomial of a polytope can be recovered from the local $h^*$-polynomials of its faces via %the formula
\[
h^*(P;t) =  \sum_{Q \subseteq P} \ell^*(Q; t)g([Q,P];t).
\]
\end{lemma}

The $h^*$- and local $h^*$-polynomials can be considered as the explicit output of the algorithm for computing the $\chi_y$-characteristic and the Hodge-Deligne numbers from \cite{DK86}. For more details see the remarks below.

\begin{remark}
In \cite{DK86} the notations of $l$ and $l^*$ have nothing in common with the local $h$- and $h^*$-polynomials (which were defined later in \cite{Stan92}). In \cite{DK86} these notations just count lattice points in the polytope and in its interior.
\end{remark}

\begin{remark}\label{h*_characteristic_rem}(There is a typo in the formula from \cite[Theorem 5.5]{KS16a}.) Consider polynomial $f$ with Newton polytope $P$ and generic coefficients. Then the $h^*$-polynomial of $P$ computes the $\chi_y$-characteristic of $\{f=0\} \subset (\Cc)^n$ via:
	$$y \chi_y (\{f=0\}) = y E(\{f=0\};-y,1) = (-1) ^ {\dim P +1}(y+1)^{\dim P} + (-1)^{\dim P} h^*(P;-y).$$
\end{remark}

\begin{remark}\label{hodge_rem}
	Define the mixed $h^*$-polynomial via the formula (from \cite[Definition 7.5]{KS16}):
	$$h^*(P;u,v) = \sum\limits_{Q\subseteq P} v^{\dim Q + 1} \ell^*(Q,uv^{-1}) g([Q,P];uv)$$
	Consider polynomial $f$ with Newton polytope $P$ and generic coefficients. Then (see \cite[\S 5.4]{KS16a}) we have the following formula for the Hodge-Deligne Polynomial (or E-polynomial) of the hypersurface $\{f=0\}  \subset (\Cc)^n$:
	
	$$uv E(\{f=0\};u,v) = (uv-1)^{\dim P} + (-1)^{\dim P + 1} h^*(P;u,v).$$
	
	For more general formulas (including limit mixed Hodge structures) see \cite{KS16a}.
\end{remark}

\begin{remark}
	Similar combinatorics explicitly computes ``stringy invariant'' in the Batyrev's construction of mirror symmetry \cite{Bat94}. This invariant was computed by Batyrev and Borisov in \cite[Theorem 4.14]{BB96} and simplified by Borisov and Mavlyutov in \cite[Theorem 7.2]{BoM03}. For more details, see \cite[Example 5.13]{KS16a}.
\end{remark}

\begin{sign}\label{star_sign}
    Consider a lattice polyhedral subdivision $\Ss$ of a polytope $F$. The \textbf{star} $\st_\Ss(F)$ of $F$ in $\Ss$ is the subcomplex consisting of all cells $F^\prime$ of $\Ss$ that contain $F$ under inclusion.
\end{sign}

\begin{remark}
    In \cite{KS16} and \cite{BKN24} it is called the ``\textit{link}'' but the ``\textit{star}'' is a more common way to name it.
\end{remark}

\begin{lemma}(Katz-Stapledon decomposition of local $h^*$-polynomials)\cite[Lemma 7.12 (2-4)]{KS16}\label{l:pushstar} Let $\sigma:\Ss\rightarrow [\emptyset,P]$ be the strong formal subdivision induced by a lattice polyhedral subdivision $\Ss$ of a lattice polytope $P$.  Then
the following holds:
\begin{enumerate}
\item\label{i:push2}\[ h^*(P;t) = \sum_{\substack {F \in \Ss\\ \sigma(F) = P}} h^*(F;t) (t-1)^{\dim P - \dim F}
\]
\item\label{i:push3}  \[ h^*(P;t) = \sum_{F \in \Ss} \ell^*(F; t)  h(\st_\Ss(F);t),\]
\item\label{i:push4}  \[ \ell^*(P; t) = \sum_{F \in \Ss} \ell^*(F;t) \ell_P(\Ss, P,F;t).\]
\end{enumerate}
\end{lemma}

\begin{theorem}\cite[Theorem 7.17]{KS16} \label{loc_h*_non_neg_th}
    The coefficients of $h^*(P;t)$ and $\ell^*(P;t)$ are non-negative.
\end{theorem}

\begin{lemma}\cite[Lemma 7.3 (1)]{KS16}\label{loc_h*_sym_lem}
    The local $h^*$-polynomial $\ell^*(P)$ has symmetric coefficients centered at $(\dim (P)+1)/2$ 
\end{lemma}

\begin{theorem}\cite[Theorem 7.20]{KS16}
    Let $P \subset \R^n$ be a non-empty lattice polytope. Then the local $h^*$-polynomial $\ell^*(P;t) = \ell^*_1 t^1 + \dots + \ell^*_{\dim P}t^ {\dim P}$ satisfies $\ell^*_1 = \#(Int(P) \cap \Z^n) \le \ell^*_i \ \text{for}\  1 \le i \le \dim P$.
\end{theorem}

\begin{remark}
The local $h^*$-polynomial does not have unimodular coefficients in general (see e.g. \cite[Example 7.22]{KS16}). However it has unimodular coefficients for IDP polytopes \cite[Theorem 2.5]{APPS22} and $s$-lecture hall polytopes \cite[Corollary 5.5]{GS20}. 
\end{remark}

\section{Strong formal subdivision with boundary}\label{boundary_sec}

In this section we define \textbf{strong formal subdivisions with boundary}. These are strong formal subdivisions with a specified lower set of elements, called the \textbf{boundary}. The restriction of the subdivision to these boundary elements must itself be a strong formal subdivision.

We define the agglutination of the strong formal subdivisions with the same boundaries along their boundary. Lemma \ref{agglut_lemma} provides a simple formula for the local $h$-polynomial of the agglutination in terms of the local $h$-polynomials of the initial strong formal subdivisions and their boundary. The most important examples (Example \ref{boundary_pol_ex}) of strong formal subdivision with boundary are obtained from polyhedral subdivisions of convenient polytopes inside a cone. We prove that the local $h$-polynomial of such agglutinations inherits the properties of the local $h$-polynomial of usual polyhedral subdivisions, including nonnegativity and symmetricity.

\subsection{Poset with boundary and twins-poset}

In this subsection we define two types of posets and prove their basic properties. We need them to define strong formal subdivisions with boundaries. Recall that a \textbf{lower set} of a poset is a subposet which, together with each element, contains all smaller elements.

\begin{definition}\label{poset_w_boundary_def}
A pair $(\Gamma, \Gamma_B)$ of embedded posets is called \textbf{poset with boundary} if $\Gamma_B$ is a lower set of $\Gamma$ and $\rk (\Gamma_B) = \rk (\Gamma) - 1$. 
\end{definition}

\begin{definition}\label{agglutination_posets_def}
    Consider two posets $(\Gamma^1, \Gamma_B)$ and $(\Gamma^2, \Gamma_B)$ with common boundary $\Gamma_B = \Gamma^1 \cap \Gamma^2$. The \textbf{agglutination} $\Gamma^1 \sharp_{\Gamma_B} \sharp\Gamma^2$ is the poset $\Gamma^1\cup\Gamma^2$ with incomparable $\Gamma^1\setminus \Gamma_B$ and $\Gamma^2\setminus \Gamma_B$, and the induced poset structure on both $\Gamma^1$ and $\Gamma^2$.
\end{definition}

The agglutination of two (lower) Eulerian posets with a common boundary is itself (lower) Eulerian (this follows from verifying the defining property on intervals).

\begin{lemma}\label{agglutination_posets_lem}
    The $h$-polynomial of the agglutinations of a pair of lower Eulerian posets $\Gamma^1, \Gamma^2$ with the same boundary $\Gamma_B$ equals the following sum:
    $$h(\Gamma^1\sharp_{\Gamma_B} \sharp\Gamma^2; t) = h(\Gamma^1;t) + h(\Gamma^2;t) + (t-1) h(\Gamma_B;t)$$
\end{lemma}

\begin{proof}

    Applying the definition of the $h$-polynomial
    $$h(\Pd;t) = t^{\rk \Pd} \sum\limits_{x \in \Pd} (t^{-1}-1)^{\rk \Pd-\rho_\Pd (x) } g([\hat 0, x]; t^{-1}),$$
    we obtain:
    \begin{align*} h(\Gamma^1 \sharp_{\Gamma_B}\sharp \Gamma^2; t)& = (h(\Gamma^1;t) - t(t^{-1} - 1) h(\Gamma_B;t)) + (h(\Gamma^2;t) - t(t^{-1} - 1) h(\Gamma_B;t)) +\\ 
    &+t(t^{-1} - 1) h(\Gamma_B;t) = h(\Gamma^1;t) + h(\Gamma^2;t) + (t-1)  h(\Gamma_B;t)
\end{align*}
    
\end{proof}

\begin{definition}\label{twins_pos_def}
    Consider a poset $R$. Denote by $R^-$ its copy and the corresponding bijection by $\phi_R^+: R^- \to R$ and by $\phi_R^-: R\to R^-$ the inverse bijection. Denote by $R^\pm$ the poset consisting of elements $R \sqcup R^-$ with relations generated by:
    \begin{enumerate}
        \item Relations in $R$ and $R^-$,
        \item Additional relations $\forall r \in R : \phi_R^-(r) < r$.
    \end{enumerate}  We call $R^\pm$ the \textbf{twins-poset born from} $R$.
\end{definition}

\begin{figure}[H]
		\begin{center}
	\includegraphics[scale=1]{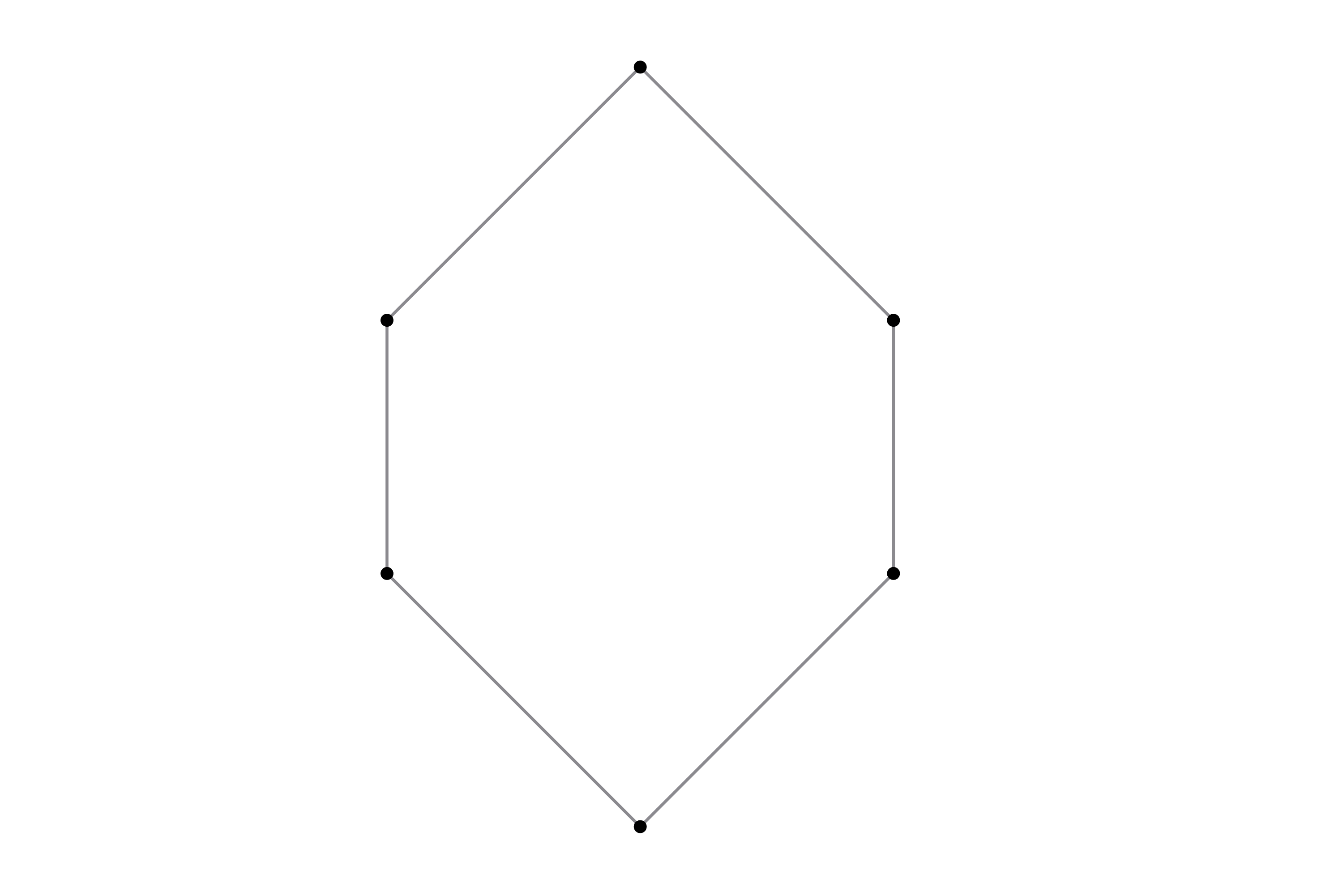}
	\includegraphics[scale=1]{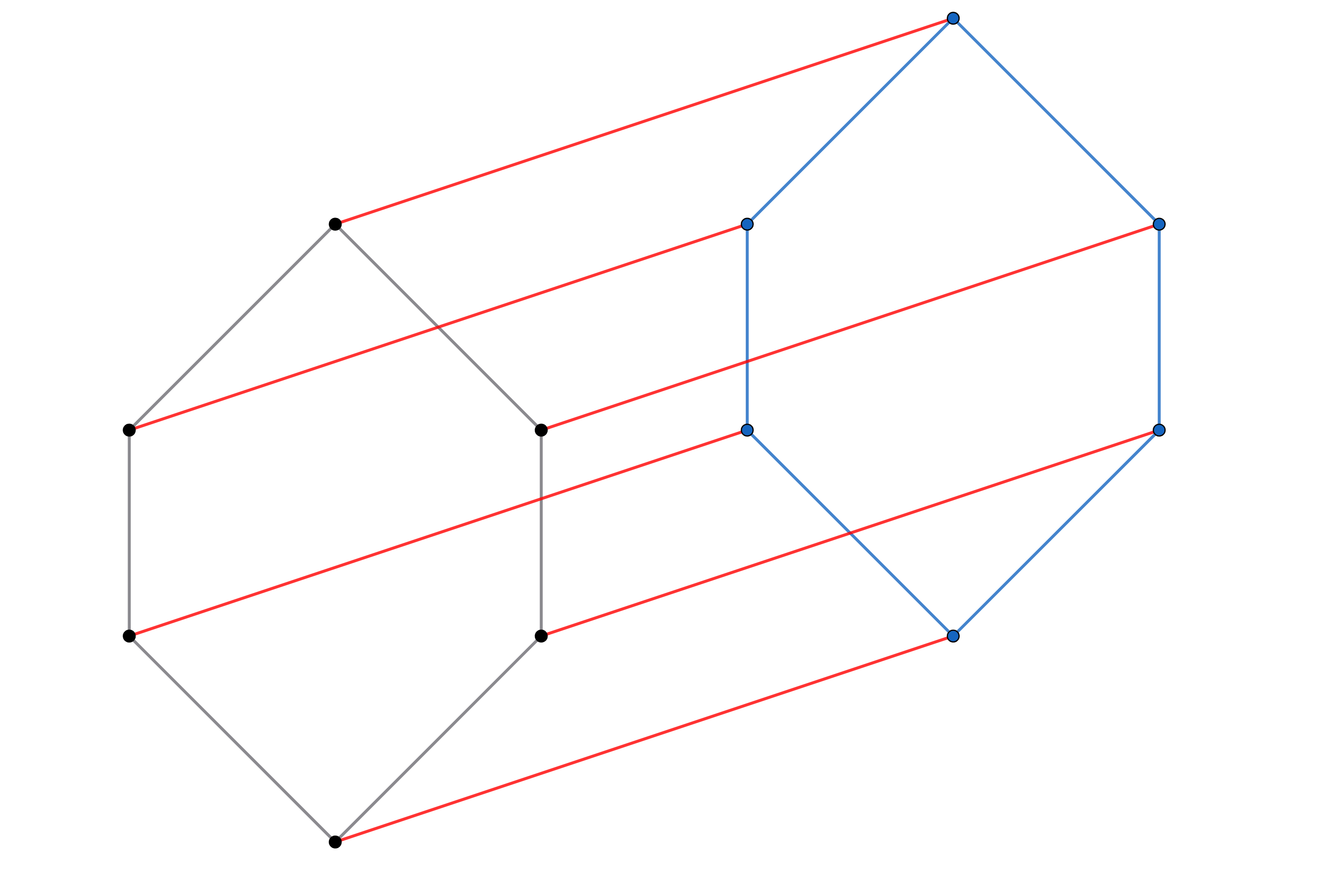}
	\end{center}
		\caption{
		\label{twin_poset_fig} Hasse diagram of a poset (left) and the twins-poset (right) born from it}
\end{figure}

Note that $R^\pm$ is (lower) Eulerian if and only if $R$ is (lower) Eulerian.

\begin{lemma}\label{twins_lem}
    For any lower Eulerian poset $R$, we have $h(R;t) = h(R^\pm;t)$.
\end{lemma}

\begin{proof}
First, we prove the lemma for Eulerian posets. In this case, we replace the $h$-polynomial by the $g$-polynomial. We proceed by induction on the rank $n$ of poset $R$. By the definition of the $g$-polynomial, we have: 

\begin{align*}
    t^{n+1}&g(R^\pm;t^{-1}) = \sum_{x \in R^\pm} g([\hat{0}_{R^\pm},x];t) (t - 1)^{n+1 - \rho(\hat{0}_{R^\pm},x)} = \\
    &= \sum_{x \in R^-} g([\hat{0}_{R^\pm},x];t) (t - 1)^{n+1 - \rho(\hat{0}_{R^\pm},x)} + \sum\limits_{x \in R\setminus \hat 1_R} g([\hat{0}_{R^\pm},x];t) (t - 1)^{n + 1 - \rho(\hat{0}_{R^\pm},x)} + g(R^\pm;t)
\end{align*}

By the induction hypothesis and the definitions of $g$-polynomial and $h$-polynomials, we have:

$$t^{n+1}g(R^\pm;t^{-1}) = (t-1) t^n g(R;t^{-1}) + (t-1) t^{n-1} h(R\setminus \hat 1_R;t^{-1}) + g(R^\pm,t)$$

By Remark \ref{h_out_top_ut}, we have: $$t^{n-1} h(R\setminus \hat 1_R;t^{-1}) = h(R\setminus \hat 1_R;t) = \frac{g (R;t) - t^n g(R;t^{-1})}{1-t},$$
therefore, we can rewrite the formula for the $g$-polynomial:

$$t^{n+1}g(R^\pm;t^{-1}) - t^{n+1} g(R;t^{-1}) =  g(R^\pm,t) - g(R;t)$$

Note that:
\begin{itemize}
\item All monomials on the left-hand side have degree greater than $(n+1)/2$
\item All monomials on the right-hand side have degree less than $(n+1)/2$
\end{itemize}

The left-hand side is a polynomial whose terms have degree greater than $(n+1)/2$, while the right-hand side is a polynomial whose terms have degree less than $(n+1)/2$. The only polynomial that can satisfy this identity is the zero polynomial. Therefore, both sides must be zero, which implies $g(R^\pm,t) = g(R;t)$..

Now we can verify the lemma for any lower Eulerian posets. Indeed, by the definition of the $h$-polynomial we have:

\begin{align*}
    t^{n+1} h (R^\pm;t^{-1}) &= \sum\limits_{x\in R^\pm} g([\hat 0_{R^\pm},x];t) (t-1)^{n+1 - \rho(\hat 0_{R^\pm},x)} = \\
    &= ((t-1)+1) (\sum\limits_{x\in R^-} g([\hat 0_{R^-},x];t) (t-1)^{n- \rho(\hat 0_{R^-},x)})  = t \cdot t^nh(R^-;t^{-1})
\end{align*}

\end{proof}

\subsection{Strong formal subdivisions with boundaries}

\begin{definition} \label{Strong_formal_subdiv_w_bo_def}
A map $\sigma: (\Gamma, \Gamma_B) \to R^\pm$ with  $ \sigma (\Gamma\setminus \Gamma_B) = R$ and $\sigma(\Gamma_B) = R^-$ from a lower Eulerian poset with boundary to an Eulerian twins-poset is called \textbf{strong formal subdivision with boundary} if both $\sigma$ and its restriction $\sigma_B$ on $\Gamma_B$ are strong formal subdivisions. 
\end{definition}

Our main examples strong formal subdivision with boundaries come from convenient polytopes in cones (see Example \ref{boundary_pol_ex}).

Note that the restriction of $\sigma$ to $\Gamma \setminus \Gamma_B$ is generally not a strong formal subdivision, since $\Gamma \setminus \Gamma_B$ may not be lower Eulerian (it can have multiple minimal elements).

%\begin{definition} \label{Strong_formal_subdiv_w_bo_def}
%Let $\Gamma$ be a poset with rank function $\rho_\Gamma$. Suppose that $\Gamma = \Gamma_I \sqcup \Gamma_B$ is the disjoint union of the inner part $\Gamma_I$ and the boundary part $\Gamma_B$ with restricted rank functions $\rho_I$ and $\rho_B$ respectively. Suppose that $\Gamma_B$ is a lower set, or equivalently, $\Gamma_I$ is an upper set. 

%Let $R$ be a poset with rank function $\rho_R$. Denote by $R^-$ the same poset with rank function $\rho_{R}-1$. Denote the corresponding bijection by $\phi^+:R \to R^+$ and its inverse by $\phi^-$. Denote by $R^{\pm}$ the poset consisting of elements $R \sqcup R^-$ with the corresponding rank functions and additional relations $\forall r \in R : \phi^-(r) < r$. 

%Consider function $\sigma: \Gamma \to R^{\pm}$. Suppose that its restrictions of on $\Gamma_I$ and $\Gamma_B$ are of the form $\sigma_I: \Gamma_I \to R$ and $\sigma_B: \Gamma_B \to R^-$. Suppose that the function $\sigma: (\Gamma = \Gamma_I \sqcup \Gamma_B) \to (R^{\pm} = R \sqcup R^-)$ is a strong formal subdivision. Then we call it a \textbf{strong formal subdivision with boundary}.
%\end{definition}

%Note that if $\sigma: \Gamma \to R^{\pm}$ is a strong formal subdivision with boundary, then its \textcolor{red} {restrictions $\sigma_I: \Gamma_I\to R$} and $\sigma_B : \Gamma_B \to R^-$ are also strong formal subdivisions. The most important property of strong formal subdivisions with boundary is that pair of them can be agglutinated into one strong formal subdivision. 

\begin{definition} \label{agglutination_def}
    Consider two strong formal subdivisions \mbox{$\sigma_1: (\Gamma^1, \Gamma_B) \to R^\pm$} and \mbox{$\sigma_2: (\Gamma^2,\Gamma_B) \to R^\pm$}  sharing the same boundary $\sigma_B: \Gamma_B \to R^-$. The \textbf{agglutination} $\sigma_1 \sharp_{\sigma_B}\sharp \sigma_2 : \Gamma^1 \sharp_{\Gamma_B}\sharp \Gamma^2  \to R$ of these strong formal subdivisions is defined as:
\begin{itemize}
\item $\phi^+_R \circ \sigma_B$ on $\Gamma_B$, 
\item $\sigma_1$ on $\Gamma^1\setminus \Gamma_B$, 
\item $\sigma_2$ on $\Gamma^2\setminus \Gamma_B$.
\end{itemize}

\end{definition}

A simple verification shows that $\sigma_1 \sharp_{\sigma_B}\sharp \sigma_2$ is indeed a strong formal subdivision.

\begin{lemma} \label{agglut_lemma}
    The local $h$-polynomial of the agglutination satisfies:
    $$\ell_R(\Gamma^1 \sharp_{\Gamma_B}\sharp \Gamma^2;t) = \ell_{R^\pm} (\Gamma^1;t)+ \ell_{R^\pm} (\Gamma^2;t) + (t+1)\ell_{R^-} (\Gamma_B;t)$$
\end{lemma}

\begin{proof}
    Recall the definition of the local $h$-polynomial of strong formal subdivisions:

%    $$t^{\rk \Pd} h(\Pd;t^{-1}) = \sum\limits_{x \in \Pd} (t-1)^{\rk \Pd-\rho_\Pd (x) } g([\hat 0, x]; t)$$

%     $$h(\Pd;t) = t^{\rk \Pd} \sum\limits_{x \in \Pd} (t^{-1}-1)^{\rk \Pd-\rho_\Pd (x) } g([\hat 0, x]; t^{-1})$$

    \begin{align*}
    \ell_R(\Gamma^1 \sharp_{\Gamma_B}\sharp \Gamma^2;t) &= \sum \limits_{x\in R} h((\Gamma^1 \sharp_{\Gamma_B}\sharp \Gamma^2)_x;t) (-1)^{\rho_{R}(x,\hat 1_{R})} g([x,\hat 1_{R}]^*;t)\\
    \ell_{R^\pm}(\Gamma^i;t) &= \sum \limits_{x\in R^\pm} h(\Gamma^i_x;t) (-1)^{\rho_{R^{\pm}}(x,\hat 1_{R^\pm})} g([x,\hat 1_{R^\pm}]^*;t)\\
    \ell_{R^-} (\Gamma_B;t) &= \sum \limits_{x\in R^-} h((\Gamma_B)_x;t) (-1)^{\rho_{R^-} (x,\hat 1_{R^-})}  g([x,\hat 1_{R^-}]^*;t)
    \end{align*}

%    Note that posets $R^{\pm}$ and $R^-$ are also boolean. This implies that $g$-polynomials in the above formula are equal to $1$.

 %   $$\ell_R(\Gamma_{ag};t) = \sum \limits_{x\in R} h((\Gamma_{ag})_x;t) (-1)^{\rho_{R}(x,\hat 1_{R})}$$

  %  $$\ell_{R^\pm}(\Gamma^i;t) = \sum \limits_{x\in R^\pm} h(\Gamma^i_x;t) (-1)^{\rho_{R^{\pm}}(x,\hat 1_{R^\pm})}$$

  %  $$\ell_{R^-} (\Gamma_B;t) = \sum \limits_{x\in R^-} h((\Gamma_B)_x;t) (-1)^{\rho_{R^-} (x,\hat 1_{R^-})}$$
    
%    The following inequality provides the lemma. For each $x \in R$ we have:

%    $$ h(\Gamma^i_{\phi^-(\hat 1_R)};t ) = h(\Gamma_B;t)$$

To prove the lemma, it suffices to show that for every $x \in R$ we have the equality:
$$h((\Gamma^1 \sharp_{\Gamma_B}\sharp \Gamma^2)_x;t) =  h(\Gamma^1_x;t) + h(\Gamma^2_x;t) - h(\Gamma^1_{\phi^-(x)};t) - h(\Gamma^2_{\phi^-(x)};t) + (t+1) h((\Gamma_B)_{\phi^-(x)};t)$$
Indeed, if this holds, then substituting into the definition of $\ell_R(\Gamma^1 \sharp_{\Gamma_B}\sharp \Gamma^2;t)$ and applying Lemma \ref{twins_lem} would yield the desired formula.

Without loss of generality, we may assume that $x = \hat 1_R$. Thus, the latter equation is equivalent to:
$$h(\Gamma^1 \sharp_{\Gamma_B}\sharp \Gamma^2;t) =  h(\Gamma^1;t) + h(\Gamma^2;t) - h(\Gamma_B;t) - h(\Gamma_B;t) + (t+1) h(\Gamma_B;t).$$

This equality follows directly from Lemma \ref{agglutination_posets_lem}.

\end{proof}

\subsection{Subdivisions of convenient polytopes in cones}

For a cone $C$, denote its Eulerian face poset by $\mathcal C$ as in Example \ref{face_pos_ex}.

\begin{ex}\label{boundary_pol_ex}
    Let $P \subset C$ be a convenient polytope and $\Ss$ be a polyhedral subdivision of $P$. Denote by $\Ss_B$ the lower set of $\Ss$ consisting of polytopes contained in the boundary $\partial_C P$ (see Definition \ref{conv_boun_def}). Denote by $\sigma_0: \Ss \to \mathcal C$ mapping $Q \in \Ss$ to the minimal face of $C$ containing $Q$. Then the map $\sigma: \Ss \to \mathcal C^\pm$ defined as $\sigma_0$ on $\Ss\setminus \Ss_B$ and as $\phi_{\mathcal C}^-\circ\sigma_0$ on $\Ss_B$ is a strong formal subdivision with boundary.
\end{ex}
\begin{figure}[H]
		\begin{center}
	\includegraphics[scale=4]{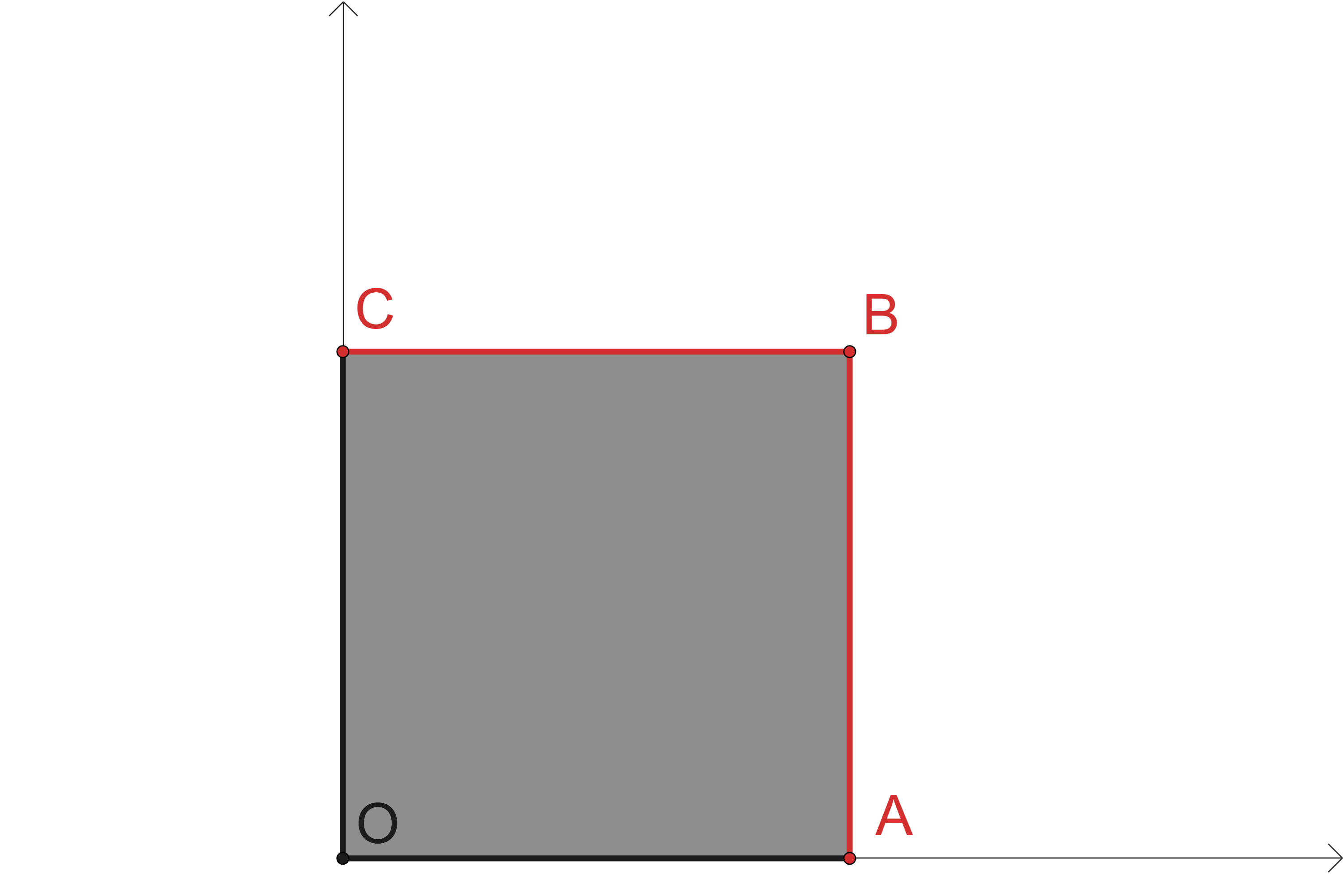}
	\end{center}
		\caption{
		\label{bound_fig1} Convenient square $P = OABC$ in the cone $C = \R^2_{\ge 0}$ with red boundary $\partial_C (P)$}
\end{figure}

\begin{figure}[H]
		\begin{center}
	\includegraphics[scale=1]{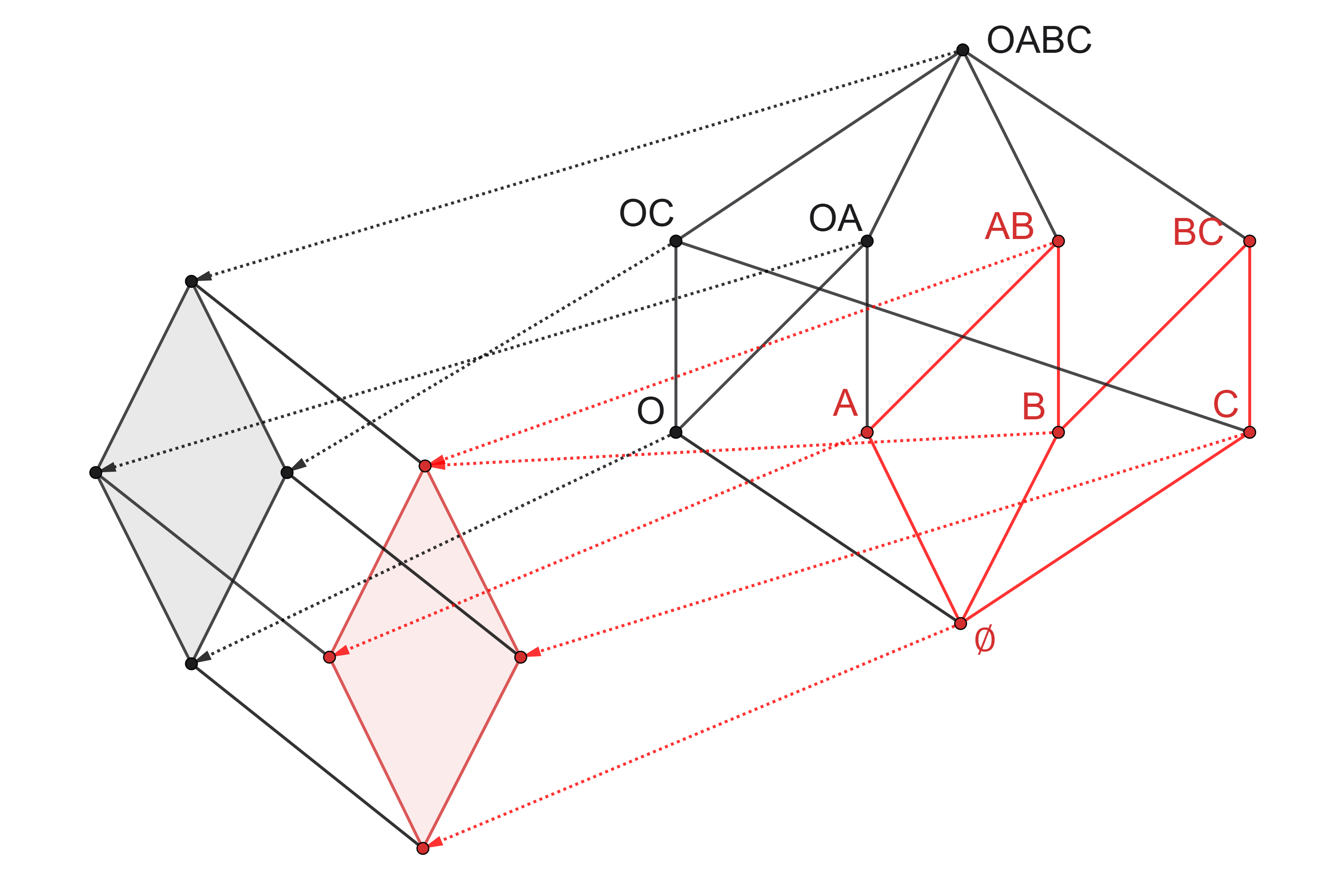}
	\end{center}
		\caption{
		\label{bound_fig2} Strong formal subdivision (dashed vectors) between the face poset $\Pd$ of the square $OABC$ (right) and $\mathcal C^\pm$ (left). Hasse diagrams of the boundary poset $\Pd_B \subset \Pd$, the lower set $\mathcal C^- \subset \mathcal C^\pm$ and the mapping between them are drawn red.}
\end{figure}

Note that the polytope's convenience implies that the map $\sigma$ is strongly surjective. The remaining conditions defining strong formal subdivisions can be confirmed the same way as for polyhedral subdivisions of polytopes \cite[Lemma 3.25]{KS16}.

\begin{remark}\label{boundary_reg_rem}
Consider the face poset $\Pd$ of a convenient polytope $P$ in a cone $C$. Note that $\Pd_B$ (the boundary poset from Example \ref{boundary_pol_ex} for $\Ss = \Pd$) can be obtained from regular polyhedral subdivision:
    
\begin{enumerate}
\item Let $C_{\min}$ be the minimal non-empty face of $C$ (i.e., its maximal affine subspace).
\item Take $X$ in relative interior of $(P \cap C_{\min})$.
\item Define the height function: $0$ on the vertices of $P$, and $1$ at $X$.
\item Denote $\mathcal{X}$ be the resulting regular subdivision of $P$.
\end{enumerate}

Then for any $F \in \Pd_B$ we have $\st_{\Pd_B} (F) = \st_{\mathcal X} (\conv (F,X))$. Thus, for any face $Q \in \mathcal C$ containing $F$, we have $\ell_{\mathcal C^-} (\Pd_B,Q,F;t) = \ell_\Pd(\mathcal X, Q, F;t)$.
\end{remark}

Consider a face $U\in \Pd_B$ of $P$ and face $Q \in \mathcal C$ such that $U \subset Q$. Denote $\dim (U) = u$ and $\dim (Q) = q$.

\begin{lemma}\label{loc_estim_lem}
    The following formula holds:
    $$\ell_{\mathcal C ^\pm} (\Pd,Q,U;t) =- \ell_{\mathcal C^-} (\Pd_B,Q,U;t) + R_{<(q-u)/2},$$
    where $R_{<(q-u)/2}$ is a polynomial of degree less than $(q-u)/2$ whose coefficients are determined by the coefficients of $\ell_{\mathcal C^-} (\Pd_B,Q,U;t)$. 
\end{lemma}

\begin{cor}\label{non_pos_cor}
    The polynomial $\ell_{\mathcal C ^\pm} (\Pd,Q,U;t)$ has symmetric coefficients centered at $(q-u)/2$ and polynomial $\ell_{\mathcal C^-}(\Pd_B,Q,U;t)$ has symmetric coefficients centered at $(q-u-1)/2$. Therefore, Lemma \ref{loc_estim_lem} fully recovers $\ell_{\mathcal C ^\pm} (\Pd,Q,U;t)$ from $\ell_{\mathcal C^-} (\Pd_B,Q,U;t)$.
    For instance, $\ell_{\mathcal C ^\pm} (\Pd,Q,U;t)$ has non-positive coefficients.
\end{cor}

\begin{proof}
    By definition of the local $h$-polynomial we have:
    $$\ell_{\mathcal C ^\pm}(\Pd,Q,U;t) = \sum\limits_{\sigma(U) \le x \le Q} h(\st_\Pd(U)|_x;t) (-1)^{\rho_{\mathcal C^{\pm}}(x,Q)} g([x,Q]^*;t)$$

    Let us separate the sum over $x\in \mathcal C$ and $x \in \mathcal C^-$. Applying Lemma \ref{twins_lem} and the fact that every face of polytope forms an Eulerian poset (and hence its $h$-polynomial equals its $g$-polynomial), we obtain:

    \begin{align*}
    \ell_{\mathcal C ^\pm}(\Pd,Q,U;t) = &\sum\limits_{\phi_{\mathcal C}^+(\sigma(U)) \le x \le Q} g(\st_\Pd(U)|_x;t) (-1)^{\rho_{\mathcal C}(x,Q)} g([x,Q]^*;t) -\\
    &-\sum\limits_{\sigma(U) \le x \le \phi_{\mathcal C}^-(Q)} h(\st_\Pd(U)|_x;t) (-1)^{\rho_{\mathcal C^{-}}(x,\phi_{\mathcal C}^-(Q))} g([x,\phi_{\mathcal C}^-(Q)]^*;t)
    \end{align*}

    Note that:
    \begin{enumerate}
        \item The first summand is the sum of products of pairs of $g$-polynomials of posets with rank sum $q-u+1$.
        \item For a poset of rank $r$, $\deg(g) \leq (r-1)/2$.
    \end{enumerate} 
    
    Therefore, the first summand has degree less than $(q-u)/2$ and can be rewritten as $R_{<(q-u)/2}$.

    The second summand is equal to $- \ell_{\mathcal C^-} (\Pd_B,Q,U;t)$ and this completes the proof.
    
\end{proof}

\begin{lemma}\label{l_agl_lem}
        We have the following formula:
        $$\ell_{\mathcal C} (\Pd\sharp_{\Pd_B} \sharp \Pd,Q,U;t) = (t-1)\ell_{\mathcal C^-} (\Pd_B,Q,U;t) + R_{<(q-u)/2},$$
where $\deg (R_{<(q-u)/2}) < (q-u)/2$ is determined by $\ell_{\mathcal C^-} (\Pd_B,Q,U;t)$.

\end{lemma}

\begin{proof}
By Lemma \ref{agglut_lemma} we have:
$$\ell_{\mathcal C} (\Pd\sharp_{\Pd_B} \sharp \Pd,Q,U;t) = 2 \ell_{\mathcal C^\pm} (\Pd,Q,U;t) + (t+1)\ell_{\mathcal C^-} (\Pd_B,Q,U;t)$$

Lemma \ref{loc_estim_lem} yields:
$$\ell_{\mathcal C} (\Pd\sharp_{\Pd_B} \sharp \Pd,Q,U;t) = (t-1)\ell_{\mathcal C^-} (\Pd_B,Q,U;t) + R_{<(q-u)/2}, $$
where $\deg (R_{<(q-u)/2}) < (q-u)/2$.

Recall that by Remark \ref{loc_h_sym_cor}, the local $h$-polynomial $\ell_{\mathcal C} (\Pd\sharp_{\Pd_B}\sharp \Pd,Q,U;t)$ has symmetric coefficients centered at $(q-u)/2$. Therefore, its coefficients are indeed determined by the lower-degree terms, which in turn are determined by $\ell_{\mathcal C^-} (\Pd_B,Q,U;t)$.

\end{proof}

By Theorem \ref{l_poly_non-neg_uni} and Remark \ref{boundary_reg_rem} the polynomial $\ell_{\mathcal C^-}(\Pd_B,Q,U;t)$ has symmetric unimodal coefficients centered at $(q-u-1)/2$. Hence, we obtained the following corollary.

\begin{cor}\label{non_neg_l_cor}
    The local $h$-polynomial $\ell_{\mathcal C}(\Pd\sharp_{\Pd_B}\sharp\Pd,Q,U;t)$:
    \begin{enumerate}
        \item Has non-negative coefficients.
        \item Vanishes if and only if $\ell_{\mathcal C^-}(\Pd_B,Q,U;t)$ vanishes.
        \item Its evaluation $\ell_{\mathcal C}(\Pd\sharp_{\Pd_B}\sharp\Pd,Q,U;1)$ equals the coefficient of $\ell_{\mathcal C^-}(\Pd_B,Q,U;t)$ at $t^{\lfloor(q-u-1)/2\rfloor}$.
    \end{enumerate}
\end{cor}

%consider the agglutinated poset $\Pd \sharp_{\Pd_B}\sharp\Pd$, where $\Pd_B$ consists of all faces of $\Pd$ which are contained in the boundary $\partial_C P$. Define the \textbf{agglutinated polytope $P\sharp_{\partial_C P}\sharp P$} as a formal polyhedral complex formed by the polytopes in the poset $\Pd \sharp_{\Pd_B}\sharp\Pd$. Let $\Ss$ be a polyhedral subdivision of $P\sharp_{\partial_C P}\sharp P$, i.e. the agglutination of two polyhedral subdivisions of $P$ (coinciding on the boundary) along the boundary. We denote by $\Ss_B$ the corresponding boundary poset.
    
Note that Katz-Stapledon decomposition formulas \ref{c:hrefine}, non-negativity of the local $h$-polynomial Theorem \ref{l_poly_non-neg_uni} and Corollary \ref{non_neg_l_cor} imply the following theorem.

\begin{theorem} \label{non_neg_l_th}
Let $P$ be a convenient polytope in a cone $C$. Consider two polyhedral subdivisions $\Ss_1$ and $\Ss_2$ of $P$ that coincide on the boundary $\partial_C P$. Let $\sigma_1$ and $\sigma_2$ be the corresponding strong formal subdivisions with boundaries. Denote by $\sigma:\Ss_1\sharp_{\Ss_B}\sharp\Ss_2\to \Cp$ their agglutination. Then the local $h$-polynomial $\ell_{\mathcal C} (\Ss_1\sharp_{\Ss_B}\sharp \Ss_2,Q, U;t)$ has symmetric non-negative coefficients centered at $(\dim Q- \dim U)/2$.     
\end{theorem}

\section{Generalization of $\mathbf{h^*}$- and $\mathbf{\ell^*}$-polynomial}\label{gen_of_polyn_sec}

\subsection{$\mathbf{h^*}$- and $\mathbf{\ell^*}$-polynomial of polyhedral strong formal subdivision}

An \textbf{abstract lattice polyhedral complex} $\Ss$ is a poset whose elements are ``abstract'' convex lattice polytopes, with the properties that: (1) every face of any polytope in $\Ss$ is also contained in $\Ss$, and (2) any two polytopes in $\Ss$ intersect along a common face. Formally, abstract lattice polyhedral complex is a map $\theta$ from a poset $\Gamma$ to lattice polytopes in (different) lattices such that for every $x \in \Gamma$ the lower set $\Gamma_{\le x}$ is the face poset of the polytope $\theta (x)$ and for every $y\le x$ there is a lattice isomorphism between the polytope $\theta (y)$ and the corresponding face of the polytope $\theta (x)$. We do not assume that there is a geometric realization in a common lattice. 

A \textbf{polyhedral strong formal subdivision} is a strong formal subdivision between an abstract lattice polyhedral complex and an Eulerian poset.
Inspired by Katz-Stapledon decomposition formulas (Lemma \ref{l:pushstar}) we define $h^*$- and local $h^*$-polynomials of polyhedral strong formal subdivisions. Note that since strong formal subdivisions are strongly surjective, any abstract lattice polyhedral complex $\Ss$ is pure, i.e. every its maximal polytope has the same dimension $\dim \Ss$.

\begin{definition}\label{h*_ag_def}
The $h^*$-polynomial of a polyhedral strong formal subdivision $\sigma: \mathcal{S} \to R$ is:
    $$h_R^*(\Ss;t) =  \sum\limits_{\substack{F \in \Ss \\ \sigma(F) = \hat 1_R}}  h^*(F;t) (t-1)^{\dim  \Ss - \dim F}$$
\end{definition}

\begin{definition}\label{loc_h*_def}
    The local $h^*$-polynomial of a polyhedral strong formal subdivision $\sigma: \Ss \to R$ is:
    $$\ell_R^*(\Ss;t) =  \sum\limits_{x \in R} (-1)^{\rk \hat 1_R - \rk x}  h_{[\hat 0_R, x]}^*(\Ss_x;t) g([x,\hat 1_R]^*;t)$$
\end{definition}

It follows e.g. from \cite[Theorem 3.11]{KS16} that we recover the $h^*$-polynomial via the local $h^*$-polynomial:

$$h_R^*(\Ss;t) = \sum\limits_{x \in R}  \ell_{[\hat 0_R, x]}^*(\Ss_x;t) g([x,\hat 1_R];t)$$

If $\Ss$ is a polyhedral subdivision of a polytope $P$ and $\sigma$ is the strong formal subdivision corresponding to this polyhedral subdivision, the definitions yield the standard $h^*(P;t)$ and $\ell^*(P;t)$. In the following subsection we give other examples of polyhedral strong formal subdivisions using agglutinations technique. 

Note that Katz-Stapledon decomposition formulas remain valid in this more general setting.

\begin{lemma}\label{ks_decomp_h*_ag} Let $\sigma:\Ss\rightarrow R$ be a polyhedral strong formal subdivision. Then the following holds:
\begin{enumerate}
\item\label{i:push2new}\[ h_R^*(\Ss;t) = \sum_{\substack {F \in \Ss\\ \sigma(F) = \hat 1 _R}} h^*(F;t) (t-1)^{\dim \Ss - \dim F}
\]
\item\label{i:push3new}  \[ h_R^*(\Ss;t) = \sum_{F \in \Ss} \ell^*(F; t)  h(\st_\Ss(F);t),\]
\item\label{i:push4new}  \[ \ell_R^*(\Ss; t) = \sum_{F \in \Ss} \ell^*(F;t) \ell_R(\Ss, C,F;t).\]
\end{enumerate}
\end{lemma}

\begin{proof}
    The first formula is implied by Definition \ref{h*_ag_def} and the corresponding Katz-Stapledon decomposition formula from Lemma \ref{l:pushstar}.
    The equivalence of these formulas is proved the same way as in \cite[Lemma 7.12]{KS16}.
\end{proof}

\subsection{$\mathbf{h^*}$- and $\ell^*$-polynomial of agglutination}

For a convenient lattice polytope $P$ in a cone $C$ with face posets $\Pd$ and $\Cp$, consider agglutinated strong formal subdivision $\sigma\sharp_{\sigma_B}\sharp\sigma: \Pd\sharp_{\Pd_B}\sharp\Pd \to \Cp$.

%For a face $Q$ of $C$ denote by $P_Q$ the intersection $P \cap Q$.

%\begin{definition}\label{loc_h*_def}
%    Define the local $h^*$-polynomial of strong formal subdivision $\sigma: \Ss \to R$ between abstract lattice polyhedral complex $\Ss$ and Eulerian poset $R$ as follows
 %   $$ \ell^*(\Ss;t) =  \sum\limits_{Q \in \mathcal C} (-1)^{\dim C - \dim Q}  h^*(P_Q\sharp_{\partial_Q P_Q} \sharp P_Q;t) g([Q,C]^*;t)$$
%\end{definition}

%$$h^*(P\sharp_{\partial_C P} \sharp P;t) = \sum\limits_{Q \in \mathcal C}  \ell^*(P_Q\sharp_{\partial_Q P_Q} \sharp P_Q;t) g([Q,C];t)$$

%\textcolor{red}{continue}

%\subsection{continue}

%Consider strongly convenient polytope $P$ in cone $C$. The agglutination $P\sharp_{\partial_CP} P$ is a pair of strongly convenient polytopes $P$ with identified boundary $\partial_C P$. We can define the local $h^*$-polynomial of the agglutination inspired by Katz-Stapledon decomposition formulas and assuming that two copies of $P$ form a polyhedral subdivision of $P\sharp_{\partial_CP} P$.

%Note that if we consider the trivial subdivision of the agglutinated polytope, we obtain:

%$$\ell^*(P\sharp_{\partial_CP} P;t) = \sum\limits_{{F \in \Pd\sharp_{\Pd_B}\Pd}}\ell^*(F;t) \cdot \ell_{\mathcal C}(\Pd\sharp_{\Pd_B}\Pd,C,F;t) = $$ $$ = 2 \ell^*(P;t) + \sum\limits_{{F \in \Pd_B}}\ell^*(F;t) \cdot \ell_{\mathcal C}(\Pd\sharp_{\Pd_B}\Pd,C,F;t)$$

\begin{definition}\label{l_nn_def}
The \textbf{$\ell$-Newton number} is defined as $\nu^\ell_C(P) = \frac{1}{2} \ell_\Cp^*(\Pd\sharp_{\Pd_B}\sharp \Pd;1)$.
\end{definition}

By Lemma \ref{agglut_lemma}, $\ell_{\mathcal{C}}(\mathcal{P}\sharp_{\Pd_B}\sharp\Pd,C,F;1)$ is even, so the $\ell$-Newton number is an integer. Recall the definition of the usual Newton number.

\begin{definition}(cf. \cite[p.351]{GKZ94})\label{newton_def} 
For a lattice polytope $P \subset C$ the \textbf{Newton number} is the following alternating sum of volumes:
    $$\nu_C(P) = \sum\limits_{F \le C} (-1)^{\dim(C) -\dim (F)} \Vol_{\Z} (P \cap F).$$
\end{definition}

Note that the Newton number can be negative even for convenient polytopes (see \cite[Example 2.5, p.353]{GKZ94}).

Corollary \ref{nu>=ell_cor} is implied by the following theorem.

\begin{theorem}\label{loc_h*_ag_th}
    For a convenient polytope $P \subset C$ with $\dim C = n$ we have:
    \begin{enumerate}
        \item The polynomial $\ell_\Cp^*(\Pd\sharp_{\Pd_B}\sharp \Pd;t)$ has nonnegative symmetric coefficients centered at $(n+1)/2$.
        \item The $\ell$-Newton number is a non-negative integer, moreover, satisfies $\nu^\ell_C(P)\ge \ell^*(P;1)$.
        \item If the face poset $\mathcal C$ is boolean, then $\nu^\ell_C(P) = \nu_C(P)$.
    \end{enumerate}
    
\end{theorem}

    \begin{proof}
    \begin{enumerate}
        \item The first item follows from Theorems \ref{non_neg_l_th}, \ref{loc_h*_non_neg_th} and Lemma \ref{loc_h*_sym_lem}; 
        \item The second item follows from Theorems \ref{non_neg_l_th}, \ref{loc_h*_non_neg_th} and Lemma \ref{agglut_lemma}.
        \item The third item follows from Remark \ref{g=1_rem}, Definition \ref{loc_h*_def}, and Remark \ref{h=vol_rem}
    \end{enumerate}
        
        \end{proof}

\begin{cor}
    The $h^*$-polynomial $h^*_\Cp (\Pd\sharp_{\Pd_B}\sharp \Pd;t)$ has non-negative coefficients.
\end{cor}

\begin{ex}[Compatibility with Stapledon's formulas for monodromy]\label{Stap_mon_ex}
    Consider a convenient polytope $P$ in $C = \R^n_{\ge 0}$ with face poset $\Pd$ and boundary poset $\Pd_B$ consisting of its faces in $\partial_C P$. Suppose that $T$ is the polyhedral subdivision of $P$ consisting of faces in $\Pd_B$ and their convex hulls with the origin. For $F \in \Pd_B$ denote by $F^0$ its convex hull with the origin. The global version of Stapledon's formulas for monodromy from \cite[Corollary 6.14]{Stap17} implies that:
    $$\nu_C(P) = \sum\limits_{F \in \Pd_B} \ell_{\Cp^-} (\Pd_B, \hat 1_{\Cp^-}, F;1) (\ell^*(F;1) + \ell^*(F^0;1))$$
    
    Let us show that this formula agrees with Lemma \ref{ks_decomp_h*_ag} (\ref{i:push4}) via:

    $$\ell_\Cp^*(T\sharp_{\Pd_B}\sharp T;t) = \sum\limits_{F \in \Pd_B} \ell_{\Cp} (T\sharp_{\Pd_B}\sharp T, \hat 1_{\Cp}, F;t) \ell^*(F;t) + 2 \sum\limits_{F \in \Pd_B} \ell_{\Cp} (T, \hat 1_{\Cp}, F^0;t) \ell^*(F^0;t)$$

    By the construction of $T$, we have $\ell_{\Cp} (T, \hat 1_{\Cp}, F^0;t) = \ell_{\Cp^-} (T, \hat 1_{\Cp^-}, F;t)$ (since the corresponding strong formal subdivisions coincide). By Lemma \ref{agglut_lemma} we have:

    $$\ell_{\Cp} (T\sharp_{\Pd_B}\sharp T, \hat 1_{\Cp}, F;t) = 2 \ell_{\Cp^\pm} (T, \hat 1_{\Cp^\pm}, F;t) + (t+1) \ell_{\Cp^-} (\Pd_B,  \hat 1_{\Cp^-}, F;t) $$

    Lemma \ref{twins_lem} implies that for any $F \in \Pd_B$, we have $\ell_{\Cp^\pm} (T, \hat 1_{\Cp^\pm}, F;t) = 0$. Thus, we obtain:

    $$\ell_\Cp^*(T\sharp_{\Pd_B}\sharp T;t) = \sum\limits_{F \in \Pd_B} \ell_{\Cp^-} (\Pd_B,  \hat 1_{\Cp^-}, F;t) ((t+1) \ell^*(F;1) +  2 \ell^*(F^0;t)) $$

    Recall that our cone $C$ has boolean face poset, so $\nu_C(P) = \frac{1}{2} \ell_\Cp^*(T\sharp_{\Pd_B}\sharp T;1)$ and this is consistent with Stapledon's formulas for monodromy. 
\end{ex}

\begin{remark}
The agglutination technique can be extended to coconvex polytopes and remains compatible with Stapledon's formulas for local monodromy \cite[Corollary 6.22]{Stap17}.
\end{remark}

Note that we do not know how to prove the inequality between the Newton number and the local $h^*$-polynomial (Corollary \ref{nu>=ell_cor}) using only Stapledon's formulas for monodromy (and avoiding the agglutination technique).

\section{Manifestations of the Newton numbers}

\subsection{Preliminaries: Mixed volume, Cayley sums, BKK-theorem}\label{Mixed_cayley_rec_subsec}

%Let us recall some definitions from convex geometry. For more details see e.g. \cite{Ew} or \cite{TOJ}. 

%We use the following definition of polytopes (for more details see e.g. \cite[\S 2.1]{Sel24}).

%\begin{definition}
%    \textit{Polyhedron} is a convex (generally non-bounded) polytope, i.e. the intersection of finite number of half-spaces.
%\end{definition}

%\begin{definition}
%    Convex polytope is  a bounded polyhedron, or the convex hull of a finite set if points. Consider finite set $\Ds$, we denote its convex hull by $\D = Conv(\Ds)$. A face of $\Ds$ is the intersection of $\Ds$ with a face of $\D$.
%\end{definition}

%\begin{definition} \label{polytope_def}
%    \textit{Lattice polytope} in $\R^n$ is a set which admits lattice polyhedral subdivision (for instance, it also admits triangulation), i.e. it is the finite union of lattice $n$-dimension convex polytopes, such that each two of them intersect by a common face (possibly empty). 
%\end{definition}

%\begin{definition}
%        The Minkowski sum $\D_1 + \D_2$ of two polytopes is the set of all vector sums $p_1+p_2, \ p_1 \in \D_1, \ p_2 \in \D_2$.
%\end{definition}

Let us recall the mixed volume and Bernstein-Kouchnirenko-Khovanskii theory (BKK-theory in short). For an introduction to the convex polyhedral geometry see \cite{Ew}, for a survey on this connection between polyhedral geometry and algebraic geometry see \cite{EKK}.

\begin{sign}\label{lattice_vol_sign}
        Suppose that $\R^k$ is an affine space with lattice $L\subset \R^k$ of full rank. We use the \textbf{lattice} volume $\Vol_\Z$, which differs from the usual one by a factor, so that the lattice volume of the initial simplex in the lattice $L$ is equal to $1$. Note that the lattice volume of any lattice polytope is integer. In the standard $\R^n$ we consider the standard lattice $\Z^n$. When we consider any operations with affine spaces we assume that we apply the same operations to the corresponding lattices.
\end{sign}

\begin{definition}\label{mv}
The \textbf{mixed volume} $\MV$ is the unique symmetric multilinear function on $n$-tuples of convex lattice polytopes in $\R^n$ that satisfies the following conditions:
		\begin{enumerate}
			\item It is symmetric in its arguments.
			\item $\MV(P, \dots, P) = \Vol_\Z(P)$ for any polytope $P$.\\
			\item For any polytopes $P_1,P_2,\dots , P_n$ and $P_1^\prime$ and non-negative numbers $\lambda, \lambda^\prime$ the following equation holds 
			$$\MV(\lambda P_1 + \lambda^\prime P_1^\prime, P_2, \dots, P_n) = \lambda \cdot \MV(P_1,P_2, \dots, P_n) + \lambda^\prime \cdot  \MV(P_1^\prime, P_2, \dots, P_n),$$
			where the sum in the left hand side is the \textbf{Minkowski sum} and multiplication is the homothety with the corresponding coefficient.
		\end{enumerate}
	\end{definition}

\begin{definition}\label{support_face_def}

Consider a polytope $P \subset \R^n$ and a covector $\xi \in (\R^n)^*$. The \textbf{support face} $P^\xi$ is the maximal face of $P$ on which the linear functional $\xi$ attains its maximum. Consider polynomial $f = \sum\limits_{\omega\in \Z^n} a_\omega z^\omega$ with Newton polytope $N \subset \R^n$. For a polytope $\Gamma \subset \R^n$ denote by $f|_\Gamma$ the restriction $\sum\limits_{\omega \in \Gamma} a_\omega z^\omega$ of $f$ on $\Gamma$. We denote by $f^\xi$ the restriction on the support face $f|_{N^\xi}$.
\end{definition}

\begin{theorem}[Bernstein-Kouchnirenko] \cite{Ber75} \label{bk_th}
    Consider Laurent polynomials $f_1,\dots, f_n \in \C[z_1^{\pm 1},\dots, z_n^{\pm 1} ]$ with Newton polytopes $P_1,\dots,P_n$ and generic coefficients. Then the number of solutions $f_1=\dots = f_n = 0$ in the torus $(\C \setminus0)^n$ is equal to the mixed volume $\MV(P_1,\dots P_n)$.

    The non-degeneracy conditions are as follows: for any covector $\xi \in (\R^n)^* \setminus 0$ the system $f_1^\xi = {\dots} = f_n^\xi = 0$ has no solutions in the torus $(\Cc)^n$.
\end{theorem}

The Bernstein-Kouchnirenko-theorem admits the following interpretation in terms of the line bundles on toric varieties (for an introduction to toric varieties see e.g. \cite{Ful}). We formulate it in the unmixed case, but of course, there exists a mixed analogue. Let $P\subset \R^n$ be an $n$-dimensional lattice polyhedron, let $\Sigma$ be its dual fan. Denote by $\mathbb T^\Sigma$ the corresponding (projective) toric variety. Consider holomorphic sections $f_1,\dots,f_n$ of the line bundle $I_P$ on $\mathbb T^\Sigma$ defined by $P$. For the definition of the intersection number see e.g. \cite[Section 1.3]{Es10}.

\begin{remark}\label{bk_th_toric_rem}
    The intersection number equals the lattice volume:
    $$\text{mult}(f_1\cdot{\dots} \cdot f_n\cdot \mathbb T^\Sigma) = \Vol_\Z (P)$$
    For finite zero sets (the case of complete intersection), the intersection number equals the sum of positive local multiplicities.
\end{remark}

Note that (for instance, by the Bernstein-Kouchnirenko theorem) the mixed volume of any collection of convex lattice polytopes is non-negative and integer. The following criteria hold for the mixed volume to be equal to zero or one.

\begin{proposition}(\cite{Kh78}; see \cite[Lemma 1.2]{Es10} for a proof). \label{mv=0_prop}
The mixed volume of convex polytopes is zero if and only if some $k$ of these polytopes sum up to a polytope of dimension strictly smaller than $k$.
\end{proposition}

\begin{proposition} \cite{EG15}\label{mv=1_prop}
    A collection of $n$ lattice polytopes in $\R^n$ has the mixed volume equal to one if and only if:
\begin{enumerate}
\item the mixed volume is not zero, and
\item there exists $k > 0$ such that, after a suitable translation:
\begin{enumerate}
\item $k$ of the polytopes are faces of the same $k$-dimensional (normalized) volume $1$ lattice simplex in a $k$-dimensional rational subspace $\R^k\subset \R^n$;
\item the images of the other $n-k$ polytopes under the projection $\R^n\to \R^n/\R^k$ have mixed volume $1$.
\end{enumerate}
\end{enumerate}
\end{proposition}

\begin{theorem}[Bernstein-Kouchnirenko-Khovanskii] \cite[Section 3]{Kh78}\label{BKK_th}
Consider Laurent polynomials $f_1,\dots, f_k \in \C[z_1^{\pm 1},\dots, z_n^{\pm 1} ]$ with Newton polytopes $P_1,\dots,P_k$ and generic coefficients. Then the Euler characteristic of the complete intersection equals the following:
$$\chi (\{f_1 = \dots = f_k = 0\}) =  (-1)^{n-k}  \sum\limits_{\substack{a_0 > 0, \dots , a_k > 0 \\ a_0+\dots + a_k = n}} \MV(\underbrace{P_0,\dots,P_0}_{a_0}, \dots, \underbrace{P_k,\dots,P_k}_{a_k}) $$
\end{theorem}

\begin{definition} \label{Cayley_sum_def}
    Let $k \le n$ be nonnegative integers. Let $e_1, \dots ,e_k$ be the standard basis of $\Z^k$. For finite sets $\Ps_0, \Ps_1,\dots, \Ps_k \subset \Z^{n-k}$, the \textbf{Cayley sum} is defined to be 
    $$\Ps_0 \ast \Ps_1 \ast \cdots \ast  \Ps_k = (\Ps_0 \times \{0\}) \cup (\Ps_1 \times \{e_1\}) \cup \dots \cup (\Ps_k \times  \{e_k\})$$
    The Cayley sum of convex polytopes is defined as the convex hull of the corresponding shifts of the polytopes.
\end{definition}

The volume of the Cayley sum of convex polytopes can be expressed in terms of mixed volumes. The original proofs of these formulas use algebraic geometry (more precisely, they consider the Euler characteristic of certain hypersurfaces in complex tori). These results also admit purely combinatorial proofs using the combinatorial version of the Cayley trick \cite[\S 5]{Stur94}.

\begin{remark}\cite[Below Remark 6.6]{DK86} \label{cayley_vol_rem}
    Consider lattice polytopes $P_0,\dots P_k \subset \R^{n-k}$. Then the volume of their Cayley sum can be calculated in terms of their mixed volumes:
    $$\Vol_\Z(P_0 \ast P_1 \ast \cdots \ast  P_k) = \sum\limits_{\substack{a_0 \ge 0, \dots , a_k \ge 0 \\ a_0+\dots + a_k = n-k}} \MV(\underbrace{P_0,\dots,P_0}_{a_0}, \dots, \underbrace{P_k,\dots,P_k}_{a_k})$$
\end{remark}

\begin{lemma}\cite[Lemma 1.7]{Es10} \label{cayley_signed_vol_lem}
    Consider lattice polytopes $P_0,\dots, P_k \subset \R^{n-k}$. Then we have the following formula
    $$\sum\limits_{\substack{I \subset \{1,\dots,k\}\\ I =\{i_1,\dots, i_{|I|}\} }} (-1)^{k-|I|} \Vol_\Z(P_{i_1} \ast \cdots \ast  P_{i_{|I|}}) = \sum\limits_{\substack{a_0 > 0, \dots, a_k > 0 \\ a_0+\dots + a_k = n-k}} \MV(\underbrace{P_0,\dots,P_0}_{a_0}, \dots, \underbrace{P_k,\dots,P_k}_{a_k})$$

    In other words, the formula in the BKK-theorem can be expressed as a Newton number. Denote by $C$ the cone $\R^{k}_{\ge 0} \oplus \R^{n-k}$ and consider any point $Q\in \Z^{n}$, then under assumptions of Theorem \ref{BKK_th} we have:

    $$\chi (\{f_1= \dots = f_k = 0\}) =  (-1)^{n-k} \nu_{C} (Q\ast P_1\ast\dots\ast P_k)$$
    
\end{lemma}

For a more general formula for the mixed volume of Cayley sums see \cite[Theorem 24]{Es12}.

\subsection{Preliminaries: Zero-dimensional critical complete intersections}\label{0-cci_subsec}

\subsubsection{Euler obstructions of sets}

Let us recall the formula for the number of points of zero-dimensional critical complete intersections from \cite[Lemma 2.50]{Es18}. First, we recall the Euler obstructions of toric varieties from \cite{MT11} (for an introduction to the toric varieties see e.g. \cite{Ful}). Euler obstructions were introduced in \cite{Mcp74} for constructing characteristic classes of singular complex algebraic varieties. We use the notations from \cite[Section 1.5]{Es10}.

\begin{definition}
A subset $\As'\subset \As\subset\Z^k$ is called a \textbf{face}
of $\As$, if it can be represented as the intersection of $\As$ with
a face of the convex hull of $\As$. The \textbf{dimension} $\dim \As$ of $\As$ is the dimension of its convex hull.
\end{definition} 

For a face $\As'$ of a finite set $\As\subset\R^k$, let $M'$ and $M\subset\R^k$ be the vector spaces, parallel to the affine spans of $\As'$ and $\As$
respectively. We denote the projection $\R^k\to\R^k/M'$ by $\pi_{\As'}$. Then the lattice volume of the set difference $\conv(\pi_{\As'}(\As)) \setminus \conv(\pi_{\As'}(\As\setminus \As'))$ is denoted by $c^{\As'}_\As\in\Z$.
Set $c^{\As}_\As = 1$, and set $c^{\As'}_{\As} = 0$ if $\As'$ is not a face of $\As$. These coefficients arise, for instance, in the Gelfand-Kapranov-Zelevinsky decomposition formula (see \cite[Theorem 1.2 from Chapter 10]{GKZ94} and \cite[Proposition 2.10]{Es10}).

Consider the square matrix $C$ with entries $c^{\As''}_{\As'}$, where
$\As''$ and $\As'$ run over of all faces of $\As$. Define
$e^{\As''}_{\As'}$ as the $(\As'', \As')$-entry of the inverse matrix $C^{-1}$.
Note that $C$ is upper triangular with $1$'s on the diagonal, if
we order faces of $\As$ by their dimension; in particular, $\det(C) = 1$, and $C^{-1}$ has integer entries.

\begin{definition}
The number $e^{\As'}_{\As}$ is called the \textbf{Euler
obstruction} of the set $\As$ at its face $\As'$.
\end{definition}

See e.g. \cite[Examples 1.24 and 1.26]{Es10} for illustrations of these combinatorial definitions. Euler obstructions of sets have the following geometric meaning.

\begin{proposition}\cite{MT11}
The Euler obstruction of the set $\As \subset \Z^k$ at its face $\As'$ equals $(-1)^{\dim \As-\dim \As'}$ times the Euler obstruction of the toric variety corresponding to $\As$ at a point in the orbit corresponding to $\As'$.
\end{proposition}

\begin{definition}\label{smooth_set_def}
    A set $\As$ is called \textbf{smooth} at its face $\As'$ if $\pi_{\As'} (\conv(\As)) \setminus \pi_{\As'}(\conv(\As \setminus \As'))$ is a unimodular simplex.
\end{definition}

\begin{remark}\label{eu=1_rem}
    By the definition of the Euler obstruction, its value is $1$ on the smooth stratum. The projective toric variety corresponding to the set $\As$ is smooth at the orbit corresponding to its face $\As'$ if and only if the set $\As$ is smooth at $\As'$ and $\As' \cap \Z^n = \text{aff}(\As') \cap \Z^n$ (see e.g. \cite[Theorem 3.16 from Chapter 5]{GKZ94}). However, one can verify combinatorially that $e^{\As'}_{\As} = (-1)^{\dim \As - \dim \As'}$ for any smooth face $\As'$ of $\As$.
\end{remark}

\subsubsection{Zero-dimensional critical complete intersections}

Denote the $(n+1)$-dimensional space of affine linear functions on $\Z^n$ by $(\Z^n)^\star$, to be distinguished from the dual space $(\Z^n)^*$. For every subspace $L\subset\Z^n$, let $L^\perp\subset (\Z^n)^\star$ be the set of affine functions on $\Z^n$ that vanish on $L$. For a Laurent polynomial $f(z)=\sum_{b\in\Z^n} c_bz^b$ on $\CC^n$ with the standard coordinates $z=(z_1,\ldots,z_n)$ and an element $\alpha\in(\Z^n)^\star$, denote the polynomial $\sum_{b\in\Z^n} \alpha(b)c_bz^b$ by $\partial_\alpha f$.

\begin{ex} If $\alpha$ is the $i$-th coordinate function, then $\partial_\alpha f=z_i\frac{\partial f}{\partial z_i}$.
\end{ex}

\begin{definition}The set  $\{\partial_D f=0\}\subset\CC^n$, given by the equations $\partial_\alpha f=0$, $\alpha\in D$, is called a {\bf critical complete intersection}.
\end{definition}

Assume that $D\subset(\Z^n)^\star$ is a hyperplane. It is uniquely determined by the common zero $D^\star\in\mathbb Q \mathbb P^n$ of all affine linear functions $\alpha\in D$. We also denote by $Q^\star, Q \in \mathbb{QP}^n$ the hyperplane in $(\Z^n)^\star$ consisting of all affine functions vanishing at $Q$. Let $\mathfrak{A}_\As (D^\star)$ be the set of faces $\Gamma$ of $\conv(\As)$ such that the affine span of $\Gamma$ contains $Q$.

\begin{proposition}\label{deg_cci_prop}\cite[Proposition 2.43 and 2.47]{Es18}
Consider a generic polynomial $f$ with support $\As$. 
The critical complete intersection defined by $\{\partial_D (f) = 0\} \subset (\Cc)^n$ is indeed a complete intersection (i.e. its dimension is zero). The system $\partial_D (f)$ can be degenerate at a face $\Gamma \subset \conv (\As)$ (in the Bernstein-Kouchnirenko sense Theorem \ref{bk_th}) only if $\Gamma \in \mathfrak A_\As(D^\star)$.
\end{proposition}

\begin{definition}\label{e_nn_def}
    Define the \textbf{$\mathbf e$-Newton number} $\nu_Q^e(\As)$ to be the sum $\sum_{\Gamma\in\mathfrak{A}_\As(Q)} e^\Gamma_\As \Vol_\Z \Gamma$, where $e^\Gamma_\As$ is the Euler obstruction of $\As$ at $\Gamma$.
\end{definition}

\begin{remark} \label{C_Q(P)_def_rem}
    One can think of $\nu_Q^e(\As)$ as the $e$-Newton number with respect to the convex cone $C_Q(\Ps)$ with vertex $Q$ and spanned by $\As$. It is important for the following subsection, where we compare different Newton numbers.
\end{remark}

\begin{lemma} \label{dim0crit_lem} (The statement appears in \cite[Lemma 2.50.]{Es18}, see \cite[Theorem 1.8]{Es24b} for a more detailed proof). Consider a finite set $\As \subset \Z^n$ and a hyperplane $D$ in $(\Z^n)^\star$, then, for a generic Laurent polynomial $f(x)=\sum_{z\in A} c_az^a$, the zero-dimensional critical complete intersection $\{\partial_D f=0\}$ consists of $\nu_{D^\star}^e(\As)$ points (counted with multiplicity).
\end{lemma}

\begin{remark}
    The above lemma is proved in \cite[Theorem 1.8]{Es24b} for $D^\star = 0$. This immediately implies the lemma for arbitrary $D^\star\in\Z^n$ since it corresponds to the shift of $\As$ by $-D^\star$. The proof for arbitrary $D^\star \in \mathbb{QP}^n$ is the same; it can be also proved by constructing an appropriate set $\As^{+1} \subset \Z^{n+1}$, so that $\nu^e_{D^\star}(\As) = \nu^e_O(\As^{+1})$ (as in Remark \ref{-1dim_rem}).
\end{remark}

\begin{remark}\label{inv_nu_e_rem}
    In \cite[Lemma 2.50]{Es18} it is said that the formula from Lemma \ref{dim0crit_lem} ``is the Bernstein formula \cite{Ber75} with some roots hidden at infinity (more precisely, at the orbits of the $\As$-toric variety, corresponding to the faces from $\As$)''. It is easier to see that if we invert the formula:
    $$\Vol_\Z (\conv (\As)) = \sum\limits_{\Gamma \in \mathfrak A_\As(Q)} c_\As^\Gamma \cdot  \nu^e_Q(\Gamma \cap \As)$$
    Namely, by the version of the Bernstein-Kouchnirenko theorem from Remark \ref{bk_th_toric_rem}, the total number of solutions on the $\As$-toric variety of the system $\{\partial_D f = 0\}$ (considered as the sections of $I_P$) is equal to $\Vol_\Z (\conv (\As))$. But the system is not Newton non-degenerate and some solutions are on tori corresponding to proper faces of $\As$. For each face $\Gamma \in \mathfrak A_\As(Q)$, we can consider the subsystem $\{\partial_{\Gamma^\star} f = 0\}$ defined by linear functionals vanishing on $\Gamma$. One can check that the non-degeneracy conditions from  \cite[Lemma 1.28]{Es10} are satisfied, so \cite[Lemma 1.28]{Es10} implies that the system $\{\partial_{\Gamma^\star} f = 0\}$ contains the toric orbit corresponding to $\Gamma$ with multiplicity $c_\As^\Gamma$. The total multiplicity of solutions of the ``factor-system'' $\{\partial_{Q^\star / \Gamma^\star} f = 0\}$ restricted to the corresponding orbit is equal to $\nu_Q^e(\Gamma \cap \As)$ by the definition of the $e$-Newton number.
    
    Thus, on the orbit corresponding to $\Gamma \in \mathfrak A_\As(Q)$, there is a discrete set of solutions with total multiplicity $c_\As^\Gamma \cdot  \nu^e_Q(\Gamma \cap \As)$.
\end{remark}

\begin{ex}\label{crit_ex}
    The number of critical points $\# \{\frac{\partial f}{\partial x_1} = \dots = \frac{\partial f}{\partial x_n} = 0\}$ of a generic polynomial $f$ with support $\As$ in the torus equals $\nu^e_O (\As)$, where $O$ is the origin.
\end{ex} 

In the case of a critical complete intersection, the Bernstein-Kouchnirenko-Khovanskii non-degeneracy conditions are violated (otherwise the number of solutions would have been $\Vol_\Z (Conv(\As))$). However, the degeneracy of such systems can be controlled and analogues of the Bernstein-Kouchnirenko-Khovanskii formulas can be obtained. The geometry of such (and more general) complete intersections was recently studied in \cite{Es24a}, \cite{Es24b} and \cite{Es25}.

\subsection{Comparison of the Newton numbers: $\nu$, $\nu^\ell$ and $\nu^e$}\label{Comparison_sec}

In this subsection we compare different types of Newton numbers: the usual Newton number (Definition \ref{newton_def}), the $\ell$-Newton number (Definition \ref{l_nn_def}) and the $e$-Newton number (Definition \ref{e_nn_def}).

\begin{remark}
    Both $\nu^\ell$ and $\nu^e$ are non-negative (by Theorem \ref{loc_h*_ag_th} and Lemma \ref{dim0crit_lem}), while the usual $\nu$ is generally not (see \cite[Example 2.5 on p.353]{GKZ94}).
\end{remark}

\begin{remark}\label{ell=nu_rem}
    Suppose that $P$ is a convenient polytope in a cone $C$ and the face poset $\Cp$ of $C$ is boolean (i.e., the projection of $C$ along the maximal vector subspace in it is a simplicial cone). Then by Theorem \ref{loc_h*_ag_th} we have $\nu_C(P) = \nu^\ell_C(P)$.
\end{remark}

\begin{lemma}\label{nu>nu_e_lem}
    Consider a finite set $\Ps\subset \Z^n$ such that $P = \conv(\Ps)$ is a convenient polytope in the cone $C = \R^m_{\ge 0} \oplus \R^{n-m}$ and a point $X \in \Z^{n-m}$. Then we have:
    $$\nu_C(P) = \nu^\ell_C(P) \ge \nu^e_X (\Ps)$$
\end{lemma}

\begin{proof}
    Without loss of generality we may assume that $X= O$ is the origin (the corresponding shift does not change the usual Newton number).
    
    Consider the dual fan $\Sigma$ of $P$ with support $(\R^n)^{*}$ and the corresponding toric variety $\mathbb T^\Sigma$. Suppose that $e_1,\dots,e_m, e_{m+1}, \dots, e_n$ is the standard basis of $\Z^m\oplus \Z^{n-m}$. Suppose that $f$ is a generic polynomial with support $\Ps$. Consider $\partial_{e_i} f = z_i \frac{\partial f}{\partial z_i}$ as sections of the line bundle $I_P$ corresponding to the polytope $P$ on $\mathbb T^\Sigma$.

    Consider rays $\xi_1,\dots ,\xi_m \subset (\Z^n)^*$ corresponding to $e_1,\dots ,e_m \subset \Z^n $ in the dual basis. Note that the section $\partial_{e_i} f = z_i \frac{\partial f}{\partial z_i}, i \le m$ of $I_P$ vanishes at the orbit corresponding to $\xi_i$. Thus, the divisor of $I_P$ is the sum of two effective divisors $D_{\xi_i} + D_{f,i}$, where $D_{\xi_i}$ is the reduced closure of the orbit, corresponding to $\xi_i$.

    Recall that the sections $\partial_{e_1} f = \dots = \partial_{e_n} f = 0$ of $I_P$ have only isolated intersections on $\mathbb T^\Sigma$ (see Remark \ref{inv_nu_e_rem}). Thus the global intersection number is the sum of positive local intersection multiplicities. By Remark \ref{bk_th_toric_rem}, the sum of multiplicities of $\{\partial_{e_1} f = \dots = \partial_{e_n} f = 0\}$ on $\mathbb T^\Sigma$ is equal to $\Vol_\Z (P)$. The inclusion-exclusion principle (similar to \cite[Section 4]{Oka79}) implies that:

    \begin{equation} \label{nu_c=mult_eq}
        \nu_C(P) = \operatorname{mult} (D_{f,1} \cdot \ldots \cdot D_{f,n} \cdot \mathbb T^\Sigma)
    \end{equation}

    Recall that the number of points of $\{\partial_{e_1} f = \dots = \partial_{e_n} f = 0\}$ in the open torus equals $\nu^e_O (\Ps)$. Thus it is less than or equal to the total multiplicity in $\mathbb T^\Sigma$. 
\end{proof}

It is also easy to prove Lemma \ref{nu>nu_e_lem} for an arbitrary cone in dimension $2$. We conjecture that the same inequality holds for every cone.

\begin{con}\label{nu_ell>nu_e_con}
    Consider a cone $C$ and a finite set $\Ps \subset C \subset \R^n$, such that $\conv (\Ps) = P$ is convenient in $C$. Suppose that $C_\text{min}$ is the minimal face of $C$. Then for every $X \in C_{\text {min}} \cap \Z^n$ we have the following inequalities between the $\ell$- and the $e$-Newton numbers:
    $$\nu_C^\ell(P) \ge \nu^e_X(\Ps)$$
\end{con}

The following remark is implied by the proof of Lemma \ref{nu>nu_e_lem} (more precisely, by equation (\ref{nu_c=mult_eq})); it is also implied by \cite[Theorem 5.10 (2)]{Sel24}.

\begin{remark}\label{mon_nu_rem}
    The Newton number in the cone $C = \R^m_{\ge 0} \oplus \R^{n-m}$ is monotonic, i.e., if $P_1\subset P_2\subset C$ are convenient, then $\nu_C(P_1) \le \nu_C(P_2)$.
\end{remark}

\begin{lemma}\label{nu_e=mv_lemma}
    Consider a cone $C\subset \R^n, \dim C = n$ with boolean face poset $\Cp$, i.e., its projection along the maximal affine subspace of $C$ is a simplicial cone. Denote its facets by $E_1,\dots,E_m$. Consider a lattice set $\Ps \subset C$, the polytope $P = \conv (\Ps),\dim P = n$ and a point $Q\in E_1\cap\dots\cap E_m$. Suppose that the affine span of any face in $\mathfrak A_\Ps (Q)$ coincides with the affine span of a face of $C$. Then we have the following formula:
$$\nu_Q^e(\Ps) = \MV(\conv(\Ps \setminus E_1), \dots, \conv(\Ps \setminus E_m), \underbrace{P, \dots, P}_{n-m})$$
\end{lemma}

\begin{remark}
    The conditions of the above lemma are applied, for instance, if $P$ is a convenient and $Q\in E_1\cap\dots\cap E_m$ is a generic point. Also note that the conditions of the lemma imply that if for a face $C'\le C$ its affine span coincides with the affine span of a face in $\mathfrak A_\Ps (Q)$, then for any face $C'\le C''\le C$ the affine span of $C''$ also coincides with the affine span of a face in $\mathfrak A_\Ps (Q)$.
\end{remark}

\begin{proof}
Consider the covectors $\xi_1,\dots, \xi_m$ orthogonal to the facets of $C$. Consider $n-m$ affine functions $\psi_{m+1}, \dots,\psi_n$ vanishing on $Q$ so that $\xi_1,\dots, \xi_m, \psi_{m+1}, \dots,\psi_n$ form a basis in $Q^\star$. The critical complete intersection is defined by the system $\partial_{\xi_1} f =\dots = \partial_{\xi_m} f = \partial_{\psi_{m+1}} f = \dots = \partial_{\psi_n} f$. Note that the Newton polytopes of the corresponding polynomials are contained in the following polytopes $$Conv(\Ps\setminus E_1) , \dots, Conv(\Ps\setminus E_m), \underbrace{P,\dots,P}_{n-m}$$ respectively. Let us verify that this system is non-degenerate in the Bernstein-Kouchnirenko (see Theorem \ref{bk_th}).
This implies that the number of solutions equals the mixed volume.

We need to check that  for generic $f$ with support $\Ps$ for each $\xi \in (\R^n)^* \setminus 0$ the system 
\begin{equation}\label{check_sys}
    (\partial_{\xi_1} f) ^\xi =\dots = (\partial_{\xi_m} f)^\xi = (\partial_{\psi_{m+1}} f)^\xi = \dots =(\partial_{\psi_n} f)^\xi = 0,
\end{equation}
considered as system of polynomials with Newton polytopes $Conv(\Ps\setminus E_1) , \dots, Conv(\Ps\setminus E_m), P,\dots,P$ respectively, has no solutions in the torus $(\Cc)^n$.

\begin{enumerate}
\item Suppose that $P^\xi \notin \mathfrak A_\Ps (Q)$; then by Proposition \ref{deg_cci_prop} the system 
\begin{equation}\label{init_sys}
    (\partial_{\xi_1} f)|_{P^\xi} =\dots = (\partial_{\xi_m} f)|_{P^\xi} = (\partial_{\psi_{m+1}} f)|_{P^\xi} = \dots =(\partial_{\psi_n} f)|_{P^\xi} = 0
\end{equation}
has no solutions in the torus. This implies that the system (\ref{check_sys}) has no solutions, since it contains all nonzero polynomials from the latter system (\ref{init_sys}).

\item Suppose that $P^\xi \in \mathfrak A_\Ps (Q)$, for instance, the affine span of $P^\xi$ coincides with the affine span of a proper face $C' \le C$. Denote by $\phi_1,\dots, \phi_{m'}$ a basis of covectors orthogonal to $C'$. Consider covectors $\eta_{m'+1},\dots, \eta_{n}$ vanishing on $Q$ so that $\phi_1,\dots, \phi_{m'}, \eta_{m'+1},\dots ,\eta_{n}$ form a basis in $Q^\star$. We can rewrite the system (\ref{check_sys}) as 

\begin{equation}\label{rewrite_sys}
  (\partial_{\phi_1} f) ^\xi =\dots = (\partial_{\phi_{m'}} f)^\xi = (\partial_{\eta_{m'+1}} f)^\xi = \dots =(\partial_{\eta_n} f)^\xi = 0
\end{equation}

Note that the subsystem of the last $n-m'$ equations coincides with the following:

\begin{equation}\label{init_res_sys}
    (\partial_{\eta_{m'+1}} f)|_{C'} =\dots = (\partial_{\eta_{n}} f)|_{C'} = 0 
\end{equation}

By Proposition \ref{deg_cci_prop} the latter system (\ref{init_res_sys}) has finitely many solutions in the factor-torus $(\Cc)^n/ \langle z^{\phi_1},\dots,z^{\phi_{m'}}\rangle$, where $\langle z^{\phi_1},\dots,z^{\phi_{m'}}\rangle$ is the subtorus generated by the corresponding one-parameter subgroups (formally speaking, $\phi_i$ are affine functionals and here we take only their linear part, but this minor abuse of notations simplifies the formula).

Consider each solution $x_\alpha$ of the system (\ref{init_res_sys}) in $(\Cc)^n/ \langle z^{\phi_1},\dots,z^{\phi_{m'}}\rangle$ separately. Let us substitute it in the first $m'$ equations of system (\ref{rewrite_sys}). We obtain a system of $m'$ equations in the torus $\langle z^{\phi_1},\dots,z^{\phi_{m'}}\rangle$ with supports contained in $\pi_{C'} (\Ps \setminus C')$, where $\pi_{C'}: \R^n \to \R^{m'}$ is the lattice affine projection whose kernel is the affine span $\text{aff}(C')$. The restriction $\xi_\pi$ of $\xi$ to $\pi_{C'} (\R^n)$ is well-defined since $\xi$ is constant on $C'$. Note that the affine span $\text{aff} ((\pi_{C'} (\Ps \setminus C')) ^{\xi_\pi}) \subset \R^{m'}$ does not contain the origin $\pi_{C'}(C')$, since $C'\cap P$ is the maximal face of $P$ on which the maximum of $\xi$ is attained. Thus the same argument as in item 1. (using Proposition \ref{deg_cci_prop}) implies that the system has no solutions. Considering every $x_\alpha$ in the finite set of solutions of the system (\ref{init_res_sys}) separately, we obtain that the system (\ref{rewrite_sys}) has no solution for generic $f$ with support $\Ps$.

\end{enumerate}

\end{proof}

%\begin{definition}
%    Convex polytope $P$ in cone $C$ is called \textbf{almost strongly convenient} if the poset $\Cp_P = \{ Y \in \Cp:\dim(P\cap Y) = \dim (Y)\}$ form an upper poset of $\Cp$, i.e. together with each face contains all the faces containing it.
%\end{definition}

The following lemma is implied by Remark \ref{eu=1_rem}.

\begin{lemma}\label{nu_e=nu_lem}
Consider a finite lattice set $\Ps$, its convex hull $P = \conv (\Ps)$, and a point $Q\in \mathbb Q^n$. Suppose that for any face $F \in \mathfrak A_\As(Q)$, the set $\Ps$ is smooth at $\Ps \cap F$. Then we have $\nu_Q^e(\As) = \nu_{\mathsmaller {C_Q(\Ps)}}(P)$, see Remark \ref{C_Q(P)_def_rem}.
\end{lemma}

\begin{remark}
Denote by $\Ds_m$ the set of vertices of the unit $m$-simplex in $\R^m_{\ge 0}$. Consider a convenient polytope $P = \conv (\Ps), \Ps \subset \Z^n$ in the cone $C = \R^m_{\ge 0} \oplus \R^{n-m}$. Suppose that $\Ds_m \subset \pi_{\R^{n-m}}(\Ps)$ and $X \in \mathbb Q^{n-m} \subset \R^{n-m}$ is a generic point (for instance, in the relative interior of $P \cap \R^{n-m}$). Then the conditions of the above lemma are satisfied and we have $\nu^e_X(\As) = \nu_C(P)$.
\end{remark}

\subsection{Gelfand-Kapranov-Zelevinsky polytopes via $\mathbf{e}$-Newton number}\label{eGKZ_subsec}

In this subsection we recall Gelfand-Kapranov-Zelevinsky polytopes (GKZ-polytopes in short) from \cite{GKZ94} and show that the $e$-Newton number computes their support function. This subsection is completely trivial from the point of view of GKZ-polytopes theory. It is not needed for our further results, but it is a good illustration of $e$-Newton numbers.

\subsubsection{Preliminaries: Gelfand-Kapranov-Zelevinsky polytopes}

Consider a finite set $\As \subset \Z^n$. Denote by $\C[\As]$ the set of Laurent polynomials with exponents of their monomials contained in $\As$. Sard's theorem implies that for generic polynomial $f \in \C[\As]$ with support $\As$ the hypersurface $\{f=0\}$ is smooth in the torus $(\Cc)^n$. Consider all polynomials $f \in \C[\As]$ such that the corresponding hypersurface $\{f = 0\} \subset (\Cc)^n$ is not smooth. The closure of this set is denoted by $\nabla_\As$. 

\begin{definition}\cite[Definition 1.2 from Chapter 9]{GKZ94}
	If the set $\As \subset \Z^n$ has the property that $\nabla_\As \subset \C[\As]$ is a subvariety of codimension 1, then by the $\As$\textbf{-discriminant} we mean an irreducible integral polynomial $\mathbf D_\As(f)$	in the coefficients $a_\omega,\omega \in \As$ of $f \in \C[\As]$, which vanishes on $\nabla_\As$. Such a polynomial is uniquely determined up to sign. If $\text{codim} \nabla_\As > 1$, we set $\mathbf D_\As = 1$ and say that the set $\As$ is \textbf{dual defective}.
\end{definition}

We recall a classification of dual defective sets in \S \ref{dual_rec_subsec}; for more details, see \cite{CD22}.

The Newton polytope $N_\As$ of $\mathbf D_\As$ is computed in \cite{GKZ94}, its tropical fan is computed in \cite{DFS07}. A similar problem for singular hyperspaces with codimension two singularities ($A_2$ or two $A_1$) is solved in \cite[Section 3.5]{Es18}.

\subsubsection{Support function of GKZ-polytope via $\mathbf{e}$-Newton number}

\begin{definition}\label{sup_fun_def}
	Consider a lattice polytope $P \in \R^n$ and a covector $\xi \in (\Z^n)^*$. \textbf{The support function} $\xi (P)$ is the maximum of the functional $\xi$ on $P$.
\end{definition}

Note that from the support function one can recover the whole polytope. For more details on polyhedral geometry see e.g. \cite{Ew}.

\begin{remark}\label{support_fun_rem}
	Consider polynomial $f \in \C[z_1,\dots,z_n]$ with Newton polytope $P \subset \R^n$ and a covector $\xi = (\xi_1,\dots,\xi_n)$. Then for generic $(a_1,\dots,a_n)$ the number of intersections of $\{f=0\}$ with the curve $(a_1 z_1^{\xi_1} , \dots, a_n z_n^{\xi_n})$ in the complex torus $(\Cc)^n$ (counted with multiplicities) is equal to $\xi(P) + (-\xi)(P)$.
\end{remark}

The latter remark implies the following proposition which allows to recover the Newton polytope of $\mathbf D_\As$ up to a shift (since $\mathbf D_\As$ is homogeneous by \cite[Chapter 9, Proposition 1.3.]{GKZ94}). For $\As \subset \Z^n$ we denote by $\langle \As - \As \rangle $ the subgroup of $\Z^n$ generated by the set $\{a_1 - a_2:  a_1,a_2 \in \As\}$ and say that $\As \subset \Z^n$ is \textbf{spanning}, if $\langle \As - \As \rangle = \Z^n$. In the following proposition we assume $\As$ to be spanning, but it is easy to extend it to non-spanning sets by multiplying the support function by the index of the corresponding sublattice (see \cite[Remark on p.18]{Es10}).

\begin{proposition}\label{gkz_from_nu_prop}
	Consider a finite spanning lattice set $\As \subset \Z^n$ and a covector $\xi \in (\Z^{|\As|})^*$ with non-negative coordinates. Denote by $\As_{+\xi}$ the set in $\Z^n \oplus \Z^1_{\ge 0}$ consisting of all the points of the form $(\mathbf a, 0), (\mathbf a, \xi_{\mathbf a}), \mathbf a \in \As$. Denote by $O_\infty$ the infinite point along the ray $\Z^1_{\ge 0}$. Then we have:
	
	$$\xi (N_\As) + (-\xi) (N_\As) = \nu^e_{O_\infty} (\As_{+\xi})$$

\end{proposition}

Note that the difference of the lattice volume $\Vol(\conv (\As_{+\xi}))$ and the $e$-Newton number $\nu^e_{O_\infty} (\As_{+\xi})$ is consistent with the GKZ-decomposition formula \cite[Chapter 10 Theorem 1.2]{GKZ94} (in our notations, see \cite[Proposition 2.10]{Es10}). 

\begin{remark} \label{Sot_rem}
The Newton number $\nu_{O_\infty}^e (\cdot)$ from Proposition \ref{gkz_from_nu_prop} also computes the number of complex critical points of the Bloch variety \cite[\S 2]{FS24}. This computation refines \cite[Theorems A and 3.8]{FS24}.
\end{remark}

\subsection{Algebraic degrees} \label{algebraic_subsec}

In this subsection, following \cite[Example 1.7(6)]{Es24b}, we recall that different algebraic degrees are equal to the number of points in zero-dimensional critical complete intersections. Hence, they are equal to the corresponding $e$-Newton numbers in the Newton non-degenerate case.

We show that the formulas for the maximum likelihood degree from \cite[Theorem 15]{CHKS06} and \cite{LNRW23}, the Euclidean distance degree (introduced in \cite{DHOST16}) from \cite{BSW22} and \cite{TT24}, and the polar degree from \cite[Corollary 4.6]{Huh13} are all implied by properties of the $e$-Newton number given in \S \ref{Comparison_sec}. Our formulas generalize those for maximum likelihood and Euclidean distance degrees, while for polar degree we provide new formulas (including expressions via the mixed volume).

We also discuss geometrically interesting cases, where the maximum likelihood degree is equal to $0$ or $1$ (see \cite{Huh14} and \cite[Theorem 1.11]{Es24c}), and the polar degree is equal to $0$ or $1$ (see \cite[Example 4.8. and below]{Huh13} and \cite{CRS08}).

%In this subsection we recall the different types of algebraic degrees, recall \cite[Example 1.7(6)]{Es24b} that these degrees are equal to the number of points of certain $0$-dimensional complete intersections. We provide formulas for all these degrees in terms of the (usual) Newton number. We show that this technique provides the formulas for the Eucledean distance degree from \cite{DHOST16}, \cite{BSW22} and \cite{TT24}, the formulas for the maximum likelihood degree (or signed Euler characteristic) from \cite[Section 3]{Kh78}, \cite[Theorem 15]{CHKS06} and \cite{LNRW23}. We provide analogues formula for the polar degree (or the degree of the gradient map) using the mixed volume and show that it is consistent with the formula \cite[Corollary 4.6.]{Huh13}. 

The following theorem is \cite[Example 1.7(6)]{Es24b} with several minor additions and corrections. We consider all the algebraic degrees in separate subsubsections below. In the following theorem we use the usual Lagrangian multiples, and the ones that we call ``outer'' Lagrangian multiples and ``dehomogenizing'' Lagrangian multiples. We call them this way since ``outer'' Lagrangian multipliers let us count the critical points outside the corresponding variety, and ``dehomogenizing'' Lagrangian multiples let us count the degree of the gradient map by extracting a single point from a projective line (for more details, see the corresponding subsubsections below).

There are two types of maximum likelihood degrees, we need to distinguish them in the theorem.

\begin{enumerate}
    \item For mappings (see \cite{CHKS06}), we call it the \textbf{MML-degree}.
    \item For subvarieties (see \cite{Huh13} and \cite{HS}), we call it the \textbf{ML-degree}.
\end{enumerate}

\begin{theorem}(cf. \cite[Example 1.7(6)]{Es24b})\label{alg_deg_th}
Algebraic degrees are equal to certain $e$-Newton numbers in the Newton non-degenerate case. Below we list all of them. Recall that \textbf{very affine variety} is a closed subvariety of the complex torus $(\Cc)^n$.

\begin{enumerate}
    \item \textbf{Mapping maximum likelihood degree (MML-degree)} from \cite{CHKS06}
    
    Consider a mapping  $\mathbf \Theta = (\Theta_1,\dots,\Theta_{m}): (\Cc)^{n-m} \to (\Cc)^{m}$. The critical points of the function $\mathbf \Theta ^ u = \Theta_1^{u_1} \dots \Theta_{m}^{u_{m}}$ on $(\Cc)^{n-m} \setminus\bigcup_i \{\Theta_i = 0\}$ are defined via ``outer'' Lagrangian multipliers:
    
    $$\partial_{Q_u^\star} (\mathbf \Theta_{\text{MML}}(z,\lambda)) = 0\text{,\quad  where}$$

    $$\mathbf \Theta_{\text{MML}}(z,\lambda) = c + \sum\limits_{i=1}^{m} \Theta_i\lambda_i \quad \quad c \in \Cc  \quad \quad Q_u  =  (u_1:{\dots}:u_{m}:\underbrace{0:{\dots}:0}_{{n-m+1}}) \in \mathbb Q \mathbb P^n$$

    In the Newton non-degenerate case with supports $\Ps_1,\dots ,\Ps_m$, the MML-degree equals:

    $$\text{MMLdeg} (\mathbf \Theta,u) = \nu^e_{Q_u} (O\ast \Ps_1 \ast \dots \ast \Ps_m)$$
    
    \item \textbf{Maximum likelihood degree (ML-degree)} from \cite{Huh14}

    The critical points of $z^u$ on the smooth very affine variety $V_{\mathbf f} = \{f_1 = \dots = f_m = 0\}\subset (\Cc)^{n-m}$ are defined via Lagrangian multipliers:

    \begin{align*}
    &\partial_{O^\star} (F_{\text{ML}}(z,\lambda)) = 0 \quad \text{, where}\\
    &F_{\text{ML}}(z,\lambda) = z^u + \sum \limits_{i=1}^m\lambda_i f_i
    \end{align*}
    
    In the Newton non-degenerate case with supports $\Ps_1,\dots ,\Ps_m$, the ML-degree equals:
    
    $$\text{MLdeg} (V_{\mathbf f},u) =  \nu^e_O(u\ast \Ps_1\ast\dots\ast \Ps_m)$$
    
    \item \textbf{Euclidean distance degree (ED-degree)} from \cite{DHOST16}

    The critical points of the distance function $\rho_u$ on the very affine variety $V_{\mathbf f} = \{f_1 = \dots = f_m = 0\}\subset (\Cc)^{n-m}$ are defined via Lagrangian multipliers:

    \begin{align*}
    &\partial_{O^\star} (F_{\text{ED}}(z,\lambda)) = 0 \quad \text{, where}\\
    &F_{\text{ED}}(z,\lambda) = \rho_u + \sum \limits_{i=1}^m\lambda_i f_i
    \end{align*}

    Denote by $\mathbf \D_\rho$ the support of the quadratic distance function. In the Newton non-degenerate case with supports $\Ps_1,\dots ,\Ps_m$, the ED-degree is equal to:

    $$\text{EDdeg} (V_{\mathbf f}) =  \nu^e_O(\Ds_\rho\ast \Ps_1\ast\dots\ast \Ps_m)$$
    
    \item \textbf{Polar degree (P-degree)}

    Consider a degree $d$ homogeneous polynomial $f \subset \C[z_0,\dots,z_n],\ d\ge 2$, and a generic linear function $L_\mathbf{c} (z) = c + c_0z_0 + \dots + c_nz_n$. Then the preimages of the point $\mathbf c = (c_0:{\dots}:c_n) \in \C \mathbb P^n$ under the gradient map:
    $$\grad (f): \C \mathbb P ^n  \dashrightarrow \C \mathbb P^n, \quad \grad(f) = (\frac{\partial f}{\partial z_0}: \dots:  \frac{\partial f}{\partial z_n})$$
    are defined via ``dehomogenizing'' Lagrangian multipliers:

    \begin{align*}
    &\partial _{O^\star} (F_{\text{P}}(z,\lambda)) = 0 \quad \text{, where}\\
    &F_{\text{P}}(z,\lambda) = \lambda L_{\mathbf c} (z) + f(z)
    \end{align*}
    
    Denote by $\Ds_{n+1}$ the set of vertices of the unit $(n+1)$-simplex in $\Z^{n+1}$ (containing the origin). In the Newton non-degenerate case of homogeneous polynomial with support $\Ps\subset \Z^{n+1}$ the P-degree (i.e. the number of the preimages of a generic point $\mathbf c \in \C \mathbb P^n$) is equal to:

    $$\text{Pdeg} (f) =  \nu^e_O(\Ps \ast \Ds_{n+1}) = \nu_{\mathsmaller{ \R^{n+2}_{\ge 0}}} (\Ps \ast \Ds_{n+1})$$

    Note that the dimension $n+1$ of the simplex $\Ds_{n+1}$ can be reduced to $n$ (see Remark \ref{-1dim_rem}).
\end{enumerate}

\end{theorem}

\subsubsection{Mapping maximum likelihood degree}

Let us show that ``outer'' Lagrangian multipliers from Theorem \ref{alg_deg_th}(1) indeed compute the MML-degree.

\begin{proof}[Proof of Theorem \ref{alg_deg_th}(1)]

Assume all $u_i$ are nonzero (otherwise we may eliminate the corresponding $\Theta_i$; also note that the formulas below can be extended to cases where some $u_i=0$, but not all simultaneously). The equation $\partial_{Q_u^\star} (\mathbf \Theta_{\text{MML}}(z,\lambda)) = 0$ in  $(\Cc)^n$ in variables $\lambda_1,\dots,\lambda_{m}, z_1,\dots,z_{n-m}$ is equivalent to:

$$
\begin{cases}
    \Theta_1\lambda_1+ \ldots + \Theta_m\lambda_{m} + c = 0\\
    \frac{\lambda_1 \Theta_1}{u_1} = \ldots = \frac{\lambda_{m} \Theta_{m}}{u_{m}}\\
    \sum\limits_{i=1}^{m} \lambda_i z_1 \frac{\partial \Theta_i}{\partial z_1} = 0 \\
    \quad  \vdots \quad  \vdots\quad \vdots \\
    \sum\limits_{i=1}^{m} \lambda_i z_{n-m} \frac{\partial \Theta_i}{\partial z_{n-m}} = 0
\end{cases}
$$

%Assume that all $u_i$ are nonzero (otherwise we can cross out the corresponding $\Theta_i$; also note that it is possible to define what the formulas below mean in the case if some $u_i = 0$ but not all at once). The $u$-ML-degree can be expressed as the number of critical points of the following $0$-dimensional critical complete intersection using ``outer Lagrangian multipliers''. Namely, it equals the number of solutions of the following system in $(\Cc)^n$ in variables $\lambda_1,\dots,\lambda_{m}, z_1,\dots,z_{n-m}$:

Note that the first two equations have at most one solution in $\lambda$'s, existing if and only if all $\Theta_i \ne 0$. When a solutions exist, all $\lambda_i$ are nonzero. The remaining equations give exactly the conditions for critical point in the torus.

\end{proof}

%\begin{definition}[Mapping ML-degree]\label{map_ml_def}
%Consider Laurent polynomials $\Theta_1,\dots,\Theta_{m} \in \C [z_1,\dots,z_{n-m}]$. Suppose that $u = (u_1,\dots,u_{m})\in \Z^{m}$, then  \textbf{mapping} $\mathbf{u}$\textbf{-ML-degree} of the map $\mathbf \Theta = (\Theta_1,\dots,\Theta_{m}): (\Cc)^{n-m} \to (\Cc)^{m}$ is the number of critical points of the rational function $\mathbf \Theta ^ u = \Theta_1^{u_1} \dots \Theta_{m}^{u_{m}}$ on $(\Cc)^{n-m} \setminus\bigcup_i \{\Theta_i = 0\}$. 
%\end{definition}

%\begin{remark}\cite{CHKS06}\label{mapml_eu_rem}
%    Under certain non-degeneracy conditions, the mapping MML-degree is the same for almost all vectors $u \in \Z^{m}$. The value of the mapping $u$-ML-degree for almost all $u$ is called \textbf{mapping ML-degree}.
%\end{remark}

\begin{theorem}\cite[Theorem 19]{CHKS06}\label{MML=eu_ch_th}
    Under certain non-degeneracy conditions, the MML-degree of is equal to the signed Euler characteristic for almost all $u \subset \Z^n$: $$\text{MMLdeg}(\mathbf \Theta_{\text{MML}},u) = (-1)^{n-m} \chi( (\Cc)^{n-m} \setminus \bigcup\limits_{i=1}^{m} \{\Theta_i = 0\})$$
\end{theorem}

\begin{ex}\label{map_ml_rem}
The formula for the MML-degree for Newton non-degenerate polynomials and with certain smoothness assumptions from \cite[Theorem 15]{CHKS06} is of the following form:
\begin{multline*}
\label{longmultline}
  \sum_{1 \leq i_1 \leq \cdots \leq i_{n-m} \leq m} \MV(P_{i_1}, \dots,
P_{i_{n-m}}) \,\,\, - \!
\mathop{\sum_{j \in \{1, \dots, r\}}}_{1 \leq i_1 \leq \cdots
\leq i_{{n-m}-1} \leq m} \MV(P_{i_1}, \dots, P_{i_{{n-m}-1}}; \tau_{\{j\}})  \\
+ \,\,\mathop{\sum_{\{j_1,j_2\} \subset \{1, \dots, r\}}}_{1 \leq i_1 \leq
\cdots \leq i_{{n-m}-2} \leq m} \MV(P_{i_1}, \dots, P_{i_{{n-m}-2}};
\tau_{\{j_1,j_2\}}) \quad + \quad \cdots \quad + \\
(-1)^{n-m} \sum_{\{j_1, \dots, j_{n-m}\} \subset \{1, \dots, r\}} \MV(\emptyset;
\tau_{\{j_1,\ldots,j_{n-m}\}}).
\end{multline*}
This formula is exactly the combination of Theorem \ref{alg_deg_th}(1), Lemma \ref{nu_e=nu_lem} and Remark \ref{cayley_vol_rem} (the summands correspond to the summands in the Newton number). The smoothness assumptions in \cite[Theorem 15.]{CHKS06} is exactly the assumption that $\Ps^{\bigotimes} =  O\ast \Ps_1 \ast \dots \ast \Ps_m$ is smooth at every face in $\mathfrak A_{\Ps^{\bigotimes}} (Q_u)$. Note that our result generalizes \cite[Theorem 15]{CHKS06} by removing these smoothness assumptions.

For generic $u$ (and after a generic shift by $\alpha \in \Z^{n-m}$ corresponding to the multiplication by a generic monomial $z^\alpha$), only the first summand remains. It agrees with the computation of the corresponding Euler characteristic (Theorem \ref{MML=eu_ch_th}) via the inclusion-exclusion principle applied to the BKK-formula Theorem \ref{BKK_th}. The same Euler characteristic also computes the degree of the so-called ``toric polar map'' (see \cite[Theorem 1.1]{FMS24}).
\end{ex}

%\begin{ex}\label{map_ml_ex}
%        The $u$-ML-degree equals the number of points in the following critical complete intersection $\partial_{D_u} (\mathbf \Theta(z,\lambda)) = 0$, where
%\begin{enumerate}
%       \item Polynomial $\mathbf \Theta(z,\lambda) = c + \sum\limits_{i=1}^{m} \Theta_i\lambda_i$ considered as polynomial in $n$ variables ($\lambda$ and $z$).
%        \item Affine hyperplane $D_u\in L^\star$ consists of all affine functions vanishing at the point $$Q_u = (u_1:\dots:u_{m},\underbrace{0:\dots:0}_{{n-m+1}}) \in \mathbb Z \mathbb P^n,$$ where $(a_1:\dots a_m: b_1: \dots b_{n-m}:1)$ is the affine chart corresponding to the monomials in $\lambda$ and $z$.
%\end{enumerate}

%Recall that we denote by $O$ the origin. Thus the mapping $u$-ML-degree for Newton non-degenerate polynomials $\Theta_1,\dots, \Theta_{m}$ with Newton polytopes $P_1,\dots, P_m \subset \R^{m}$ is equal to $$\nu^e_{Q_u}(O\ast P_1\ast\dots\ast P_{m})$$
%\end{ex}

%Note that all the polynomials except for the first in the previous remark have Newton polytopes in the hyperplane $L$ = \{sum of first $(m)$-coordinates is equal to one\} and projection of the first one along this hyperplane is the unit segment.

%This implies the following remark generalizing \cite[Theorem 15.]{CHKS06} to the case without smoothness assumptions (on certain orbits of certain toric projective variety).

\subsubsection{Maximum likelihood degree of very affine varieties}

Theorem \ref{alg_deg_th} (2) follows directly from standard Lagrange multipliers techniques. Recall that the ML-degree is equal to the signed Euler characteristic in the smooth case.

\begin{theorem}\cite[Theorem 1]{Huh13}
    For a smooth very affine variety $U\subset (\Cc)^m$ and a generic character $\phi_u:(\Cc)^m \to \Cc, u \in \Z^m$ all critical points of $\phi_u$ on $U$ are non-degenerate and their number $\text{MLdeg}(U,u)$ is equal to the signed Euler characteristic $(-1)^{\dim U}\chi(U)$.
\end{theorem}

%\begin{definition}
%    Define $\mathbf{u}$\textbf{-Maximum Likelihood degree} (or $\mathbf{u}$\textbf{-ML-degree}) of smooth very affine variety as the number of critical points of $\phi_u$ on $U$. Define \textbf{Maximum Likelihood degree} (or \textbf{ML-degree}) as the $u$-ML-degree for generic $u$.
%\end{definition}

%\begin{remark}\cite[Example 1.5 (6)]{Es24b}
%    Critical points of $z^u$ on the very affine variety $\{f_1 = \dots = f_m = 0\}\subset (\Cc)^{n-m}$ are obtained from the critical point of the Lagrange multipliers function 
%    $$F(z,\lambda) = z^u + \sum \limits_{i=1}^m\lambda_i f_i$$
%    Thus for critical complete intersection defined by Newton non-degenerate polynomials with Newton polytopes $P_1,\dots,P_m$, the $u$-ML-degree equals:
%    $$\nu^e_O(u\ast P_1\ast\dots\ast P_m)$$
%\end{remark}

\begin{ex}\cite[Theorem 2.13]{LNRW23}\label{mldeg=mv_ex}
    For Newton non-degenerate polynomials $f_1,\dots,f_m \in \C[z^{\pm 1}_1,\dots,z^{\pm 1}_{n-m}]$ with Newton polytopes $P_1,\dots,P_m$ and generic $u\in \Z^{n-m}$, the ML-degree of the complete intersection $V_{\mathbf f} = \{f_1 = {\dots} = f_m = 0\}$ is:
    $$\text{MLdeg} (V_{\mathbf f},u) = \MV(P_1,\dots, P_m, \underbrace{P,\dots,P}_{n-m}),$$
    where $P = \{u\}\ast P_1\ast\dots\ast P_m \subset \R^{m}_{\ge 0} \oplus \R^{n-m}$.
\end{ex}

Note that Example \ref{mldeg=mv_ex} follows from Lemma \ref{nu_e=mv_lemma}. For generic $u$, Lemma \ref{nu_e=nu_lem} provides a formula for the ML-degree as the (usual) Newton number in $\R^{m}_{\ge 0} \oplus \R^{n-m}$. The result is consistent with the BKK-theorem (see Lemma \ref{cayley_signed_vol_lem}). Note that $P$ is not convenient in the cone $\R^{m}_{\ge 0} \oplus \R^{n-m}$ and we do not know how to obtain a formula for the ML-degree via the $\ell$-Newton number.

\begin{remark}
Characterizations exist for very affine varieties $V$ with:
\begin{itemize}
\item ML-degree 1 (via a rational map called ``Horn uniformization'' $\C \mathbb{P}^{m-1} \dashrightarrow V$, see \cite{Huh14})
\item ML-degree 0 (via invariance under subtorus actions, see \cite[Theorem 1.11]{Es24c})
\end{itemize}
\end{remark}

\begin{Que}\label{ml=vol01_que}
Can we relate the following for Newton non-degenerate complete intersections?
\begin{enumerate}
\item The classification of very affine varieties with ML-degree $0$ or $1$
\item Criteria for the mixed volume to be equal to $0$ or $1$ (Propositions \ref{mv=0_prop}, \ref{mv=1_prop})
\item Furukawa-Ito's classification of dual defect sets (Theorem \ref{dual_defect_classification_th} and Lemma \ref{dual=no_crit_points_lem})
\end{enumerate}
\end{Que}

\subsubsection{Euclidean distance degree}

%\begin{definition} \label{toric_ed_deg_def}
%    The \textbf{Euclidean distance degree (or ED-degree)} of very affine variety is the number of critical points on the non-singular part of the variety of the distance function to a generic point in the torus.
%\end{definition}

Theorem \ref{alg_deg_th}(3) follows directly from standard Lagrange multipliers techniques. Recall that $\Ds_\rho$ is the support of the quadratic distance function  (containing the origin, all coordinate basis vectors, and their doubles). Denote by $E_i$ the $i$-th coordinate hyperplane in $\R^n$.

\begin{ex}(\cite{BSW22} and \cite{TT24})
    For Newton non-degenerate polynomials $f_1,\dots,f_m \in \C[z^{\pm 1}_1,\dots,z^{\pm 1}_{n-m}]$ all whose supports $\Ps_1,\dots,\Ps_m$ contain the origin $O$, the ED-degree of the complete intersection $V_{\mathbf f} = \{f_1 = {\dots} = f_m = 0\}$ is equal to:
    $$\text{EDdeg} (V_{\mathbf f}) = \MV(\conv(\Ps^\rho \setminus E_1), \dots,  \conv(\Ps^\rho \setminus E_n)),$$
    where $\Ps^\rho = \Ds_\rho\ast \Ps_1\ast\dots\ast \Ps_m \subset \R^{m}_{\ge 0} \oplus \R^{n-m}$. Note that the first $m$ polytopes in the above mixed volume are $\conv(\Ps_1), \dots, \conv(\Ps_m)$. 
\end{ex} 

Lemma \ref{nu_e=mv_lemma} implies the latter example. We can also apply Lemma \ref{nu_e=nu_lem} to obtain a formula for the ED-degree via the (usual) Newton number. In this case the $e$-Newton number also equals the $\ell$-Newton number, see Remark \ref{ell=nu_rem}.

The $e$-Newton number formula (Theorem \ref{alg_deg_th}(3)) is more general since it does not require the supports to contain the origin. In this more general case the formulas with mixed volume and the (usual or $\ell$-) Newton number are no longer correct. Here is an example.

\begin{ex}\label{not_countre_ex}

Consider one polynomial $f \in \mathbb C[z_1,z_2,z_3,z_4]$ in $4$ variables with support $\Ps$ consisting of the following $6$ points and generic coefficients:

\begin{table}[H]
    \centering
    \begin{tabular}{c |c |c| c| c| c}
        $A$ & $B$ & $X$ & $Y$ & $V$ & $W$\\
        $e_1$ & $e_2$ & $e_1+e_3$ & $e_1+e_4$ & $e_2+e_3$ & $e_2+e_4$
    \end{tabular}
    \label{12}
\end{table}

\begin{figure}[H]
		\begin{center}
	\includegraphics[scale=4]{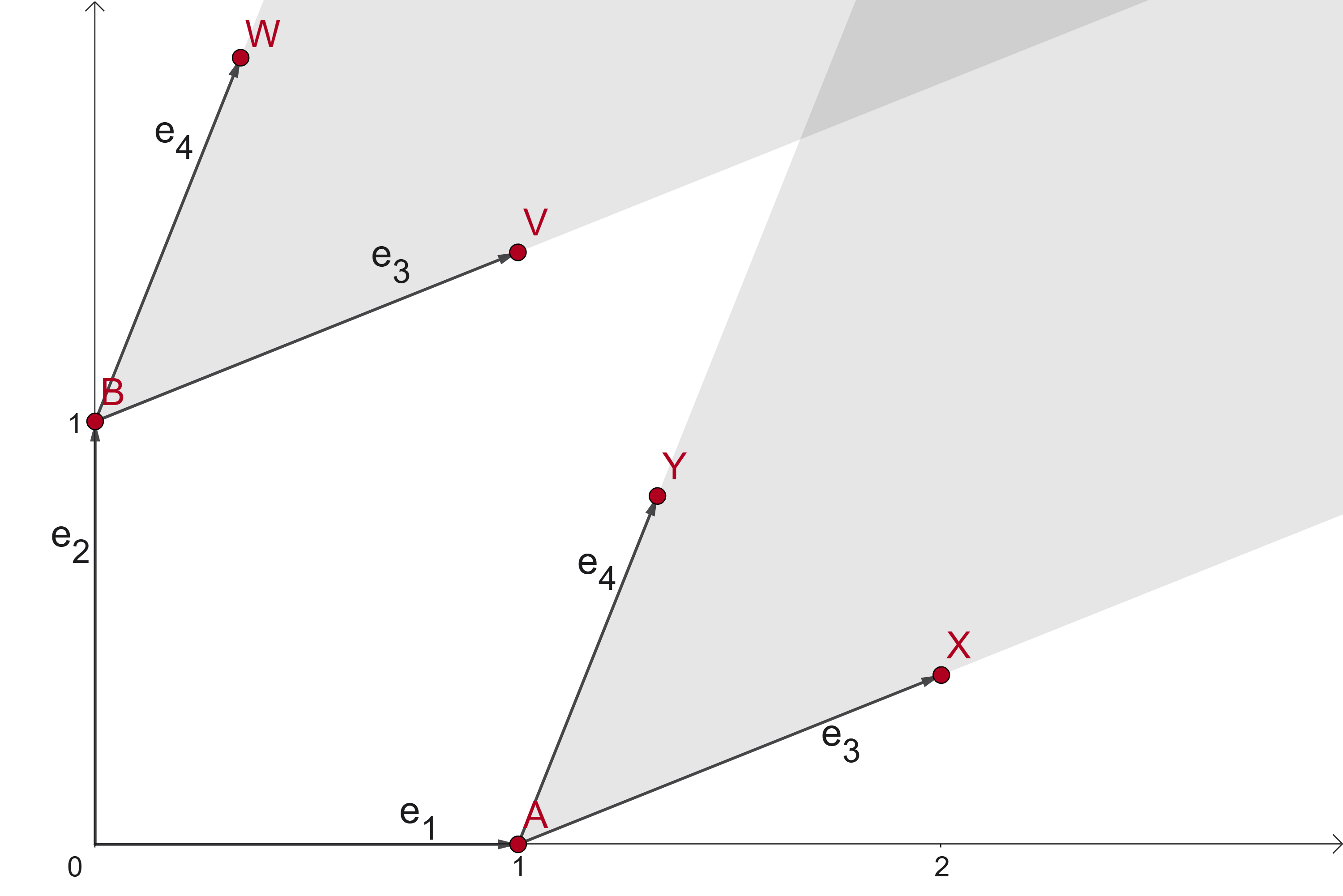}
		\end{center}
		\caption{
		\label{ed_example} The support $\Ps$ in dimension $4$}
\end{figure}

Denote by $\rho$ a generic quadratic distance function in $4$ variables. Let us verify that the system of derivatives
\begin{equation} \label{ed_sys}
    \frac{\partial (\lambda f + \rho)}{\partial z_1} = \frac{\partial (\lambda f + \rho)}{\partial z_2} = \frac{\partial (\lambda f + \rho)}{\partial z_3} = \frac{\partial (\lambda f + \rho)}{\partial z_4} = \frac{\partial (\lambda f + \rho)}{\partial z_\lambda} =0
\end{equation}
with the supports 
$$(\Ds_\rho \ast\Ps) \setminus E_1, \quad  (\Ds_\rho \ast\Ps) \setminus E_2, \quad (\Ds_\rho \ast\Ps) \setminus E_3, \quad (\Ds_\rho \ast\Ps) \setminus E_4, \quad \Ps $$
is degenerate for the covector $\xi = (-1,-1,0,0,1)$ in the Bernstein-Kouchnirenko sense Theorem \ref{bk_th}. The affine span of the support face $(\Ds_\rho \ast \Ps)^\xi$ has dimension $4$ and contains the origin. The lattice $\text{aff}(\Ds_\rho \ast \Ps)^\xi \cap \Z^5$ has the lattice basis 
$$v_1 = (1,0,0,0,1) \quad v_2 = (0,1,0,0,1) \quad v_3 = (0,0,1,0,0) \quad v_4 = (0,0,0,1,0).$$
In this basis the face $(\Ds_\rho \ast \Ps)^\xi$ consists of the following points:

\begin{table}[H]
    \centering
    \begin{tabular}{c |c |c| c| c| c| c| c|c|c|c}
        $O$ & $A'$ & $B'$ & $X'$ & $Y'$ & $V'$ & $W'$& $R_1$ & $S_1$ & $R_2$ & $S_2$\\
        $0$ & $v_1$ & $v_2$  & $v_1 + v_3$ & $v_1 + v_4$ & $v_2 + v_3$ & $v_2 + v_4$&  $v_3$ & $v_4$ & $2v_3$ & $2v_4$
    \end{tabular}
    \label{12}
\end{table}

Degeneracy of the system (\ref{ed_sys}) for $\xi$ is equivalent to the inequality $\nu^e_O((\Ds_\rho \ast \Ps)^\xi) > 0$ in the lattice span by $v_1,v_2,v_3,v_4$. By Lemma \ref{nu_e=mv_lemma} and Proposition \ref{mv=1_prop} we have:
$$\nu^e_O((\Ds_\rho \ast \Ps)^\xi) =  \MV (A'X'Y', B'V'W', X'V'R_1R_2, Y'W'R_2S_2) = 1,$$
thus the system (\ref{ed_sys}) is degenerate for $\xi$.

Using SageMath we computed that:
$$\MV (\conv((\Ds_\rho \ast\Ps) \setminus E_1),  \dots ,\conv( (\Ds_\rho \ast\Ps) \setminus E_4), \conv(\Ps)) = 7,$$
and that the number of solutions of the system (\ref{ed_sys}) in $\C^5$ is equal to $6$. So our SageMath computations also confirm that the system (\ref{ed_sys}) is degenerate.

\end{ex}

%\textcolor{red}{Нужна более общая лемма и раньше; может уже и не нужна}

%\begin{cor}\label{cayley_nu_cor}
%    Consider lattice polytopes $P_0,\dots P_k \subset \R^{n-k}$. Then the Cayley sum $P_0 \ast P_1 \ast \cdots \ast  P_k$ is contained in the cone $C = \R^{k}_{\ge 0} \oplus \R^{n-k}$ and the corresponding Newton number equals:
%    $$\nu_{C} (P_0 \ast P_1 \ast \cdots \ast  P_k) = \sum\limits_{\substack{a_0 \ge 0, a_1> 0, \dots , a_k > 0 \\ a_0+\dots + a_k = n-k}} MV(\underbrace{P_0,\dots,P_0}_{a_0}, \dots, \underbrace{P_k,\dots,P_k}_{a_k} )$$
%\end{cor}

\subsubsection{Polar degree}\label{Huh_pol_sec}

The study of the gradient map of a homogeneous polynomial is one of the central topics in classical projective geometry. See e.g. \cite{CRS08} and \cite{Dol} for a historical introduction.

First, we show that the “dehomogenizing” Lagrangian multipliers from Theorem \ref{alg_deg_th}(4) indeed compute the P-degree.

\begin{proof}[Proof of Theorem \ref{alg_deg_th}(4)]
    The equation $\partial _{O^\star} (F_{\text{P}}(z,\lambda)) = 0$ in variables $z_0,\dots,z_n$ and $\lambda$ is equivalent to the system:
$$
\begin{cases}
    \frac{\partial f}{\partial z_0} = \lambda c_0\\
    \dots \\
    
    \frac{\partial f}{\partial z_n} = \lambda c_n\\
    c_0z_0+ \dots + c_n z_n = c
\end{cases}
$$

First, note that for every homogeneous $f$ and generic $(c_0,\dots c_n)$ the system has no solutions for $c = 0$. Indeed, in this case this system implies $f = 0$ and if we replace the last line by $f = 0$, we see that the obtained system has no solutions for generic $(c_0,\dots,c_n)$ since the possible values of the gradient of the homogeneous polynomial on its zero locus form a variety of positive codimension.

Second, note that the first $n+1$ equations are equivalent to saying that $\text{grad} (f)$ maps the point $(z_0:{\dots} :z_n)$ to $\mathbf c = (c_0: {\dots} : c_n)$. As we proved above, for generic $\mathbf c$ all the points in the preimage satisfy $c_0z_0+ \dots + c_n z_n \ne 0$. Thus the equation $c_0z_0+ \dots + c_n z_n = c$ for $c \ne 0$ extracts a single ``dehomogenized'' solution $(z_0,\dots ,z_n) \in (\Cc)^{n+1}$ from the solution $(z_0:{\dots} :z_n) \in \C \mathbb P^n$.
\end{proof}

It is also possible to dehomogenize the problem in a different way and obtain the following Remark. We denote by $\D^\circ \subset \R^{n+1}$ the lattice $n$-simplex $\{\sum_i x_i = 1, x_i\ge 0\}$ and by $\Ds^\circ \subset \Z^{n+1}$ the set of its vertices. Denote by $C_d\subset \R^{n+2}$ the $(n+1)$-dimensional cone containing $d\D^\circ \ast \D^\circ$ such that the simplices $\D^\circ$ and $d \D^\circ$ are its hyperplane section.

\begin{remark}\label{-1dim_rem}
    Consider a Newton non-degenerate degree $d$ homogeneous polynomial $f\in \C[z_0,\dots,z_n]$ with support $\Ps \subset d\D^\circ\subset  \R^{n+1}$ and Newton polytope $P = \conv(\Ps)$.  Although the cone $C_d$ is not a lattice cone, its only non-lattice face is its apex $O_d$, so the polytope $P\ast\D^\circ$ intersects only its lattice faces. Thus, the Newton number $\nu_{\mathsmaller{C_d}}(P \ast \D^\circ)$ is well-defined. Then we have the following formula for the polar degree in Newton non-degenerate case:
    $$\text{Pdeg}(f) = \nu^e_{\mathsmaller{O_d}} (\Ps \ast \Ds^\circ) = \nu_{\mathsmaller{C_d}}(P \ast \D^\circ)$$
\end{remark}

Proof that the latter remark computes the preimages under the gradient map is the same as for Theorem \ref{alg_deg_th}(4). The equivalence to the (usual) Newton number is implied by Lemma \ref{nu_e=nu_lem}.

\begin{figure}[H]
		\begin{center}
	\includegraphics[scale=.3]{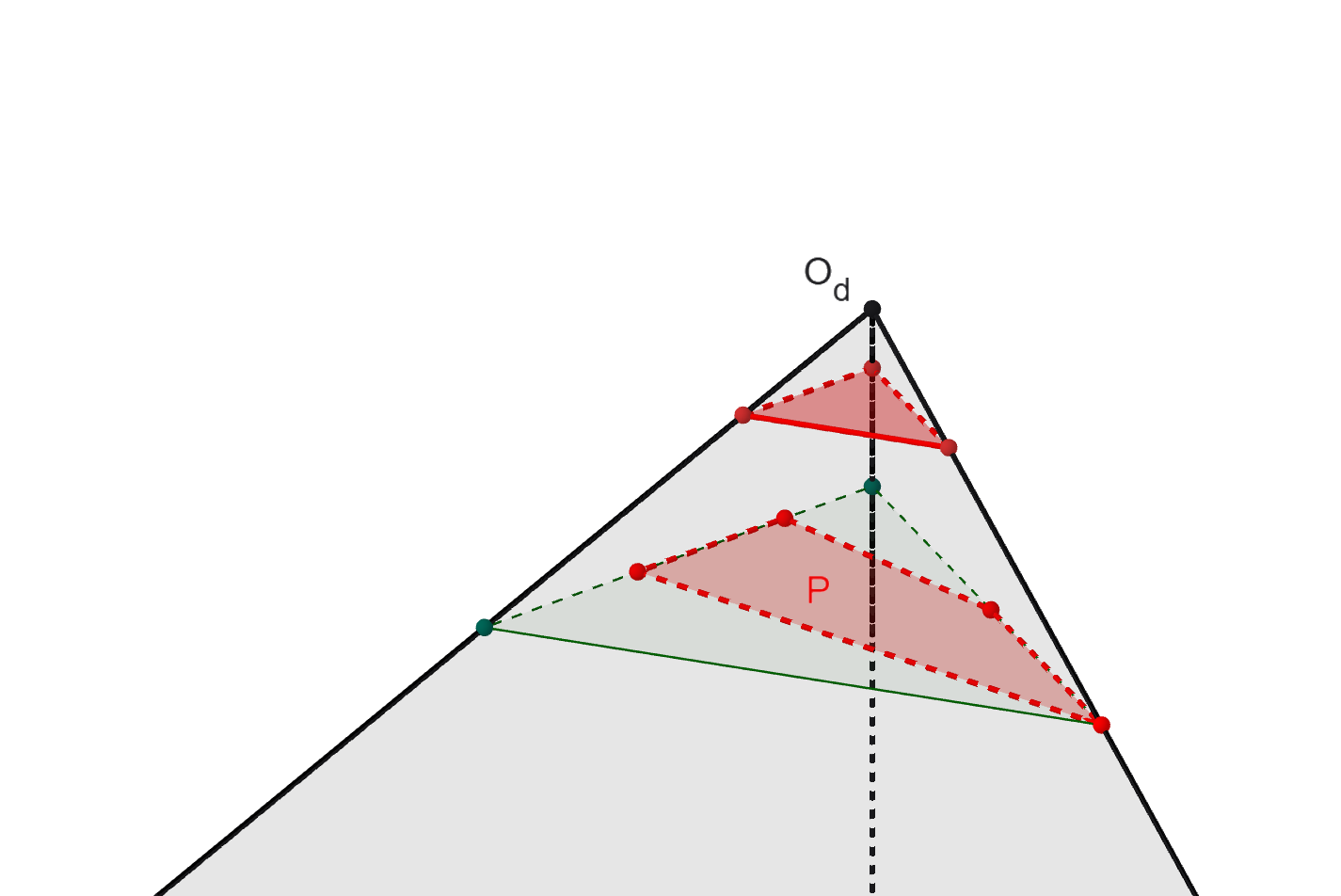}
		\end{center}
		\caption{
		\label{pdeg_fig} The polytope $P$ and the triangle $\D^\circ$ in the cone $C_d$}
\end{figure}

Recall Huh's formula for the degree of the gradient map \cite[Corollary 4.6]{Huh13}. Consider a homogeneous degree $d$ polynomial $f\in \C[z_0,\dots,z_n]$. Denote by $P\subset d\D^\circ$ the Newton polytope of $f$. Define
\[
m_i(P):=\MV_n(\underbrace{\D^\circ,\ldots,\D^\circ}_{n-i},\underbrace{P,\ldots,P}_i) \quad \text{for} \ i=0,\ldots,n.
\]
Define the mixed Newton numbers of $P$ by the formula
\[
\nu_i(P)=V_{n,i}-V_{n-1,i-1}+\cdots+(-1)^i V_{n-i,0} \quad \text{for} \ i=0,\ldots,n,
\]
where $V_{k,l}$ is the sum of the $k$-dimensional mixed volumes $m_l$ of the intersections of $P$ with all possible $(k + 1)$-dimensional coordinate planes. 

\begin{proposition}(\cite[Corollary 4.6]{Huh13})
    Suppose that $f$ is nondegenerate with respect to its Newton polytope. Then the degree of the gradient map $\text{grad} (f): \C\mathbb P^n\to \C \mathbb P^n$ is the sum
    \[
    \text{Pdeg}(f) = \sum\limits_{i=0}^n \nu_i(P)
    \]
\end{proposition}

Note that Remark \ref{-1dim_rem} and Lemma \ref{nu_e=nu_lem} together with the formula for the volume of the Cayley sum Remark \ref{cayley_vol_rem} implies the latter proposition. 

From Lemma \ref{nu_e=mv_lemma} we obtain the following formula for the polar degree via the mixed volume.

\begin{proposition}\label{pdeg=mv_prop}
Denote by $E_0^d, \dots, E_n^d$ the facets of $C_d$. Denote by $D_i$ the polytope $\conv((\Ps \ast \Ds^\circ )\setminus E^d_i)\subset C_d$. Then for a generic homogeneous degree $d$ polynomial with support $\Ps \subset \Z^{n+1}$, we have the following formula for the polar degree:
$$\text{Pdeg} (f) = \MV(D_0, \dots, D_{n})$$
\end{proposition}

\begin{remark}\label{pol_01_rem}
    The most interesting cases are the cases of the polar degree equals zero (vanishing hessian hypersurfaces) or one (homaloidal hypersurfaces corresponding to Cremona transformations); see e.g. \cite{CRS08}. Note that due to the Proposition \ref{pdeg=mv_prop}, the Propositions \ref{mv=0_prop} and \ref{mv=1_prop} provide criteria for these cases (in Newton non-degenerate case).
\end{remark}

\begin{Que}\label{huh_que}
    Can we classify polytopes $P\subset d\D^\circ$, so that $\nu_{C_d}(P \ast \D^\circ)$ is equal to $0$ or $1$?    
\end{Que}

For instance, we have a conjecture concerning Newton non-degenerate homaloidal hypersurfaces. The conjecture holds for the examples \cite[below Example 4.9]{Huh13} and we did not find any counterexamples.

\begin{con}\label{hom_conj}
    If a Newton non-degenerate polynomial $f$ of degree $d \ge 3$ defines a homoloidal hypersurface, then there is a renumbering of the polytopes $D_0,\dots, D_{n}$ from Proposition \ref{pdeg=mv_prop}, such that for each $0\le i \le n$ the quotient $D_{i} / (D_0+ \dots +D_{i-1})$ is a segment of lattice length $1$. 
\end{con}

%\begin{definition}\label{grad_map_def}
%    Consider homogeneous polynomial $f\in \C[z_0,\dots,z_n]$ of degree $d$. To avoid trivialities, we assume that $d \ge 2$. \textbf{Gradient map} of $f$ is the rational map 
%    $$ \grad (f): \C \mathbb P ^n  \dashrightarrow \C \mathbb P^n, \quad \grad(f) = (\frac{\partial f}{\partial z_0}: \dots:  \frac{\partial f}{\partial z_n})$$
%\end{definition}

Note that the polytope $P\ast \D^\circ$ is not convenient in $C_d$, so we cannot define the local $h^*_{\Cp_d}$-polynomial in the usual way. But we can fix it by agglutinating the unit simplex in the following way.

Consider the abstract lattice polyhedral poset $\Pd^\bullet$ defined as the agglutination of the polytope $P\ast \D^\circ$ and the face poset of the unit $(n+1)$-simplex $\D^\bullet$ along the face poset of the $n$-simplex $\D^\circ$.  If we ignore the lattice structure of the polytopes in $\Pd^\bullet$, we see that it is isomorphic to a (non-lattice) subdivision of $\conv(O_d,\D^\circ, P)$ into $P\ast \D^\circ$ and $\conv(O_d,\D^\circ)$, where $O_d$ is the apex of the cone $C_d$. Note that $\conv (O_d,\D^\circ, P)$ is a (non-lattice) convenient polytope in $C_d$. This induces strong formal subdivision with boundary $\sigma: (\Pd^\bullet, \Pd^\bullet_B) \to \Cp^\pm$, where $\Cp$ is the face poset of the cone $C_d$ (or, equivalently, the face poset of $\D^\circ$).

\begin{lemma}\label{grad_deg_lem}
    Formula for the polar degree of non-degenerate polynomials can be rewritten as:
    $$\text{Pdeg}(f) = \nu_{\mathsmaller{C_d}}(P \ast \D^\circ) =\frac{1}{2} \ell^*_{\Cp} (\Pd^\bullet\sharp_{ \Pd^\bullet_B}\sharp \Pd^\bullet;1)$$
\end{lemma}

\begin{proof}
    The poset $\Cp$ is boolean, so the local $h_\Cp^*$-polynomial is the alternating sum of the corresponding $h_\Cp^*$-polynomials. The evaluation of the $h_\Cp^*$-polynomial of polytopes at $t=1$ is the lattice volume. The corresponding alternating sum of volumes for $\D^\bullet$ equals $0$, and for $\D^\circ \ast P$ equals $2 \nu_{\mathsmaller{C_d}}(P \ast \D^\circ)$.
\end{proof}

Hence we have the following corollary.

%We believe that this formula will be useful for Problem \ref{huh_problem}. Denote by $\D \subset C_d$ the convex hull of the vertex of the cone and the simplex $\D^\circ \subset C_d$. Denote by $\Pd^{\ast1}$ the face poset of $P\ast \D^\circ$. Denote by $\Pd^{\ast1\cup\D}$ the face poset of the (non-lattice) polytope $(P\ast \D^\circ)\cup \D$. Note that polytope  $(P\ast \D^\circ)\cup \D$ is strongly convenient in $C_d$, hence we can define the poset with boundary $(\Pd^{\ast1\cup\D},\Pd^{\ast1\cup\D}_B)$. Note that all the boundary faces are lattice, so their local $h$-polynomial is well-defined. The following lemma immediately implies Corollary \ref{gra_deg_cor}.

%\begin{lemma}\label{nu_Cd=ell_nu_lem}
%    We have the following non-negative formula for the degree of the gradient map:
%    $$\nu_{C_d}(P\ast \D^\circ) = \frac{1}{2} (2\ell^*(P\ast\D^\circ;1) + \sum\limits_{{F \in \Pd^{\ast1\cup\D}_B}}\ell^*(F;t) \cdot \ell_{[\sigma(F);C]}(\st_{\Pd^{\ast1\cup\D}\sharp_{\tiny \Pd^{\ast1\cup\D}_B}\Pd^{\ast1\cup\D}}(F);t))$$
%\end{lemma}

\begin{cor}\label{pdeg<ell_cor}
The polar degree of hypersurfaces defined by homogeneous polynomials with Newton polytope $P \subset \R^{n+1}$ and non-degenerate coefficients is bounded by the local $h^*$-polynomial of its Cayley sum with the unit $n$-simplex:
$$\text{Pdeg}(f) \le \ell^* (P\ast \D^\circ;1) $$
\end{cor}

\section{$\mathbf{B_k}$-polytopes}\label{B_k_sec}

\subsection{Preliminaries: Dual defective sets} \label{dual_rec_subsec}

The set $\As \subset \Z^n$ is called \textbf{dual defective} if the projective dual set to the corresponding projectively toric variety $X_{\As} \subset \C P^{|\As|}$ is not a hypersurface (see e.g. \cite{FI21}). The projective dual set to $X_\As$ is known as the Gelfand-Kapranov-Zelevinsky $\As$-discriminant surface (see \cite[Proposition 1.1 from Chapter 9]{GKZ94}). Equivalently, the set $\As$ is dual defective if and only if the $\As$-discriminant is equal to $1$. 

In this paper we need the Furukawa-Ito classification of dual defective sets from \cite{FI21}, so we recall it in this subsection. First, let us briefly recall all the known combinatorial descriptions of dual defective sets (see \cite{CD22} for the overview).

There are four descriptions combinatorial of dual defective sets in terms of:
\begin{enumerate}
	\item ranks of certain matrices associated to the vectors in the kernel of $\As$ \cite{DFS07}.
	\item a matroid associated to the Gale transform of $\As$ \cite{CC07}.
	\item existence of (iterated) circuits in $\As$ \cite{Es18}.
	\item existence of certain Cayley decomposition of $\As$ called Furukawa-Ito decomposition \cite{FI21} (see Theorem \ref{dual_defect_classification_th} in our paper).
\end{enumerate} 

These combinatorial descriptions are equivalent since they all describe the dual defective sets.  In \cite{CD22} all these approaches are compared from the combinatorial point of view. It is shown that the items 1-3 are equivalent using linear-algebraic methods which require only the combinatorial definitions. However for the Furukawa-Ito classification (combinatorially) it is only proved that if a set has Furukawa-Ito decomposition then it also satisfies items 1-3, but not vise versa (see \cite[Remark 3]{CD22}).

\begin{Que}
	Can we prove combinatorially (using linear-algebraic methods) that descriptions from \cite{DFS07}, \cite{CC07}, \cite{Es18} (items 1-3) of dual defective sets imply existence of Furukawa-Ito decomposition (item 4)?
\end{Que}

The latter question appears to be surprisingly non-trivial. For instance, our proof of $B_k$-theorem \ref{neg_class_th} significantly relies on the Furukawa-Ito classification of dual defective sets and we do not know a proof using other combinatorial descriptions of dual defective sets.

Recall that $\C[\As]$ is the vector space of Laurent polynomials with supports contained in the finite lattice set $\As$.

\begin{remark}
The original text of the following Lemma contains some typos, we fixed them.
\end{remark}

\begin{lemma} \cite[Lemma 2.12]{Es10} \label{dual=no_crit_points_lem} 
    Set $\As$ containing the origin $O$ is dual defective if and only if $\nu^e_O(\As) = 0$ or, equivalently, generic polynomials in $\C[\As]$ have no singular points in the torus $(\C\setminus 0)^n$.
\end{lemma}

Let us recall the classification of dual defective sets in terms of Cayley sums from \cite{FI21}. Recall that for $\As \subset \Z^n$ we denote by $\langle \As - \As \rangle $ the subgroup of $\Z^n$ generated by the set $\{a_1 - a_2:  a_1,a_2 \in \As\}$ and say that $\As \subset \Z^n$ is \textbf{spanning}, if $\langle \As - \As \rangle = \Z^n$.

\begin{definition}
    Let $\As_0,\dots, \As_k \subset \Z^n$ be finite sets. We say that the Cayley sum $\As_0 \ast \dots \ast \As_k$ is of join type if the homomorphism $\langle \As_0 - \As_0 \rangle \oplus \dots \oplus \langle \As_k - \As_k \rangle \to  \langle \As_0 - \As_0 \rangle + \dots + \langle \As_k - \As_k \rangle \subset \Z^n$  given by $(a_0, \dots, a_k) \to a_0 + \dots  + a_k$ is injective.
\end{definition}

\begin{theorem}(Furukawa-Ito \cite{FI21})\label{dual_defect_classification_th}
    Spanning set $\As$ is dual defective if and only if  there exist natural numbers $c < r$ and a lattice projection $\pi: \Z^n \to  \Z^{n-c}$ such that $\pi(\As) \cong \mathbf B_0\ast \dots \ast \mathbf B_r$ where this Cayley sum $\mathbf B_0\ast \dots \ast \mathbf B_r$ is of join type and $\mathbf B_i \ne \emptyset$ for all $0 \le i \le r$.
\end{theorem}

\begin{remark}\label{dual_defect_classification_rem}
    If the set $\As$ is not spanning, then the same theorem works, but instead of $\Z^n$ we should consider the lattice span by $\langle \As - \As \rangle$ (and shift the set $\As$ in this lattice if needed).
\end{remark}

\subsection{Preliminaries: Thin polytopes}\label{thin_rec_subsec}

In this subsection we recall what is known about thin polytopes. Our $B_k$-polytopes are thin and represent a non-trivial and wide class of thin polytopes.

\begin{definition}\label{thin_def}
    Polytope $P$ is called \textbf{thin} if its local $h^*$-polynomial vanishes $\ell^*(P;t) = 0$.
\end{definition}

Thin simplices were defined in \cite[Chapter 11.4.B]{GKZ94}, where they were classified up to dimension $2$. In \cite{GKZ94} they define thin simplices through the Newton number (here it is the alternating sum of volumes of all faces of the polytope including the empty face). Note that this Newton number has nothing in common with the local $h^*$-polynomial for arbitrary polytopes, for instance, it can be negative (it equals $-1$ for the cube $[0,1]^3$), hence such definition cannot be extended to arbitrary polytopes.

The study of non-simplicial thin polytopes began in \cite{BKN24}, where the $3$-dimensional classification is obtained. In \cite{Kur24a} the classification of thin simplices in dimension $4$ is obtained.

%We note that the $B_k$-polytopes are thin which provides non-trivial examples of thin polytopes. For instance, they provide a natural class of counterexamples to \cite[Question 3.16]{BKN24} (note that some counterexample is given in \cite[Corollary 2.]{Kur24b}). Let us recall some of the results and questions concerning the thin polytopes.

Recall that the \textbf{degree} $\deg(P)$ of a polytope $P$ is the degree of its $h^*$-polynomial. The \textbf{codegree} $\text{codeg}(P)$ is the smallest integer $\lambda \ge 1$ such that the $\lambda$-th dilate $\lambda P$ contains a lattice point in its relative interior. By convention, a point has codegree $1$. It follows from Ehrhart-MacDonald reciprocity \cite{Mcd71} that

$$\deg(P) + \text{codeg}(P)  = \dim(P) + 1$$

\begin{definition}
A \textbf{Lawrence prism} is a $d$-dimensional lattice polytope in $\R^d$ isomorphic to $\conv(0,e_1,\ldots, e_{d-1}, k_0 e_d, e_1+k_1 e_d, \ldots, e_{d-1}+k_{d-1} e_d)$ with $k_0, k_1, \ldots, k_{d-1} \in \Z_{\ge 1}$, i.e. it is the Cayley sum of lattice segments in $\R^1$.
\end{definition}

Recall that \textbf{lattice pyramid} over a polytope $P$ is a pyramid with base $P$ and height $1$. We say that two polytopes are isomorphic if there is a lattice bijection between a sequence of lattice pyramids over one of them and the other.

\begin{theorem}\cite{Bat05}
The number of non-isomorphic polytopes with given $h^*$-polynomial is finite.
\end{theorem}

\begin{theorem}\cite{BN07}
  Any lattice polytope of degree $1$ is isomorphic to a Lawrence prism or to $2 \Delta_2$ (the triangle with vertices $(0,0), (0,2), (2,0)$).
\label{thm:BN}  
\end{theorem}

Other polytopes whose $h^*$-polynomial is binomial are simplices by Remark \ref{h=vol_rem} (2). They have an interpretation in terms of coding theory, see \cite{BH13}. The classification of degree two polytopes is an open problem. 

\begin{theorem} \cite[Theorem 4.3]{BKN24}
Let $P$ be a $3$-dimensional lattice polytope. Then $P$ is thin if and only if up to isomorphism
	\begin{itemize}
				\item $P$ is a lattice pyramid over a lattice polygon, or
				\item $P$ is a Lawrence prism.
	\end{itemize}
	\label{thm:3d}
\end{theorem}

\begin{remark}
	For $k\ge 2$ every $B_k$-polytope in $\R^{k+1}$ is either a lattice pyramid or a Lawrence prism. Moreover, every Lawrence prism in dimension $k+1 \ge 3$ is isomorphic to such a $B_k$-polytope. In dimension $3$ all thin polytopes are equivalent to $B_k$-polytopes (for $k=1$ or $2$). In dimension $4$ there are already other examples of thin polytopes (including simplices which are not Cayley sums at all, see \cite{Kur24a}).
	So $B_k$-polytopes can be considered as a generalization of lattice pyramids and Lawrence prisms. 
\end{remark}

In dimension $4$ already the classification of thin simplices is sufficiently more complicated and contains ``sporadic'' cases (see \cite{Kur24a}).

\begin{definition}\label{triv_thin_def}
    A polytope $P$ is called \textbf{trivially thin} if $\dim(P) \ge 2 \deg (P)$.
\end{definition}

The following proposition is implied by the symmetry property of the local $h^*$-polynomial.

\begin{proposition}\label{triv_thin_prop}
    Trivially thin polytopes are indeed thin.
\end{proposition}

\begin{definition}
A lattice polytope $P$ is called the \textbf{join} of two non-empty faces $Q_1,Q_2$ if $$\langle Q_1-Q_1\rangle \cap \langle Q_2-Q_2\rangle = 0, \quad Q_1\cap Q_2 = \emptyset \quad \text{and} \quad \conv(Q_1,Q_2) = P$$\\
Let $Q_1 \subset \R^n$ and $Q_2 \subset \R^m$ be lattice polytopes. We call 
\[Q_1 \circ_\Z Q_2 := \conv(Q_1 \times \{0\} \times \{0\}, \{0\} \times Q_2 \times \{1\}) \subset \R^n \times \R^m \times \R,\]
the \textbf{free join} of $Q_1$ and $Q_2$. 
\end{definition}

\begin{proposition} \label{prop:mult} (see e.g. \cite[Proposition 3.14]{BKN24})
Let $P \subset \R^n$ and $Q \subset \R^m$ be lattice polytopes. Then \[h^*_{P \circ_\Z Q}(t) = h^*_P(t) \, h^*_Q(t),\; \text{ and }\; \ell^*_{P \circ_\Z Q}(t) = \ell^*_P(t) \, \ell^*_Q(t).\]
\end{proposition}

\begin{cor}
The free join of two lattice polytopes is thin if and only if at least one of the two factors is thin.
\end{cor}

A question-conjecture on the classification of thin polytopes from \cite{BKN24} is as follows.

\begin{Que} (already answered in \cite[Corollary 2]{Kur24b}) \cite[Question 3.16]{BKN24}\label{question}
    \begin{enumerate}
    \item Is every thin polytope trivially thin or a join?
    \item Is every spanning thin polytope trivially thin or a free join?
    \end{enumerate}
\end{Que}

Note that in dimension $3$ every thin polytope is either a lattice pyramid or trivially thin. In \cite[Corollary 2]{Kur24b}  (using so-called Lawrence twist) there is obtained an example of spanning thin polytope in dimension $5$, which is not a join. We recall it in Example \ref{kur_thin_ex} and show that it is a $B_2$-polytope. An up-to-date question concerning the classification is as follows.

\begin{Que}\cite[Question 4.]{Kur24b}
    Can we classify thin polytopes that are not free joins?
\end{Que}

%Recall that the width of a lattice polytope is the minimal lattice distance (number of parallel lattice hyperplanes inside plus $1$) between parallel hyperplanes, so that the polytope is contained between them. 
%Let us also mention another up-to-date question concerning thin polytopes (though our technique provides no insight concerning it since the width is equal to one if and only if the polytope is a non-trivial Cayley sum).

Another up-to-date question is as follows.

\begin{Que} (\cite[Question 3.19]{BKN24} and \cite[Conjectures 6, 7]{Kur24b})
    What is the possible width of thin polytopes?
\end{Que}

%In our paper we provide counterexamples to both questions (see Example \ref{Nill_counteex}) called $B_k$-polytopes. Recall that there is an example \cite[Example 3.17]{BKN24} of thin $4$-simplex $\D^4$ with width $2$ (also note that the ($\ell$-)Newton number of $\D^4$ in the cone over it is nonzero and considering another vertex instead of $0$ will not change that). Our $B_k$-polytopes are of width $1$, so our examples say nothing about the possible width of thin polytopes and provide no insights concerning \cite[Question 3.19.]{BKN24}. 

Let us mention a similar classification problem for polyhedral subdivisions.

\begin{definition}
    A polyhedral subdivision of a polytope (or a cone) is called \textbf{thin} if its local $h$-polynomial vanishes.
\end{definition}

Note that the local $h$-polynomial of subdivision of an $n$-polytope coincides with the local $h$-polynomial of the corresponding subdivision of the $(n+1)$-cone over the polytope.

\begin{Que}
    Can we classify thin polyhedral subdivisions of polytopes (or cones)?
\end{Que}

The question is not answered even for triangulations of simplices. The answer is known for $2$- and $3$-dimensional simplices (see \cite{dMGPSS20}). For more results concerning thin triangulations of simplices (also called triangulations with vanishing local $h$-polynomials), see \cite{LPS23a} and \cite{LPS23b}. 

The classification problem is significantly motivated by the monodromy conjecture (see e.g. \cite{LPS23a} and \cite{Ve}). It is also important for Arnold's monotonicity problem. Namely, the classification of thin triangulations of the cones $\R^3_{\ge 0} \oplus \R^1$ and $\R^2_{\ge 0} \oplus \R^2$, conjectured in \cite[Conjecture 2.28]{Sel24}, would lead to the complete solution of Arnold's monotonicity problem in dimension $5$ (see \cite[Theorem 3.6]{Sel24} and the remark below it). We believe that the same motivation can be redirected to the problem of classification of thin polytopes.

\begin{remark}
    Katz-Stapledon decompositions formulas imply that any lattice polyhedral subdivision of any thin polytope is also thin.
\end{remark}

We believe that polyhedral subdivisions of $B_k$-polytopes provide non-trivial examples of thin polyhedral subdivisions.

\subsection{Properties of $\mathbf{B_k}$-polytopes}

By the $B_k$-Theorem \ref{neg_class_th} any $B_k$-polytope in $\R^n$ based on $\R^k_{\ge 0}$ is negligible in the cone $\R^k_{\ge 0}\oplus \R^{n-k}$. By Theorem \ref{loc_h*_ag_th} the negligible polytopes in $\R^k_{\ge 0}\oplus \R^{n-k}$ are thin. Thus $B_k$-polytopes are thin. Let us establish some properties of $B_k$-polytopes showing that they are non-trivial examples of thin polytopes. We compare $B_k$- and $B$-polytopes from \cite[Definition 1.4]{ELT22} in \S \ref{comp_B_Bk_subsubsec}. 

\begin{proposition}
    The codegree of the Cayley sum of polytopes satisfies the following:
    \begin{enumerate}
        \item $\text{codeg} (P_0\ast P_1\dots \ast P_k) \le k+1$ (equivalently, $\deg P \ge n-k$). 
        \item $\text{codeg} (P_0\ast P_1\dots \ast P_k) = k+1$ if and only if $P_0 + \dots + P_k$ contains an interior lattice point.
    \end{enumerate}
\end{proposition}

\begin{cor}\label{not_triv_thin_cor}
   A $B_k$-polytope is generally not trivially thin if $n > 2k$.
\end{cor}

\begin{remark}
    Using the language of \cite{Nil20}, the second part of the proposition claims that $\text{codeg} (P_0\ast P_1\dots \ast P_k) = k+1$ if and only if we have the following inequality for the mixed codegree $\text{mcd} (P_0,\dots,P_k) \le k+1$, i.e. $\text{mcd} (P_0,\dots,P_k) \ne k+2$.
\end{remark}

\begin{proof}
Note that the projection of $\lambda P$ onto $\R^k$ contains no interior lattice points if $\lambda \le k$, which proves the first part of the proposition. The projection of $(k+1)P$ onto $\R^k$ contains a unique interior lattice point $(1,\dots,1)$. The fiber over it is exactly the polytope $P_0+ \dots+P_k$ and this proves the second part of the proposition.
\end{proof}

\begin{proposition}
    Consider a $B_k$-polytope $P = P_0\ast P_1\dots \ast P_k$ in $\R^n$. Suppose that for all $i$ $\dim (P_i) \ge 1$ and $P_0$ is not a join. Then $P$ is also not a join.
\end{proposition}

\begin{cor}\label{not_join_cor}
    A $B_k$-polytope in $\R^n$ is generally not a join if $k\ge 2$ and $n-k \ge 2$.
\end{cor}

\begin{proof}
    Assume that $P$ is a join of $Q_1$ and $Q_2$. Since $P_0$ is not a join, then it is contained either in $Q_1$ or in $Q_2$. Without loss of generality, assume that $P_0 \subset Q_1$. Then, since $\text{aff} (P_0) = \R^{n-k}$, each $P_i$, $i\ge 1$, is either contained in $Q_1$ or contained in $Q_2$. Some of them must be contained in $Q_2$ since $Q_2 \ne \emptyset$. Recall that $\langle P_i - P_i \rangle \subset \text{aff} (P)$, so it is not a join.
\end{proof}

Thus, the smallest case, provided by the latter proposition, when a $B_k$-polytope is not trivially thin and not a join is the case $n=5,k=2$. Here is such an example. We also confirmed it using the SageMath code from \cite{Kre}.

\begin{ex}\label{Nill_counteex} Suppose that $P \subset \R^5$ is the Cayley sum of two parallel segments of length $1$ and a polytope $P_0 \subset \R^3$ which is not a join and contains an interior lattice point $(1,1,1)$. For example, it is the convex hull of the following set of points (where $\{e_i\}$ is the standard basis of $\Z^5$):

\begin{table}[H]
    \centering
    \begin{tabular}{c|c|c|c|c|c|c|c|c}
        $A_1$ & $B_1$ & $A_2$ & $B_2$ & $A$ & $B$ & $C$ & $D$ & $E$ \\
         $e_1$ & $e_1 + e_3$ & $e_2$ & $e_2+e_3$ & $0$ & $3e_3$ & $3e_4$ & $3e_5$ & $e_3+e_4+e_5$
    \end{tabular}
    \label{12}
\end{table}

\begin{figure}[H]
		\begin{center}
	\includegraphics[scale=7]{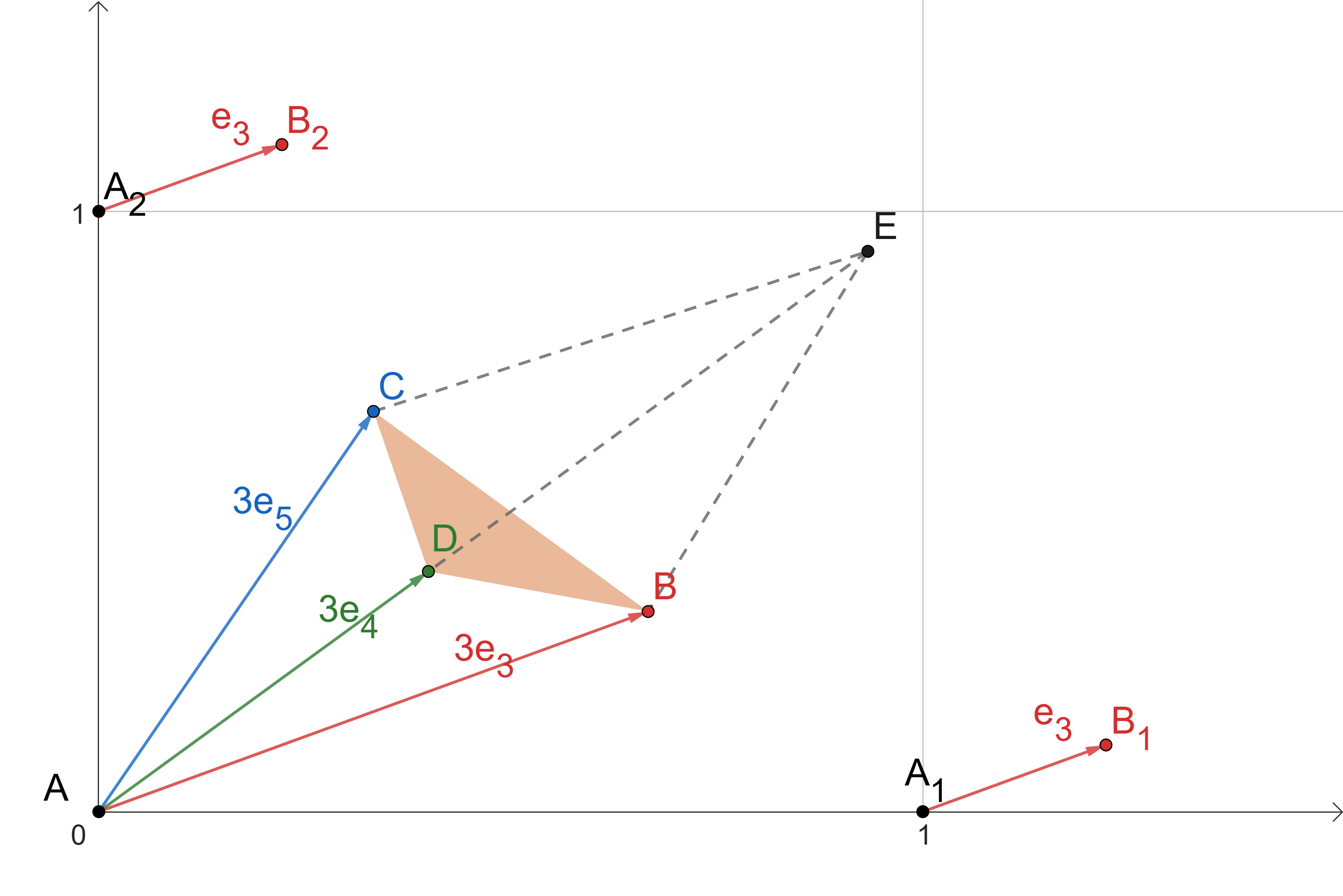}
		\end{center}
		\caption{
		\label{Thin_example} Example of a thin polytope which is neither trivially thin nor a join}
\end{figure}

From Corollary \ref{nu>=ell_cor} we obtain that this polytope is thin. By Corollaries \ref{not_triv_thin_cor} and \ref{not_join_cor} this polytope is neither trivially thin nor a join. 
    
\end{ex}

\begin{ex} \label{kur_thin_ex}
    Another example of a non-trivially thin (non-join type) 5-dimensional $B_2$-polytope is given in the PhD thesis \cite[Corollary 2]{Kur24b}. Let us show that it is also a $B_2$-polytope. The polytope $P$ is the convex hull of the following points.
    $$
    \begin{pmatrix*}[r]
        0 & 2 & 0 & 0 & 0 & -2 & 0 & -1\\
        0 & 0 & 1 & 0 & 0 & 3 & 0 & 0\\
        0 & 0 & 0 & 1 & 0 & -1 & 0 & 0\\
        0 & 0 & 0 & 0 & 1 & -1 & 0 & 0\\
        0 & 0 & 0 & 0 & 0 & 0 & 1 & 1\\
    \end{pmatrix*}
    $$
Consider the projection $\pi_2$ along the 3-space generated by $e_1,e_2,e_3$ to the 2-space generated by $e_5$ and $e_2+e_3+e_4$. The projection maps $P$ to the unit $2$-simplex. Thus, $P$ is isomorphic to the Cayley sum of two segments and a tetrahedron with vertices:

$$
\begin{pmatrix*}[r]
0 & 2\\
0 & 0\\
0& 0
\end{pmatrix*}\quad
\begin{pmatrix*}[r]
0 & -1\\
0 & 0\\
0& 0
\end{pmatrix*}\quad
\begin{pmatrix*}[r]
0 & 0 & 0 & -2\\
0 & 1 & 0 & 3\\
0 & 0 & 1 & -1
\end{pmatrix*} 
$$

Their sum contains the interior lattice point $(0,1,0)$; thus, by Corollary \ref{not_triv_thin_cor}, the polytope $P$ is not trivially thin. By Corollary \ref{not_join_cor} it is not a join.

\end{ex}

\subsubsection{Comparison of $\mathbf{B_k}$-polytopes and $\mathbf B$-polytopes} \label{comp_B_Bk_subsubsec}

Let us compare $B_k$- and $B$-polytopes from \cite[Definition 1.4]{ELT22}. Any $B_k$-polytope is a $B$-polytope, but not vice versa. First let us recall the definition of $B$-polytopes.

\begin{definition}\label{B_pol_def}  \begin{enumerate}
        \item A lattice simplex in $\R_{\ge 0}^m \oplus \R^{n-m}$ with the standard coordinate system $v_1,\dots,v_m, w_{m+1}, \dots , w_{n}$ is called a $\mathbf{B_1}$\textbf{-simplex} with respect to the $i$-th coordinate if one of its vertices lies in the plane $v_i = 1$ and all other vertices lie in the plane $v_i = 0$.
        \item A lattice polytope in $\R_{\ge 0}^m \oplus \R^{n-m}$ is called a $\mathbf{B}$\textbf{-polytope}, if every lattice simplex that it contains is a $B_1$-simplex.
    \end{enumerate}
\end{definition}

The following remarks admit straightforward verification.

\begin{remark}\label{Bk_is_B_rem}
Every $B_k$-polytope in $\R_{\ge 0}^m \oplus \R^{n-m}$ is a $B$-polytope.
\end{remark}

\begin{remark}\label{B=Bk_rem}
A convenient polytope $P \subset \R_{\ge 0}^m \oplus \R^{n-m}$ with $n \le 3$ is a $B$-polytope if and only if it is a $B_k$-polytope for some $k \le m$.
\end{remark}

Starting from dimension $4$ the $B_k$- and $B$-polytopes are different.

\begin{ex}\label{B_not_Bk_ex}

The convex hull of the following set of points in dimension $4$ is a $B$-polytope but not a $B_k$-polytope (in $\R^4_{\ge 0}$).

\begin{table}[H]
    \centering
    \begin{tabular}{c|c|c|c|c|c|c}
        $O$ & $A$ & $B$ & $C$ & $F$ & $G$ & $H$  \\
         $0$ & $e_1$ & $e_2$ & $e_1+e_2$ & $e_3$ & $e_4$ & $e_3+e_4$
    \end{tabular}
    \label{12}
\end{table}

\begin{figure}[H]
		\begin{center}
	\includegraphics[scale=1]{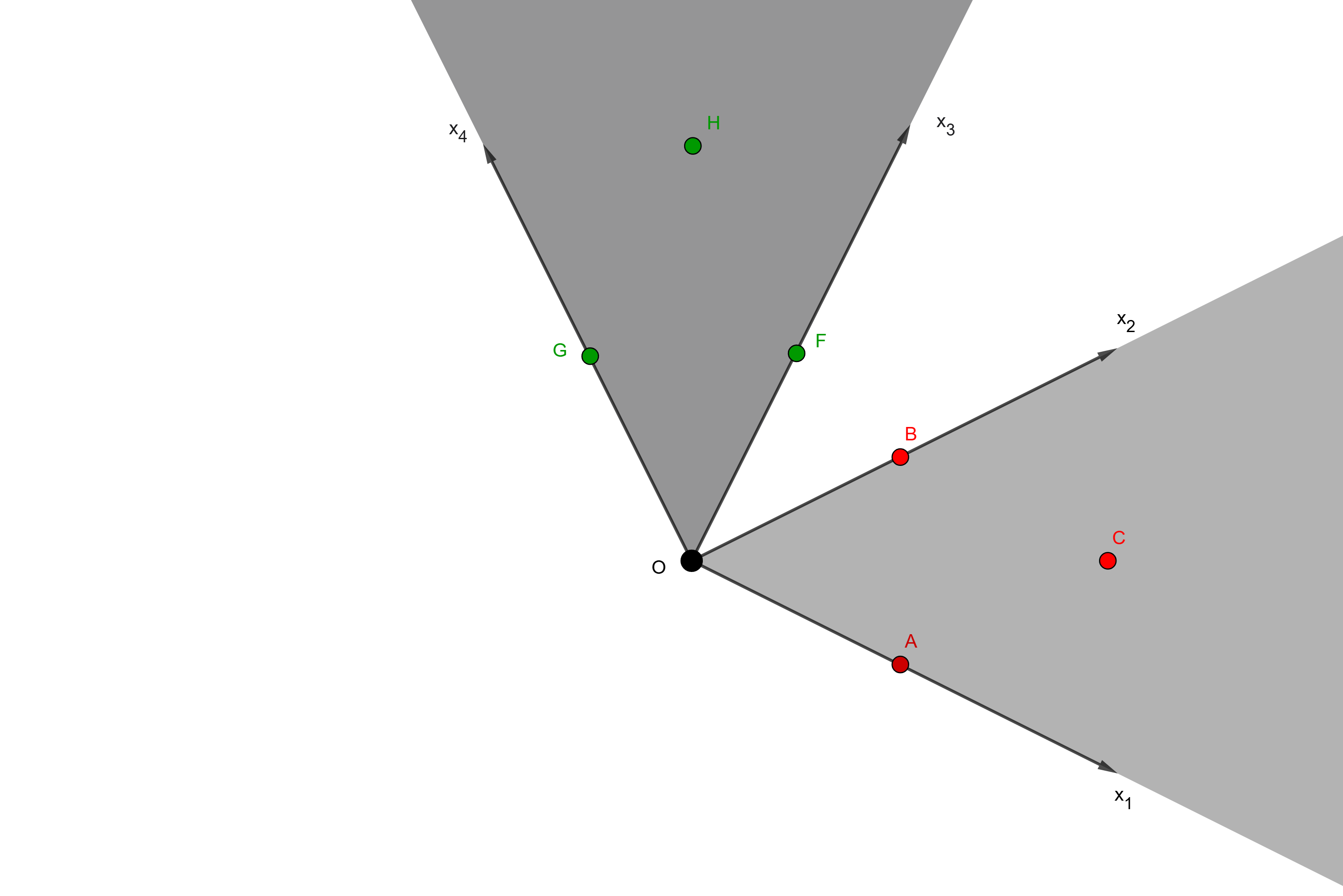}
		\end{center}
		\caption{
		\label{B_pol_ex} Example of a $B$-polytope which is not a $B_k$-polytope}
\end{figure}

\end{ex}

\subsection{Proof of the $\mathbf{B_k}$-Theorem}

In this section we prove $B_k$-Theorem \ref{neg_class_th}.

We split the proof into two parts. First, we check that $B_k$-polytopes are indeed negligible and this is the easy part. Second, using the Furukawa-Ito classification of dual defective sets, we prove that every negligible polytope is a $B_k$-polytope.

\subsubsection{$\mathbf{B_k}$-polytopes are negligible}

\begin{lemma}(Simple part of $B_k$-theorem)\label{B_k-neg_lem}
    Consider the cone $C = \R^m_{\ge 0} \oplus \R^{n-m}$. Then any $B_k$-polytope $P\subset \R^m_{\ge 0} \oplus \R^{n-m}$ based on a coordinate subspace $\R^k_{\ge 0} \subset \R^m_{\ge 0}$ is negligible ($\nu_C(P) = 0$).
\end{lemma}

\begin{proof}

First, let us prove the lemma for $B_k$-polytopes in $C = \R^k_{\ge 0} \oplus \R^{n-k}$. Applying Lemma \ref{cayley_signed_vol_lem} to $P_0 \ast P_1 \ast \cdots \ast  P_k$ and $P_1 \ast \cdots \ast  P_k$ we obtain:

$$\nu_{Q} (P_0 \ast P_1 \ast \cdots \ast  P_k) = \sum\limits_{\substack{a_0 \ge 0, a_1> 0, \dots , a_k > 0 \\ a_0+\ldots + a_k = n-k}} \MV(\underbrace{P_0,\dots,P_0}_{a_0}, \dots, \underbrace{P_k,\dots,P_k}_{a_k})$$

Recall that by the definition of $B_k$-polytopes, we have $\dim(P_1 + \dots + P_k) < k$; by Proposition \ref{mv=0_prop} each summand in the above sum equals $0$. Thus, the Newton number vanishes:
$$\nu_Q(P_0 \ast P_1 \ast \cdots \ast  P_k) = 0$$

Now, we consider the general case  $C = \R^m_{\ge 0} \oplus \R^{n-m}$ with $m \ge k$. Let us decompose the cone $C$ as follows: $C = \R^k_{\ge 0} \oplus \R^{m-k}_{\ge 0} \oplus \R^{n-m}$. To avoid confusion, denote the first summand by $L_k = \R^k_{\ge 0}$ and the last summand by $L_{n-m} = \R^{n-m}$. For $J \subset \{k+1,\dots,m\}$, denote by $E_J\subset \R^n$ the vector space generated by the corresponding coordinate vectors. Denote by $C_J \subset \R^n$ the cone $L_k +E_J + L_{n-m}$, which has dimension $\dim C_J = n + k - m + |J|$. Then we have:

$$\nu_C(P_0 \ast P_1 \ast \cdots \ast  P_k) = \sum\limits_{J \subset \{k+1,\dots,m\}} (-1)^{m-k-|J|} \nu_{\mathsmaller {C_J}} ((P_0 \ast P_1 \ast \cdots \ast  P_k )\cap C_J)$$

Note that each summand is a $B_k$-polytope in the cone $\R^k_{\ge 0}\oplus \R^{n-m+|J|}$ based on $\R^k_{\ge 0}$. Hence, all these summands are equal to $0$, as established at the beginning of the proof.
    
\end{proof}

\subsubsection{Negligible polytopes are $\mathbf{B_k}$-polytopes}

Consider a negligible polytope $P \subset C = \R^m_{\ge 0}\oplus \R^{n-m}$. Denote by $\Ps \subset \Z^m_{\ge 0} \oplus \Z^{n-m}$ a lattice set such that $\conv(\Ps) = P$ (e.g., the set of vertices of $P$ or $P \cap \Z^n$).

The idea of the proof is to consider the Newton number in $C$ as the $e$-Newton number and apply the Furukawa-Ito classification theorem of dual defective sets. The Furukawa-Ito theorem provides that our set is a Cayley sum with certain properties, but we need the Cayley sum to be a ``coordinate'' Cayley sum, and verifying it is the non-trivial part of the proof.

%\begin{lemma}(Complicated part of the $B_k$-theorem in spanning case)\label{b_k_span_lem}
%Consider cone $C = \R^m_{\ge 0} \oplus \R^{n-m}$ and a spanning lattice set $\Ps \subset \Z^m_{\ge 0} \oplus \Z^{n-m}$. Suppose that $P = \conv(\Ps)$ is a negligible polytope $C$. Then $P$ is a $B_k$-polytope based on a face of $C$.
%\end{lemma}

First, we need the following proposition.

\begin{proposition}\label{no_cay_prop}
    Consider a full-dimensional lattice set in $\Ps \subset \Z^r$. Then for a generic point $X \in \Z^r$ the set $\Ps \cup X$ has width greater than $1$ (i.e. is not a non-trivial Cayley sum).
\end{proposition}

\begin{proof}
    Consider a subset $\Ds \subset \Ps$ consisting of the vertices of an $r$-dimensional simplex. For a pair of opposite faces $F_1 \sqcup F_2 = \Ds$, consider the stripe parallel to $\text{aff} (F_1-F_2)$ such that $F_1$ and $F_2$ are on its boundary. Then for any point $X$ outside these $2^r$ stripes, the union $\Ps \cup X$ has width greater than $1$.
\end{proof}

\begin{remark}
    A similar Proposition \ref{no_cay_prop} can be proved for arbitrary large width.
\end{remark}

Consider a generic lattice point $X \in \Z^{n-m}$. By Lemma \ref{nu>nu_e_lem}, we have the inequality $\nu^e_X(\Ps) \ge \nu_C(P)$, so $\nu^e_X(\Ps) = 0$. Note that $\nu^e_X(\Ps) = \nu^e_X(\Ps \cup X)$ (it is implied e.g. by Remark \ref{dim0crit_lem}). Denote by $\As \supset O$ the set $(\Ps \cup X)-X$, i.e. the set $\Ps \cup X$ shifted by the vector $-X$. By Lemma \ref{dual=no_crit_points_lem} it is dual defective.

%Consider negligible polytope $P \subset \R^n_{\ge 0}$. Recall that $\nu(P)$ is the number of critical points of generic polynomial with Newton polyhedron in $\C^n$. The polytope $P$ is negligible, hence generic polynomial with Newton polytope $P$ has no critical points in $\C^n$ and hence no critical points in the torus $(\C \setminus 0)^n$. Thus, by Lemma \ref{dual=no_crit_points_lem} the set $\Ps = P \cap \Z^n$ is dual defective. Note that $\Ps$ is spanning, since it contains the vertices of the unit simplex (and similar argument does not work for the cones $\R^m_{\ge 0}\oplus\R^{n-m}, n>m$). 

Denote by $\Z_\As^n$ the minimal affine sublattice of $\Z^n$ containing $\As$.
From the Furukawa-Ito classification Theorem \ref{dual_defect_classification_th} (and Remark \ref{dual_defect_classification_rem}), we obtain that there are integers $k>c$, and two projections $\pi_k: \Z_\As^n \to \Z^k$ and $p: \Z_\As^n/\Z^k \to \Z^{n-k-c}$ such that:
\begin{enumerate}
    \item $\pi_k (\As) = Conv (O,q_1,\dots,q_k)$ is the unit simplex in $\Z^k$.
    \item Define $\As_0 = (\pi_k^{-1} (O) \cap \As)/\Z^k$ and $\As_i = (\pi_k^{-1} (q_i) \cap \As)/\Z^k$. Then  $p(\As_0) \ast \dots \ast p(\As_k)$ is of join type.
\end{enumerate}

Without loss of generality, we can assume that $\pi_k (O) = O$, where $O$ is the origin (otherwise we can compose $\pi$ with an appropriate affine automorphism of $\Z^k$). Consider the affine extension of $\pi_k$ on $\R^n$ and denote it also by $\pi_k$. Note that the image of the lattice $\pi_k(\Z^n)$ may not be contained in the lattice $\Z^k$.

\begin{remark}\label{to_O_rem}
    By Proposition \ref{no_cay_prop}, we have $\pi_k(\R^{n-m}) = O$.
\end{remark}

\begin{proposition}\label{pk_map_type_prop}
    There are positive integers $a_i^{(j)} \in \Z_{>0}$ and a renumbering of the standard basis $e_i^{(j)}$ of $\Z^n$ such that the map $\pi_k$ can be written as follows:
    $$
\pi_k:
\begin{cases}
e_0^{(1)}, \dots , e_0^{(j_0)} &  \to 0\\
a_1^{(1)} e_1^{(1)}, \dots , a_1^{(j_1)} e_1^{(j_1)} &  \to q_1\\
 \ \  \ \ \ \ \  \dots  & \dots \\
a_k^{(1)} e_k^{(1)}, \dots , a_k^{(j_k)} e_k^{(j_k)} &  \to q_k
\end{cases} $$

\end{proposition}

\begin{proof}
    By Remark \ref{to_O_rem}, we have $\pi_k(\R^{n-m}) = O$, so the first $n-m$ elements $e_0^{(1)}, \dots, e_0^{(n-m)}$ are the standard basis of $\Z^{n-m}$. For each standard basis vector $v_i$ of $\Z^m_{\ge 0}$, there exists $a_i \in \Z_{>0}$ and a point in $\As$ whose projection along $\Z^{n-m}$ is $a_i v_i$, since $\conv(\As)$ is convenient in $C$. So $\pi_k(a_i v_i)$ is either the origin $O$ or a basis vector $q_j$. Considering this way all the coordinate rays of $\Z^{m}_{\ge 0}$, we obtain the required form.
\end{proof}

Denote by $\pi_K^O : C \to \R^{k+r}_{\ge 0}$ the projection along the vector subspace generated by $e_0^{(1)}, \dots , e_0^{(j_0)}$ (here $r = n-k-j_0$).

\begin{remark}
    The projection $\pi_k^O(\As)$ contains exactly one point on each coordinate ray of $\Z^{k+r}_{\ge 0}$ apart from the origin, and its coordinate is the corresponding $a_i^{(j)}$.
\end{remark}

First, let us prove the $B_k$-theorem if all these points have unit coordinates.

\begin{lemma}\label{unit_bk_proof_lem}
    Suppose that $\pi_k^O(\As)$ is the set of vertices of the unit simplex $\Ds_{k+r}$ (equivalently, all $a_i^{(j)} = 1$). Then the polytope $\conv (\As)$ is a $B_k$-polytope based on some $\R^k_{\ge 0} \subset\R^{k+r}_{\ge 0}$.
\end{lemma}

\begin{proof}

Let us prove that instead of $\pi_k$, we can consider the coordinate projection $\tl \pi_k$ (with the same $p$) such that the conditions of the Furukawa-Ito Theorem \ref{dual_defect_classification_th} are preserved:

$$\tl \pi_k:
\begin{cases}
 e_0^{(1)}, \dots , e_0^{(j_0)}, e_1^{(2)}, \dots , e_1^{(j_1)} , \dots , e_k^{(2)}, \dots , e_k^{(j_k)} &  \to 0\\
 e_1^{(1)} &  \to q_1\\
 \ \  \ \ \ \ \  \dots  & \dots \\
 e_k^{(1)} &  \to q_k
\end{cases} $$

Note that for any point $x \in \As$, at most one of the coordinates $e_\alpha^{(\beta)}$ with $\alpha \ge 1$ can be nonzero, and any such nonzero coordinate must be equal to $1$. Define the sets $\tl \As_i \subset \Z^n/\Z^k$ analogously to $\As_i$, but using the projection $\tl \pi_k$. Denote the sets
$$\As_\alpha^{(\beta)} = \{ x \in \As: e_\alpha^{(\beta)} - \text{coordinate of $x$ is $1$} \} / \Z^k \subset \Z^n/\Z^k \quad\text{for } \alpha \ge 1$$

Note that:
\begin{align*}
\As_\alpha &= \bigcup\limits_{\beta = 1}^{j_\alpha} \As_{\alpha}^{(\beta)}, \quad \alpha = 1,\dots k\\
\tl \As_\alpha &= \As_{\alpha}^{(1)}, \quad \alpha = 1,\dots k\\
\tl \As_0 &= \As_0 \cup \bigcup\limits_{\alpha = 1}^{k}  \bigcup\limits_{\beta = 2}^{j_\alpha} 
 \As_{\alpha}^{(\beta)}
\end{align*}

Denote by $E_0$ the lattice spanned by $\{e_0^{(i)}/\Z^k, i = 1,\dots, j_0\}$. Note that $\langle \As_0-\As_0\rangle = E_0$ since $\conv (\As)$ is convenient, and that $\langle\As_\alpha^{(\beta)} - \As_\alpha^{(\beta)} \rangle \subset E_0$ for any $1\le \alpha\le k$ and $1\le \beta\le j_\alpha$. Recall that $p(\As_0) \ast \dots \ast p(\As_k)$ is of join type. Therefore, $p (\langle\As_\alpha^{(\beta)} - \As_\alpha^{(\beta)} \rangle) = 0$ for any $1\le \alpha\le k$ and $1\le \beta\le j_\alpha$. Thus $p(\tl \As_0) \ast \dots \ast p(\tl \As_k)$ is also of join type since $p(\tl \As_\alpha) = 0$ for all $\alpha = 1,\dots, k$.
\end{proof}

In what follows, we prove the $B_k$-theorem for arbitrary $a_i^{(j)} \in \Z_{> 0}$ by induction on the dimension $n$. 

\begin{proposition}
    Consider the cone $C^{(1)}$ spanned by $\{e_i^{(j)}| a_i^{(j)} = 1\}$. If the polytope $\conv (\As) \cap C^{(1)}$ is negligible in the cone $C^{(1)}$, then $\conv (\As)$ is a $B_k$-polytope.
\end{proposition}

\begin{proof}
    If all the $a_i^{(j)}$ are equal to $1$, then this case is already proved in Lemma \ref{unit_bk_proof_lem}. Otherwise, by the induction hypothesis, the polytope $\conv (\As) \cap C^{(1)}$ is a $B_k$-polytope. Note that (here we essentially use that the corresponding $a_i^{(j)}$ are equal to $1$) the coordinate projection of $\As \setminus C^{(1)}$ to $C^{(1)}$ is only the origin $O$. Then the polytope $\conv(\As)$ is also a $B_k$-polytope (differing from  $\conv(\As) \cap C^{(1)}$ only in the first summand). 
\end{proof}

Thus the following proposition implies the $B_k$-theorem.

\begin{proposition}\label{final_bk_prop}
The polytope $\conv (\As) \cap C^{(1)}$ is negligible in the cone $C^{(1)}$.
\end{proposition}

Consider all pairs $i,j$ such that $a_i^{(j)} \ge 2$. Consider points $D_i^{(j)}$ such that $\pi_k^O(a_i^{(j)} e_i^{(j)}) = \pi_k^O(D_i^{(j)})$ when $i \ne j$; such points exist since $\conv(\As)$ is convenient. Consider the set $\mathbf D = \bigcup_{i,j} D_i^{(j)}\cup  (\As \cap C^{(1)})$ . Note that $\conv (\mathbf D)$ is convenient in the cone $C$ and 
$$\nu_C(\conv(\mathbf D)) = \prod\limits_{(i,j): a_i^{(j)} \ge 2} (a_i^{(j)} - 1) \nu_{C^{(1)}}(\conv(\As) \cap C^{(1)})$$

Recall that the Newton number is monotonic (see Remark \ref{mon_nu_rem}).  Thus, since $\conv (\As)$ is negligible, the convex hull $\conv(\mathbf D)$ is also negligible in $C$. Therefore, $\conv (\As) \cap C^{(1)}$ is negligible in $C^{(1)}$.

\end{document}